\documentclass[final]{amsart}
\usepackage[scale=0.75,centering]{geometry}
\usepackage[cmyk]{xcolor}
\usepackage{amsmath,amssymb,amscd}
\usepackage[mathscr]{eucal}
\usepackage{mathtools}
\usepackage{amsthm}
\numberwithin{equation}{section}
\usepackage[colorlinks = true,linkcolor = blue,urlcolor  = blue,citecolor = blue,anchorcolor = blue,linktocpage=true
]{hyperref}

\usepackage{todonotes}
\makeatletter
\providecommand\@dotsep{5}
\renewcommand{\listoftodos}[1][\@todonotes@todolistname]{%
	\@starttoc{tdo}{#1}}
\makeatother

\usepackage[color]{showkeys}
\usepackage[inline]{enumitem}

\definecolor{terminology}{cmyk}{0,0.86,1,0}

\usepackage{tikz}
\usetikzlibrary{positioning}
\usetikzlibrary{cd}
\usetikzlibrary{shapes.geometric}
\usetikzlibrary{decorations.pathreplacing,decorations.markings}
\usetikzlibrary{backgrounds}
\usetikzlibrary{arrows.meta}

\allowdisplaybreaks[2]

\newtheorem{theorem}{Theorem}[section]
\newtheorem{proposition}[theorem]{Proposition}
\newtheorem{conjecture}[theorem]{Conjecture}
\newtheorem{corollary}[theorem]{Corollary}
\newtheorem{lemma}[theorem]{Lemma}

\theoremstyle{definition}
\newtheorem{remark}[theorem]{Remark}
\newtheorem{example}[theorem]{Example}
\newtheorem{definition}[theorem]{Definition}

\theoremstyle{remark}
\newtheorem*{acknowledgment}{Acknowledgment}

\DeclareMathOperator{\id}{id}
\DeclareMathOperator{\diag}{diag}
\DeclareMathOperator{\image}{Im}

\DeclareMathOperator{\lcm}{lcm}

\DeclareMathOperator{\mat}{Mat}

\newcommand*{\CC}{\mathbb{C}}
\newcommand*{\QQ}{\mathbb{Q}}
\newcommand*{\RR}{\mathbb{R}}

\newcommand*{\ZZ}{\mathbb{Z}}
\newcommand*{\PP}{\mathbb{P}}
\newcommand*{\maxzero}[1]{[#1]_{+}}
\newcommand*{\level}{\ell}

\begin{document}
\title{Difference equations arising from cluster algebras}
\author[Y.~Mizuno]{Yuma Mizuno}
\address{Department of Mathematical and Computing Science,
	Tokyo Institute of Technology,
	2-12-1 Ookayama, Meguro-ku, Tokyo 152-8550, Japan.}
\email{mizuno.y.aj@m.titech.ac.jp}
\begin{abstract}
	We characterize Y/T-system type difference equations arising from cluster algebras by triples of matrices, which we call T-data, that have a certain symplectic property.
	We show that all mutation loops are essentially obtained from T-data, which generalizes the general solution for period $1$ quivers given by Fordy and Marsh.
	We also show that any T-datum associated with a periodic Y/T-system
	has the simultaneous positivity.
	As an application,
	we propose a version of Nahm's conjecture from a viewpoint of cluster algebras.
	We conjecture that
	given a periodic T/Y-system of a certain type,
	we have a family of hypergeometric $q$-series that are also modular functions.
\end{abstract}
\maketitle
\tableofcontents

\section{Introduction}
Cluster algebras were introduced by Fomin and Zelevinsky in the seminal paper~\cite{FZ1}.
A \emph{cluster algebra} is a commutative ring equipped with a combinatorial structure called a \emph{cluster pattern}.
A cluster pattern is a graph whose vertices are \emph{clusters}, which are tuples of \emph{cluster variables}, and edges are \emph{exchange relations}.
Such combinatorial structures have been found in many areas of mathematics, and thus the theory of cluster algebra has many applications.

One of the main applications of the theory of cluster algebras
is the study of discrete dynamical systems.
In their fourth paper on cluster algebras~\cite{FZ4},
Fomin and Zelevinsky introduced \emph{bipartite belts}, which are discrete dynamical systems associated with
bipartite symmetrizable generalized Cartan matrices.
They proved that the bipartite belt associated with a generalized Cartan matrix $A$ is periodic if and only if $A$ is of finite type,
that is, there exists a vector $v>0$ such that $Av>0$.
Thus, periodic bipartite belts are classified by the Cartan-Killing classification.
This result generalizes and refines the periodicity of Zamolodchikov's Y-systems, which was conjectured by Zamolodchikov~\cite{Zamo} in the study of thermodynamic Bethe ansatz,
and proved by Fomin and Zelevinsky in~\cite{FZ_Ysystem} prior to their fourth paper~\cite{FZ4}.
They also proved that
there is a bijection between
the set of terms appear in a bipartite belt associated with a finite type Cartan matrix $A$ and the set of almost positive roots in the root system associated with $A$.
A key fact in the proof of these results is that
terms in a bipartite belt are realized as cluster variables in some cluster algebra,
and recurrence relations of this bipartite belt are realized as exchange relations in the same cluster algebra.

Bipartite belts are very special cases of discrete dynamical systems called \emph{Y-systems} and \emph{T-systems} in cluster algebras, in the sense of Nakanishi's paper~\cite{Nakb}.
These discrete dynamical systems have nice properties inherited from general properties of cluster algebras such as
the Laurent phenomenon~\cite{FZ1},
the Laurent positivity~\cite{LeeSchiffler,GHKK},
the synchronicity phenomenon~\cite{nakanishi2019synchronicity},
and the quantization~\cite{BerensteinZelevinsky,FG}. 
It has been discovered that many interesting discrete dynamical systems can be realized as Y-systems or T-systems in cluster algebras, for example:
\begin{itemize}
	\item periodic discrete dynamical systems that are generalization of Zamolodchikov's Y-systems~\cite{FZ4,GalashinPylyavskyy,IIKKNa,IIKKNb,Keller,NakanishiStella},
	\item non-periodic but integrable discrete dynamical systems such as
	Q-systems~\cite{DFK2009,DFK2010,Kedem2008},
	pentagram maps~\cite{GSTVpentagram,Glick},
	the $q$-Painlev\'{e} equations~\cite{BGM,HoneInoue,Okubo},
	and discrete dynamical systems associated with mutation-periodic quivers~\cite{FordyHone,FordyMarsh} and bipartite recurrent quivers~\cite{galashin2016quivers,galashin2017quivers}.
\end{itemize}
Because of
these nice properties and interesting examples,
it is natural to ask what discrete dynamical systems arise from cluster algebras in general.
In this paper, we give an answer to this question.

\paragraph*{\textbf{Main result}}
Let $r$ be a positive integer, and we denote by $[1,r]$ the set $\{1,\dots, r\}$.
Given a triple of matrices $(N_0,N_+,N_-)$ in $\mat_{r \times r}(\ZZ[z])$ whose entries are written as
\begin{align*}
	N_{\varepsilon}=\biggl(\sum_{p \in \ZZ_{\geq 0}} n_{ab;p}^{\varepsilon} z^p \biggr)_{a,b \in [1,r]},
\end{align*}
we consider the following relation for each $(a,u) \in [1,r] \times \ZZ$ among indeterminates in $\{T_a(u) \mid (a,u) \in [1,r] \times \ZZ \}$ :
\begin{align}\label{intro:T-system}
\prod_{b=1}^{r} \prod_{p\geq 0}  T_b(u+p)^{n_{ba;p}^{0}}  =  \prod_{b=1}^{r} \prod_{p\geq 0}  T_b(u+p)^{n_{ba;p}^{+}} +  \prod_{b=1}^{r} \prod_{p\geq 0} T_b(u+p)^{n_{ba;p}^{-}}.
\end{align}
We impose the following conditions on $(N_0,N_+,N_-)$: 
\begin{enumerate}[label=(N\arabic*),leftmargin=2cm]
	\item $n_{ab;p}^{0} = \delta_{ab} \delta_{p0} + \delta_{a\sigma(b)} \delta_{pp_a}$ for
	some $\sigma \in \mathfrak{S}_r$ and $p_a \in \ZZ_{>0}$,
	\item $n_{ab;p}^{+} \geq 0$ and $n_{ab;p}^{-} \geq 0$ for any $a,b,p$,
	\item $n_{ab;p}^{+} = 0$ and $n_{ab;p}^{-}=0$ unless $0<p<p_a$,
	\item $n_{ab;p}^{+} n_{ab;p}^{-} = 0$ for any $a,b,p$,
\end{enumerate}
where $\mathfrak{S}_r$ is the symmetric group on $[1,r]$ and $\delta$ is the Kronecker delta.
The condition (N1) says that the left-hand side in \eqref{intro:T-system} is equal to $T_a(u) T_{\sigma(a)}(u+p_{\sigma(a)})$.
The condition (N2) says that the right-hand side in \eqref{intro:T-system} is a sum of two monomials.
The condition (N3) together with (N1) implies that any $T_a(u)$ can be written as a rational function in the initial variables $(T_a(p))_{(a,p) \in R_{\mathrm{in}}}$,
where 
\begin{align*}
	R_{\mathrm{in}} = \{ (a,p) \in [1,r] \times \ZZ \mid 0 \leq p <p_a \}.
\end{align*}
The condition (N4) says that the two monomials in the right-hand side in \eqref{intro:T-system} do not have common divisors.

\begin{definition}\label{intro:T-datum}
	We say that a triple of matrices $\alpha=(A_+,A_-,D)$ is a \emph{T-datum} of size $r$ if
	$A_\pm$ can be written as $A_\pm=N_0 - N_\pm$ by a triple of matrices $(N_0,N_+,N_-)$ in $\mat_{r \times r}(\ZZ[z])$ satisfying (N1)--(N4), and $D$ is a positive integer diagonal matrix satisfying the following conditions:
	\begin{itemize}
		\item $N_0 D= D N_0$,
		\item $D^{-1}N_{\pm}D \in \mat_{r \times r}(\ZZ[z])$,
		\item $A_+ D A_-^{\dagger} = A_- D A_+^{\dagger}$,
	\end{itemize}
	where $A_{\pm}^{\dagger} := (A_{\pm}|_{z=z^{-1}})^{\mathsf{T}}$.
\end{definition}

\begin{definition}
	Let $\alpha$ be a T-datum.
	Let $\mathscr{T}(\alpha)$ be the commutative ring generated by the indeterminates $(T_a(u)^{\pm1})_{(a,u) \in [1,r] \times \ZZ }$ subject to the relations \eqref{intro:T-system}
	and $T_a(u) T_a(u)^{-1} = 1$ for any $(a,u) \in [1,r] \times \ZZ$.
	We define $\mathscr{T}^{\circ}(\alpha)$ to be the subring of $\mathscr{T}(\alpha)$ generated by $(T_a(u))_{(a,u) \in [1,r] \times \ZZ }$.
	We say that $\mathscr{T}^{\circ}(\alpha)$ is the \emph{T-algebra} associated with $\alpha$.
	We also say that the family of relations \eqref{intro:T-system} is the \emph{T-system} associated with $\alpha$.
\end{definition}

Let $I$ be a finite index set.
For a pair $(B,x)$ of an $I \times I$ skew-symmetrizable integer matrix $B$ and an $I$-tuple $x=(x_i)_{i \in I}$ of algebraically independent commuting variables,
the cluster algebra associated with the initial seed $(B,x)$ is defined~\cite{FZ1,FZ4}, which is denoted by $\mathcal{A}(B,x)$.
In Section \ref{section:consequences}, we prove the following:

\begin{theorem}\label{intro:main theorem}
	Let $\alpha$ be a T-datum of size $r$.
	Let $x=(x_{a,p})_{(a,p) \in R_{\mathrm{in}}}$ be an $R_{\mathrm{in}}$-tuple of algebraically independent commuting variables.
	Then there exists a unique $R_{\mathrm{in}} \times R_{\mathrm{in}}$ skew-symmetrizable integer matrix $B$ such that
	\begin{enumerate}
		\item there exists a unique injective ring homomorphism $\iota: \mathscr{T}^{\circ}(\alpha) \hookrightarrow \mathcal{A}(B,x)$
		such that $\iota(T_a(p)) = x_{a,p}$ for any $(a,p) \in R_{\mathrm{in}}$,
		\item $\iota(T_a(u))$ is a cluster variable in $\mathcal{A}(B,x)$ for any $(a,u) \in [1,r] \times \ZZ$,
		\item the image of the relation \eqref{intro:T-system} by $\iota$ is an exchange relation in $\mathcal{A}(B,x)$ for any $(a,u) \in [1,r] \times \ZZ$.
	\end{enumerate}
\end{theorem}

Conversely, we also prove that T-systems in cluster algebras (in the sense in ~\cite{Nakb}) yield T-data (Section \ref{section:T-system and Y-system} and \ref{section:T-data from mutation loops}).
Therefore, our definition of T-data completely characterize
when a system of difference equations of the form \eqref{intro:T-system} is realized as a family of exchange relations in a cluster algebra.
In the following,
we give remarks and applications of Theorem \ref{intro:main theorem}.

\paragraph*{\textbf{Sequences of mutations that preserve exchange matrices}}
The matrix $B$ in Theorem \ref{intro:main theorem} is called the initial exchange matrix in the cluster algebra $\mathcal{A}(B,x)$.
In the proof of Theorem \ref{intro:main theorem},
we give the explicit formula \eqref{eq:def of B from cap} expressing $B$ using a matrix coefficients in a T-datum.
We also construct a sequence of mutations, which are fundamental operations in the theory of cluster algebras, that preserves the exchange matrix $B$ up to relabeling of indices.
Such a sequence of mutations is called a \emph{mutation loop}.

Mutation loops themselves are of interest from a geometric viewpoint:
they are representatives of elements in cluster modular groups~\cite{FG}, which are cluster algebraic counterparts of mapping class groups of surfaces.
We show that essentially all mutation loops are obtained by the formula \eqref{eq:def of B from cap} (Theorem \ref{theorem:Ml Td bijection}).
The formula \eqref{eq:def of B from cap} gives a effective way to find mutation loops since finding T-data is usually easier than finding mutation loops.
We give many examples of T-data in Section \ref{section:examples},
which recover or generalize mutation loops
in the literature.
In Section \ref{section:period 1},
we classify T-data of size $1$ (Theorem \ref{theorem:T-datum of size 1}), which turns out to recover the classification of period $1$ quivers by Fordy and Marsh~\cite{FordyMarsh}.
In Section \ref{section:commuting Cartan},
we define T-data associated with pairs of commuting Cartan matrices.
They are generalization of bipartite belts by Fomin and Zelevinsky~\cite{FZ4}.
In particular, our definition also works for non-bipartite cases such as the ``tadpole type''.
In Section \ref{section:affinization},
we define T-data associated with level restricted T-systems for quantum affinizations~\cite{KNSaff}.
These T-systems are restricted version of T-systems for quantum affinizations discovered by Hernandez~\cite{Hernandez},
where ``T-systems for quantum affinizations'' mean algebraic relations among $q$-characters of Kirillov-Reshetikhin modules over quantum affinizations.
Although mutation loops corresponding to these T-data are already constructed in~\cite{IIKKNa,IIKKNb,KNSaff,NakanishiTamely},
our method gives a simple systematic way to produce these mutation loops.

\paragraph*{\textbf{T-systems with coefficients and Y-systems}}
Theorem \ref{intro:main theorem} can be extended to T-systems \emph{with coefficients}.
In fact, we show Theorem \ref{intro:main theorem} in this generality (Theorem \ref{theorem:embedding into cluster algebra}).
Coefficients of T-systems are governed by Y-systems, which are generalization of Zamolodchikov's Y-systems~\cite{Zamo}.
In terms of T-data,
the coefficients of the T-system associated with a T-datum $\alpha$ is described by the \emph{Langlands dual T-datum}
$\alpha^{\vee} = (A_+^{\vee},A_-^{\vee},D^{\vee})$.
If we write the entries of the matrices in $\alpha^{\vee}$ as
\begin{align*}
	N_{\varepsilon}^{\vee}=\biggl(\sum_{p \in \ZZ_{\geq 0}} \check{n}_{ab;p}^{\varepsilon} z^p \biggr)_{a,b \in [1,r]},
\end{align*}
the coefficients of the T-system associated with $\alpha$ is governed by the following system of relations:
\begin{align*}
	\prod_{b=1}^{r} \prod_{p\geq 0} Y_b (u-p)^{ \check{n}_{ab;p}^{0}}
	&=
	\frac{ \prod_{b=1}^{r} \prod_{p\geq 0}  \bigl(1 \oplus Y_b (u-p)\bigr)^{ \check{n}_{ab;p}^{-}}}{ \prod_{b=1}^{r} \prod_{p\geq 0}  \bigl(1 \oplus Y_b (u-p)^{-1}\bigr)^{ \check{n}_{ab;p}^{+}}},
\end{align*}
where $\oplus$ is the ``auxiliary addition'' in the underlying semifield to which the coefficients belong.
We call this family of relations the Y-system associated with $\alpha$.

\paragraph*{\textbf{Periodic T-systems and Y-systems}}
We say that a T-datum is of \emph{finite type} if
the set $\{ T_a(u) \in \mathscr{T}^{\circ}(\alpha) \mid (a,u) \in [1,r] \times \ZZ \}$ is a finite set.
This is equivalent to saying that the T-system associated with $\alpha$ is periodic.
By the synchronicity phenomenon of cluster algebras~\cite{nakanishi2019synchronicity},
this is also equivalent to the periodicity of the \emph{Y-system} associated with $\alpha$ in universal semifields.

Many examples of finite type T-data have been found in the literature, which are associated with the following data:
\begin{itemize}
	\item finite type Cartan matrices~\cite{Zamo,FZ_Ysystem,FZ4},
	\item tensor products of pairs of finite type Cartan matrices~\cite{RVT,Keller},
	\item untwisted quantum affine algebras~\cite{KunibaNakanishi92,IIKKNa,IIKKNb},
	\item the sine-Gordon Y-systems and the reduced sine-Gordon Y-systems~\cite{Tateo,NakanishiStella}, which are associated with continued fractions,
	\item admissible $ADE$ bigraphs~\cite{GalashinPylyavskyy}.
\end{itemize}
In many cases in this list, the periodicities of Y-systems in universal semifields were conjectured in the 1990s in physics~\cite{Zamo,RVT,KunibaNakanishi92,Tateo}, and proved in the 21st century by using the theory of cluster algebras~\cite{FZ_Ysystem,FZ4,Keller,IIKKNa,IIKKNb,NakanishiStella,GalashinPylyavskyy}.

Since there are many interesting examples of finite type T-data as in this list, the classification of finite type T-data is a interesting problem.
Except for special cases~\cite{FZ4,GalashinPylyavskyy}, however, the classification of finite type T-data is still not well understood.
In this paper, we prove that any finite type T-datum satisfies the following simultaneous positivity:

\begin{theorem}[Theorem \ref{theorem:simultaneous positivity}]\label{intro:simultaneous positivity}
	Let $\alpha=(A_+,A_-,D)$ be a T-datum.
	If $\alpha$ is of finite type,
	then there exists a vector $v>0$ such that $\mathring{A}_+^{\mathsf{T}} v >0$ and $\mathring{A}_-^{\mathsf{T}} v >0$,
	where $\mathring{A}_{\pm} =A_{\pm}|_{z=1}$.
\end{theorem}

Theorem \ref{intro:simultaneous positivity} gives a effective method to determine that a given T-datum is not of finite type (see Example \ref{example:simultaneous positivity}).
This theorem is also used in the next topic: relationship between cluster algebras and Nahm's conjecture. 

\paragraph*{\textbf{Nahm's conjecture}}
In~\cite{Nahm}, Nahm gave a connection between rational conformal field theories and torsion elements in Bloch groups.
In particular, Nahm's conjecture states that the modularity of certain hypergeometric $q$-series is related to torsion elements in Bloch groups~(see \cite[Chapter II, Section 3]{Zagier}).
We give a version of Nahm's conjecture from a viewpoint of cluster algebras.

Let $\alpha = (A_+,A_-,D)$ be a Cartan-like T-datum (see Definition \ref{def:Cartan-like}) of finite type.
By Theorem \ref{intro:simultaneous positivity},
we can show that the system of equations
\begin{align}\label{intro:nahm eq}
f_a = \prod_{b=1}^{r} (1-f_b)^{\check{\kappa}_{ab}} \quad\quad (a \in [1,r])
\end{align}
has a unique real solution such that $0<f_a<1$ for any $a \in [1,r]$,
where we define $K^{\vee} = (\check{\kappa}_{ab}) \in \mat_{r \times r}(\QQ)$ by $K^{\vee} = (\mathring{A}_{+}^{\vee})^{-1} \mathring{A}_{-}^{\vee}$.
For this solution, the value
\begin{align}\label{intro:def of c}
	c_{\alpha} := \frac{6}{\pi^2} \sum_{a=1}^{r} d_a L(f_a)
\end{align}
turns out to be a rational number, where $L(x)$ is the Rogers dilogarithm function.
This fact follows from dilogarithm identities in cluster algebras that are proved by Nakanishi~\cite{Nakb}.
Moreover, for any solution $(f_1,\dots, f_r) \in \overline{\QQ}^r$ of \eqref{intro:nahm eq},
we can define a torsion element in the Bloch group $\mathcal{B}(F)$, where $F$ is a number field containing the solution.

Motivated by Nahm's conjecture,
we introduce a family of hypergeometric $q$-series
$(\mathcal{Z}_{\alpha,\sigma}(q))_{\sigma \in S_{\alpha}}$
for any Cartan-like T-datum $\alpha$ of finite type,
where $S_{\alpha}$ is a finite abelian group associated with $\alpha$.
We call these $q$-series the \emph{partition $q$-series} of $\alpha$.
In fact, these are generalization of partition $q$-series of mutation loops introduced by Kato and Terashima~\cite{KatoTerashima}.
We conjecture that partition $q$-series are modular functions:

\begin{conjecture}[Conjecture \ref{conj:q-series}]
	\label{intro:modular conjecture}
	Let $\alpha$ be a Cartan-like T-datum of finite type.
	Then
	$q^{-c_{\alpha}/24} Z_{\alpha,\sigma}(q)$ is a modular function for any $\sigma \in S_{\alpha}$,
	where $c_{\alpha}$ is the rational number defined by \eqref{intro:def of c}.
\end{conjecture}

We prove this conjecture for $r=1$ using Rogers-Ramanujan type identities (Theorem \ref{theorem:modularity r=1}).
We also give the following examples (Example \ref{example:modular Zagier's lists}-\ref{example:modular nil-daha}) supporting the conjecture for $r \geq 2$: Zagier's lists of $2 \times 2$ and $3 \times 3$ matrices concerning the Nahm's conjecture~\cite{Zagier},
a $q$-series in the Andrew-Gordon identity~\cite{Andrews74},
fermionic formulas for quantum affine algebras~\cite{HKOTT},
and $q$-series appear in the theory of nilpotent double affine Hecke algebras~\cite{CherednikFeigin}.

\begin{acknowledgment}
	The author thanks Yuji Terashima and Shunsuke Kano for fruitful discussions.
	This work is supported by JSPS KAKENHI Grant Number JP18J22576.
\end{acknowledgment}

\section{Y/T-systems in cluster algebras}
\subsection{Preliminaries on cluster algebras}
We review cluster algebras following~\cite{FZ4}.
Let $I$ be a finite index set.
An $I \times I$ integer matrix $B=(B_{ij})_{i,j \in I}$ is called \emph{skew-symmetrizable} if there exist a tuple of positive integers $d=(d_i)_{i \in I}$ such that $B_{ij} d_j = - B_{ji} d_i$.
Such a tuple is called a (right) \emph{symmetrizer} of $B$.
For any $I \times I$ matrix $B=(B_{ij})_{i,j \in I}$ and bijection $\nu: I \to I'$ between finite index sets,
we define an $I' \times I'$ matrix $\nu(B) = (B'_{i'j'})_{i',j' \in I'}$ by $B'_{\nu(i)\nu(j)} = B_{ij}$.

\begin{definition}
	Let $B=(B_{ij})_{i,j \in I}$ be a skew-symmetrizable integer matrix, and let $k \in I$.
	The \emph{matrix mutation} $\mu_k:B \mapsto B'$ is a transformation that transforms $B$ into a new skew-symmetrizable integer matrix $B'=(B'_{ij})_{i,j \in I}$ defined as follows:
	\begin{align}
	\label{eq: B mutation}
	B'_{ij} &=
	\begin{cases}
	-B_{ij} & \text{if $i=k$ or $j=k$},\\
	B_{ij} + \maxzero{B_{ik}}\maxzero{B_{kj}} - \maxzero{-B_{ik}}\maxzero{-B_{kj}} & \text{otherwise},
	\end{cases}
	\end{align}
	where $\maxzero{x} := \max(0,x)$.
\end{definition}

If $d$ is a symmetrizer of $B$, then it is also a symmetrizer of $\mu_k(B)$.
In particular,
if $B$ is a skew-symmetric integer matrix, then $\mu_k(B)$ is also a skew-symmetric integer matrix.
In this case, it is convenient to describe matrix mutations
in terms of quivers.
A \emph{quiver} is a finite oriented graph without loops and $2$-cycles.
For any skew-symmetric integer matrix $B$, we define a quiver $Q(B)$ as follows:
the vertex set is $I$, and there are $\maxzero{B_{ij}}$ arrows from $i$ to $j$.
Conversely, we can recover a skew-symmetric integer matrix $B(Q)$ from a quiver $Q$ by $B(Q)_{ij} = Q_{ij} - Q_{ji}$,
where $Q_{ij}$ is the number of arrows from $i$ to $j$.
This gives a bijection between the set of $I \times I$
skew-symmetric integer matrices and the set of quivers whose vertex set is $I$.

\begin{definition}\label{def:quiver mutation}
	Let $Q$ be a quiver, and let $k$ be a vertex of $Q$.
	The \emph{quiver mutation} $\mu_k$ is a transformation that transforms $Q$ into a quiver $\mu_k(Q)$ defined by the following three steps:
	\begin{enumerate}
		\item For each length two path $i \to k \to j$, add a new arrow $i \to j$.
		\item Reverse all arrows incident to the vertex $k$.
		\item Remove all 2-cycles.
	\end{enumerate}
\end{definition}

Matrix mutations and quiver mutations are compatible.
The transformation
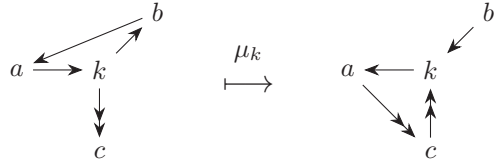
\[
\begin{tikzpicture}
[scale=1.1]
\node(k) at (0,0) []{$k$};
\node(a) at (-1,0) []{$a$};
\node(b) at (0.7,0.7) []{$b$};
\node(c) at (0,-1) []{$c$};
\draw[arrows={-Stealth[scale=1.2] Stealth[scale=1.2]}] (k)--(c);
\foreach \from/\to in {k/b,a/k,b/a}
\draw[arrows={-Stealth[scale=1.2]}] (\from)--(\to);
\node(mutation) at (1.8,-0.2) [scale=1.2] {$\longmapsto$};
\node(mutation) at (1.8,0.2) [] {$\mu_k$};
\node(k') at (4,0) []{$k$};
\node(a') at (3,0) []{$a$};
\node(b') at (4.7,0.7) []{$b$};
\node(c') at (4,-1) []{$c$};
\draw[arrows={-Stealth[scale=1.2] Stealth[scale=1.2]}] (c')--(k');
\draw[arrows={-Stealth[scale=1.2] Stealth[scale=1.2]}] (a')--(c');
\draw[arrows={-Stealth[scale=1.2]}] (b')--(k');
\draw[arrows={-Stealth[scale=1.2]}] (k')--(a');
\end{tikzpicture}
\]
is an example of a quiver mutation.

A set $\PP$ is called a \emph{semifield} if it is an abelian multiplicative group
endowed with an binary operation $\oplus$ which is commutative, associative, and distributive with respect to the multiplication.
We denote by $\ZZ\PP$ the group ring of $\PP$ over $\ZZ$.
This ring is an integral domain since the abelian multiplicative group of $\PP$ is torsion-free.
Throughout this paper, a $\ZZ\PP$-algebra means
a commutative associative $\ZZ\PP$-algebra with an identity element,
and we assume that a $\ZZ\PP$-algebra homomorphism sends the identity element to the identity element.
We denote by $\QQ\PP$ the field of fractions of $\ZZ \PP$.
We fix a field $\mathcal{F}$ that is isomorphic to the field of rational functions over $\QQ \PP$
in $\lvert I \rvert$ variables.

\begin{example}
	Let $J$ be a finite index set.
	\begin{enumerate}
		\item Let $\mathrm{Trop}(u_j)_{j \in J}$ be the abelian multiplicative group generated by the indeterminates $(u_j)_{j \in J}$.
		We define a binary operation $\oplus$ on $\mathrm{Trop}(u_j)_{j \in J}$ by
		\begin{align*}
		\prod_{j \in J} u_j^{a_j} \oplus \prod_{j \in J} u_j^{b_j}
		=\prod_{j \in J} u_j^{\min(a_j,b_j)}.
		\end{align*}
		This binary operation makes $\mathrm{Trop}(u_j)_{j \in J}$ a semifield,
		which is called a \emph{tropical semifield}.
		If $J$ is the empty set, $\mathrm{Trop}(u_j)_{j \in J} = \{ 1 \}$ is called
		the \emph{trivial semifield}.
		\item Let $\QQ_{\mathrm{sf}}(u_j)_{j \in J}$ be the subset of $\QQ(u_j)_{j \in J}$ consisting of all rational functions that can be written as subtraction-free expressions in $(u_i)_{i \in J}$.
		The set $\QQ_{\mathrm{sf}}(u_j)_{j \in J}$ is a semifield with respect to the usual multiplication and addition,
		which is called a \emph{universal semifield}.
	\end{enumerate}
\end{example}

\begin{definition}
	An \emph{($I$-labeled) Y-seed} in $\PP$ is a pair $(B,y)$, where
	\begin{itemize}
		\item $B=(B_{ij})_{i,j \in I}$ is an $I \times I$ skew-symmetrizable integer matrix, 
		\item $y=(y_i)_{i \in I}$ is an $I$-tuple in elements of $\PP$.
	\end{itemize}
\end{definition}

\begin{definition}
	An \emph{($I$-labeled) seed} in $\mathcal{F}$ is a pair $(B,y,x)$, where
	\begin{itemize}
		\item $(B,y)$ is an $I$-labeled Y-seed in $\PP$,
		\item $x=(x_i)_{i \in I}$ is an $I$-tuple of elements in $\mathcal{F}$
		forming a free generating set, that is, $(x_i)_{i \in I}$ is algebraically independent over $\QQ \PP$, and $\mathcal{F} = \QQ \PP (x_i)_{i \in I}$.
	\end{itemize}
\end{definition}

For a seed $(B,y,x)$, we refer to
$B$ as the \emph{exchange matrix},
$y$ as the \emph{coefficient tuple},
$x$ as the \emph{cluster},
$y_i$'s as the \emph{coefficients},
and $x_i$'s as the \emph{cluster variables}.

\begin{definition}
	Let $(B,y,x)$ be an $I$-labeled seed in $\mathcal{F}$,
	and let $k \in I$.
	The \emph{seed mutation} $\mu_k : (B,y,x) \mapsto (B',y',x')$ is a transformation that transforms $(B,x,y)$
	into a new seed $(B',y',x')$
	defined as follows:
	\begin{itemize}
		\item $B'= (B'_{ij})_{i,j \in I}$ is given by \eqref{eq: B mutation},
		\item $y' = (y'_{i})_{i \in I}$ is given by
		\begin{align}\label{eq: y-mutation}
		y_{i}' =
		\begin{cases}
		y_{k}^{-1} & \text{if $i=k$},\\
		y_i (1 \oplus y_k)^{-B_{ki}} & \text{if $i\neq k$ and $B_{ki}\leq 0$},\\
		y_i (1 \oplus y_k^{-1})^{-B_{ki}} & \text{if $i\neq k$ and $B_{ki}\geq 0$},
		\end{cases}
		\end{align}
		\item $x' = (x'_{i})_{i \in I}$ is given by $x'_{i} = x_{i}$ if $i \neq k$, and
		\begin{align}\label{eq: x-mutation}
			x'_{k} = x_{k}^{-1} \biggl(  
			\frac{y_k}{1 \oplus y_k} \prod_{j \in I} x_{j}^{\maxzero{B_{jk}}}
			+
			\frac{1}{1 \oplus y_k} \prod_{j \in I} x_{j}^{\maxzero{-B_{jk}}}
			\biggr).
		\end{align}
	\end{itemize}
	We also say that the transformation $\mu_k: (B,y) \mapsto (B',y')$ is the \emph{Y-seed mutation}.
\end{definition}

The relation \eqref{eq: x-mutation} is called the \emph{exchange relation}.
Seed mutations are involutions, that is, $\mu_k(\mu_k(B,y,x)) = (B,y,x)$.

The \emph{$I$-regular} tree $\mathbb{T}_I$ is the tree such that all vertices have degree $\lvert I \rvert$ and the edges that are incident to each vertex are labeled by the elements in $I$.
A \emph{cluster pattern} is an assignment of an $I$-labeled seed to every vertex in $\mathbb{T}_I$, such that the two seeds
assigned to the endpoints of any edge labeled by $k \in I$
are obtained from each other by the seed mutation $\mu_k$.

\begin{definition}
	The \emph{cluster algebra} $\mathcal{A}$ associated with a given cluster pattern is the $\ZZ\PP$-subalgebra of $\mathcal{F}$ generated by all cluster variables in the pattern.
	We denote $\mathcal{A}=\mathcal{A}(B,y,x)$, where $(B,y,x)$ is any seed in the underlying cluster pattern.
	We often denote $\mathcal{A}(B,y,x)$ by $\mathcal{A}(B,x)$ when $\mathbb{P}$ is the trivial semifield.
\end{definition}

\subsection{T-systems and Y-systems in cluster algebras}
\label{section:T-system and Y-system}
In this section, we review T-systems and Y-systems in cluster algebras following \cite{Nakb}.
Simply put, T-systems and Y-systems are algebraic relations that $x_k$'s and $y_k$'s, respectively, at mutation indices satisfy.

Let $B = (B_{ij})_{i,j \in I}$ be a skew-symmetrizable integer matrix.
Let $r$ be a positive integer.
For any sequence of indices $\mathbf{i} = (i_1 ,\dots, i_r) \in I^r$,
we denote the composition of mutations $\mu_{i_r} \circ \cdots  \circ \mu_{i_1}$ by $\mu_{\mathbf{i}}$.
If $B_{i_ai_b}=0$ for any $a,b \in [1,r]$,
it is easy to see that
$\mu_{\mathbf{i}} (B) = \mu_{\rho(\mathbf{i})} (B)$
for any permutation $\rho \in \mathfrak{S}_r$,
where $\mathfrak{S}_r$ is the group of bijections on $[1,r]$ and $\rho(\mathbf{i}) := (i_{\rho^{-1}(1)} ,\dots , i_{\rho^{-1}(r)})$.
We say that a transformation $B \mapsto \mu_{\mathbf{i}}(B)$ is a \emph{simultaneous mutation} if
$B_{i_ai_b}=0$ for any  $a,b \in [1,r]$, and $a \neq b$ implies $i_a \neq i_b$ for any $a,b \in [1,r]$.

Let $\mathbf{i} = (i_1 ,\dots, i_r) \in I^r$ be a sequence of indices.
Consider a partition of $\mathbf{i}$:
\begin{equation}\label{eq:parition into simultaneous mutations}
	\begin{split}
		\mathbf{i} &= \mathbf{i}(0) \mid \mathbf{i}(1) \mid \dots \mid \mathbf{i}\mathbf(t-1), \\
		\mathbf{i}(u) &= (i(u)_1, \dots, i(u)_{r_u}), \quad
		\sum_{u=0}^{t-1} r_u = r,
	\end{split}
\end{equation}
where we allow $\mathbf{i}(u)$ to be the empty sequence.
Formally, a partition of $\mathbf{i}$ is an order-preserving map $[1,r] \to \{ 0, \dots, t-1 \}$ where $t$ is a positive integer.
A partition \eqref{eq:parition into simultaneous mutations} is called a \emph{partition into simultaneous mutations} 
if all $B(u) \mapsto B(u+1)$ in the following mutation sequence are simultaneous mutations:
\begin{align}\label{eq: mutation sequence}
B =: B(0) \xmapsto{\mu_{\mathbf{i}(0)}}
B(1)  \xmapsto{\mu_{\mathbf{i}(1)}}
\cdots
\xmapsto{\mu_{\mathbf{i}(t-1)}}
B(t).
\end{align}

\begin{definition}\label{def:mutation loop}
	We say that a quadruple $\gamma=(B,d,\mathbf{i},\nu)$
	is a \emph{mutation loop} if
	\begin{itemize}
		\item $B = (B_{ij})_{i,j \in I}$ is a skew-symmetrizable integer matrix,
		\item $d=(d_i)_{i \in I}$ is a right symmetrizer of $B$,
		\item $\mathbf{i}= (i_1 , \dots, i_r)$ is a sequence of elements in $I$
		equipped with a partition into simultaneous mutations $\mathbf{i}= \mathbf{i}(0) \mid \mathbf{i}(1) \mid \dots \mid \mathbf{i}\mathbf(t-1)$,
		\item $\nu:I \to I$ is a bijection such that $\mu_{\mathbf{i}}(B)= \nu(B)$ and $d=\nu(d)$.
	\end{itemize}
\end{definition}

The integer $r$ is called the \emph{length} of $\gamma$.
Many examples of mutation loops are given in \cite[Section 3]{Nakb}.

The partition \eqref{eq:parition into simultaneous mutations} decomposes $[1,r]$ into $t$ parts.
We define the subgroup $\mathfrak{S}_{r_0,\dots,r_{t-1}} \subseteq \mathfrak{S}_r$, which is isomorphic to $\mathfrak{S}_{r_0} \times \dots \times \mathfrak{S}_{r_{t-1}}$, consisting of permutations that fix the each part as a set.

\begin{definition}\label{def:ml equivalence}
We say that two mutation loops $\gamma=(B,d,\mathbf{i},\nu)$ and $\gamma'=(B',d',\mathbf{i}',\nu')$, where $\mathbf{i}=\mathbf{i}(0) \mid \dots \mid \mathbf{i}\mathbf(t-1)$ and $\mathbf{i}'=\mathbf{i}'(0) \mid \dots \mid \mathbf{i}'\mathbf(t'-1)$, are \emph{equivalent} if there exists a bijection $f:I \to I'$ between the index sets of $B$ and $B'$,
and a permutation $\rho \in \mathfrak{S}_{r_0,\dots,r_{t-1}}$
such that
\begin{itemize}
	\item $B'=f(B)$,
	\item $d'=f(d)$,
	\item $t=t'$ and $\mathbf{i}'(u) = f(\rho(\mathbf{i}(u)))$ for each $u=0,\dots,t-1$,
	\item $\nu'=f \circ \nu \circ f^{-1}$.
\end{itemize}
\end{definition}

For any mutation loop $\gamma=(B,d,\mathbf{i},\nu)$, we have the following infinite length mutation sequence that extends \eqref{eq: mutation sequence}:
\begin{align}\label{eq: infinite mutation sequence}
\begin{array}{llllllll}
&&&
&\cdots
&\xmapsto{\mu_{\mathbf{i}(-2)}}
&B(-1) &\xmapsto{\mu_{\mathbf{i}(-1)}}\\
B(0) &\xmapsto{\mu_{\mathbf{i}(0)}}
&B(1)  &\xmapsto{\mu_{\mathbf{i}(1)}}
&\cdots
&\xmapsto{\mu_{\mathbf{i}(t-2)}}
&B(t-1) &\xmapsto{\mu_{\mathbf{i}(t-1)}} \\
B(t) &\xmapsto{\mu_{\mathbf{i}(t)}}
&B(t+1)  &\xmapsto{\mu_{\mathbf{i}(t+1)}}
&\cdots
&\xmapsto{\mu_{\mathbf{i}(2t-2)}}
&B(2t-1) &\xmapsto{\mu_{\mathbf{i}(2t-1)}}\\
B(2t) &\xmapsto{\mu_{\mathbf{i}(2t)}}
&\cdots &&&&
\end{array}
\end{align}
where $B(nt+k)=\nu^n(B(k))$ and $\mathbf{i}(nt+k)=\nu^n(\mathbf{i}(t+k))$ for any $n \in \ZZ$ and $0\leq k \leq t-1$.
Let $P_{\gamma}$ be the set defined by
\begin{align*}
P_{\gamma} = \{  (i,u) \in I \times \ZZ \mid i \in \mathbf{i}(u) \}.
\end{align*}
Elements in $P_{\gamma}$ are called \emph{mutation points} of $\gamma$.
We also define an integer
$\lambda(i,u)$ for any $(i,u) \in I \times \ZZ$ by
\begin{align*}
\lambda(i,u) = 
\min \{ v \in \ZZ_{\geq 0} \mid (i,u+v) \in P_{\gamma} \}
\end{align*}
if there exists $v \in \ZZ_{\geq 0}$ such that $(i,u+v) \in P_{\gamma}$.
Otherwise, we set $\lambda(i,u) = \infty$. 
The number $\lambda(i,u)$ is called the \emph{latency} of $(i,u)$.
For any $(i,u) \in I \times \ZZ$ with $\lambda(i,u)< \infty$, we define an element
$s(i,u) \in P_{\gamma}$ by
\begin{align*}
	s(i,u)=
	\begin{cases}
		(i,u+\lambda(i,u)) & \text{if $(i,u) \notin P_{\gamma}$},\\
		(i,u+1+\lambda(i,u+1)) & \text{if $(i,u) \in P_{\gamma}$}.
	\end{cases}
\end{align*}
The element $s(i,u)$ is called the \emph{next mutation point} of $(i,u)$.

A mutation loop is called \emph{complete}
if all latencies are finite, that is, $\lambda(i,u)< \infty$ for any $(i,u) \in I \times \ZZ$, or equivalently, for any $(i,0) \in I \times \{ 0 \}$.
In the rest of this paper, we usually assume that mutation loops are complete.

In order to describe the T-system and the Y-system so that the relationship between them is apparent, we need another parameterization of the mutation points.
For any elements $(i,u),(j,v) \in P_{\gamma}$, we write $(i,u) \sim (j,v)$ if there exists $g \in \ZZ$ such that
$j=\nu^g (i)$ and $v=u+gt$.
Let $\pi:P_{\gamma} \to [1,r]$ be the surjective map defined by
$\pi(i,u) =a$ where $a$ is the unique element in $[1,r]$ such that
$(i,u)\sim(i_a,v)$ and $0 \leq v \leq t-1$.
We define a set $R_{\gamma}$ by
\begin{align}\label{eq:def of R gamma}
	R_{\gamma} = \{ (\pi(i,u) ,u ) \mid (i,u) \in P_{\gamma} \}.
\end{align}
\begin{lemma}
	The map $P_{\gamma} \to R_{\gamma}$ defined by $(i,u) \mapsto (\pi(i,u),u)$ is a bijection.
\end{lemma}
\begin{proof}
	The surjectivity is apparent since $\pi$ is surjective.
	We assume that $(i,u),(j,u) \in P_{\gamma}$ satisfy $\pi(i,u) = \pi(j,u)$.
	Then we obtain $(i,u) \sim (j,u)$,
	and this implies that $i=j$ by the definition of the equivalence relation.
\end{proof}

Let $\sigma \in \mathfrak{S}_r$ be the bijection defined by
\begin{align}\label{eq:def of sigma}
	\sigma(a) = \pi(s(i,u)),
\end{align}
where $(i,u) \in \pi^{-1}(a)$ is any mutation point that maps to $a$ by $\pi$.
The definition of $\sigma$ does not depend on the choice of $(i,u)$.
For any $a \in [1,r]$, we denote by $\lambda_{a}$ the positive integer $1+\lambda(i,u+1)$ where $(i,u) \in \pi^{-1}(a)$.
In other words, $\lambda_{a}$ is the positive integer satisfying $s(i,u)=(i,u+\lambda_a)$.
The definition of $\lambda_{a}$ also does not depend on the choice of $(i,u)$.

Let us describe a Y-system associated with a mutation loop $\gamma$.
For any Y-seed $(B,y)$, we have the following infinite length sequence
of Y-seeds:
\begin{align}\label{eq:infinite Y-seeds}
\begin{array}{lllll}
&
&\cdots
&(B(-1),y(-1)) &\xmapsto{\mu_{\mathbf{i}(-1)}}\\
(B(0),y(0)) &\xmapsto{\mu_{\mathbf{i}(0)}}
&\cdots
&(B(t-1),y(t-1)) &\xmapsto{\mu_{\mathbf{i}(t-1)}} \\
(B(t),y(t)) &\xmapsto{\mu_{\mathbf{i}(t)}}
&\cdots &
\end{array}
\end{align}
where $(B(0),y(0)) = (B,y)$
and we define negative ones using the involution property of mutations.
We define an element $Y_a(u) \in \PP$ for any $(a,u) \in R_{\gamma}$ by
\begin{align}\label{eq: def of Y_a(u)}
	Y_a(u) = y_i (u),
\end{align}
where $i \in I$ is a unique index such that $(i,u) \in P_{\gamma}$ and $a=\pi(i,u)$.

Let $N_{0}^{\vee} = (\sum_{p \in \ZZ} \check{n}_{ab;p}^{0} z^p )_{a,b \in [1,r]} \in \mat_{r \times r}(\ZZ[z])$ be the $r \times r$
matrix whose entries are integer coefficients polynomials in the variable $z$ defined by 
\begin{align}\label{eq: NY0}
	\sum_{p \in \ZZ} \check{n}_{ab;p}^{0} z^p = \delta_{ab} + \delta_{a'b} z^{\lambda_{a'}},
\end{align}
where $a'=\sigma^{-1}(a)$.
We also define two matrices $N_+^{\vee}=(\sum_{p \in \ZZ} \check{n}_{ab;p}^{+} z^p )_{a,b \in [1,r]}$ and $N_-^{\vee} = (\sum_{p \in \ZZ} \check{n}_{ab;p}^{-} z^p )_{a,b \in [1,r]}$ in $\mat_{r \times r}(\ZZ[z])$ by
\begin{align}
	\label{eq: NY+-}
	\sum_{p \in \ZZ} \check{n}_{ab;p}^{\pm} z^p  &= 
	\sum_{\substack{(j,v) \in P_{\gamma} \\ s(k,v) = (k,u), \pi(j,v)=b }}
	\maxzero{\pm B_{jk}(v)} z^{\lambda(k,v)},
\end{align}
where $(k,u) \in \pi^{-1}(a)$.
The definition of $N_{\pm}^{\vee}$ does not depend on the choices of $(k,u)$.

\begin{figure}[t]
	\centering
	\begin{tikzpicture}[label distance=-0.75mm,every node/.style={scale=1.0}]
	\node[inner sep=2pt,label=above:$k$] (o) at (0,0,0) {};
	\node[] (x) at (1,0,0) {};
	\node[] (y) at (0,1,1) {};
	\node[] (z) at (0,1,0) {};
	\fill[black] (o) circle (1.5pt);
	\node[inner sep=2pt] (p) at (-7.5,0,0) {};
	\fill[black] (p) circle (1.5pt);
	\draw[white,line width=3pt] (p) -- (o);
	\draw[black,thick,->] (p) -- (o);
	\draw[white,line width=3pt] (o) -- (1.3,0,0);
	\draw[black,thick] (o) -- (0.8,0,0);
	\draw[black,thick,densely dotted] (0.8,0,0) -- (1.3,0,0);
	\node[inner sep=2pt,label=above:$j_1$] (j1) at (-2,0.9,0.3) {};
	\fill[black] (j1) circle (1.5pt);
	\draw[black,->] (j1) -- (-2,0,0);
	\node[inner sep=2pt,label=below:$j_2$] (j2) at (-4,-1.1,-0.3) {};
	\fill[black] (j2) circle (1.5pt);
	\draw[black,->] (-4,0,0) -- (j2);
	\node[inner sep=2pt,label=below:$j_3$] (j3) at (-6,-0.4,0.3) {};
	\fill[black] (j3) circle (1.5pt);
	\draw[black,->] (j3) -- (-6,0,0);
	\foreach \j in {j1,j2,j3}{
		\draw[white,line width=3pt] (\j) --++ (1.3,0,0);
		\draw[black,thick] (\j) --++ (0.8,0,0);
		\draw[black,thick,densely dotted] (\j) + (0.8,0,0) --++ (1.3,0,0);
	}
	\foreach \a in {-2,-4,-6}{
		\draw[white,line width=2.5pt] (\a,-3,-1) -- ++(0,4,0);
		\draw[gray, draw opacity=0.7] (\a,-3,-1) --++ (0,4,0);
	}
	\begin{pgfonlayer}{background}
	\foreach \a in {-2,-4,-6}{
		\draw[gray,draw opacity=0.4] (\a,-1,1) --++ (0,4,0);
	}
	\end{pgfonlayer}
	\foreach \a in {-2,-4,-6}{
		\draw[gray,draw opacity=0.6] (\a,1,-1) --++ (0,2,2);
	}
	\foreach \a in {-2,-4,-6}{
		\draw[gray,draw opacity=0.5] (\a,-3,-1) --++ (0,2,2);
	}
	\draw[fill=gray,draw opacity=0,fill opacity=0.15] (-2,-3,-1) -- ++(0,4,0) -- ++(0,2,2) -- ++(0,-4,0) -- cycle;
	\draw[fill=gray,draw opacity=0,fill opacity=0.15] (-4,-3,-1) -- ++(0,4,0) -- ++(0,2,2) -- ++(0,-4,0) -- cycle;
	\draw[fill=gray,draw opacity=0,fill opacity=0.15] (-6,-3,-1) -- ++(0,4,0) -- ++(0,2,2) -- ++(0,-4,0) -- cycle;
	\draw[black,thick,->] (-8.5,-4,0) -- (1.5,-4,0);
	\foreach \a in {-2,-4,-6}{
		\draw[black,dashed] (\a,-2,0) -- (\a,-4,0);
	}
	\draw[black,dashed] (o) --++ (0,-4,0);
	\draw[black,dashed] (p) --++ (0,-4,0);
	\node[label=below:$u$] (u) at (0,-4,0) {};
	\node[label=below:$v_1$] (v1) at (-2,-4,0) {};
	\node[label=below:$v_2$] at (-4,-4,0) {};
	\node[label=below:$v_3$] at (-6,-4,0) {};
	\node[label=below:$u-\lambda_{a'}$] at (-7.5,-4,0) {};
	\draw[] (v1.center) to[bend left=35] node[inner sep=1.5pt,fill=white]{$\lambda(k,v_1)$} (u.center);
	\end{tikzpicture}
	\caption{A schematic description of a Y-system. A black point represents a mutation point. An arrow in a plane from (resp. to) a mutation point $(j,v)$ to (resp. from) the right arrow that ends at $(k, u)$ indicates that $\maxzero{B_{jk}(v)} \neq 0$ (resp. $\maxzero{-B_{jk}(v)} \neq 0$).}
	\label{fig:Y-system}
\end{figure}
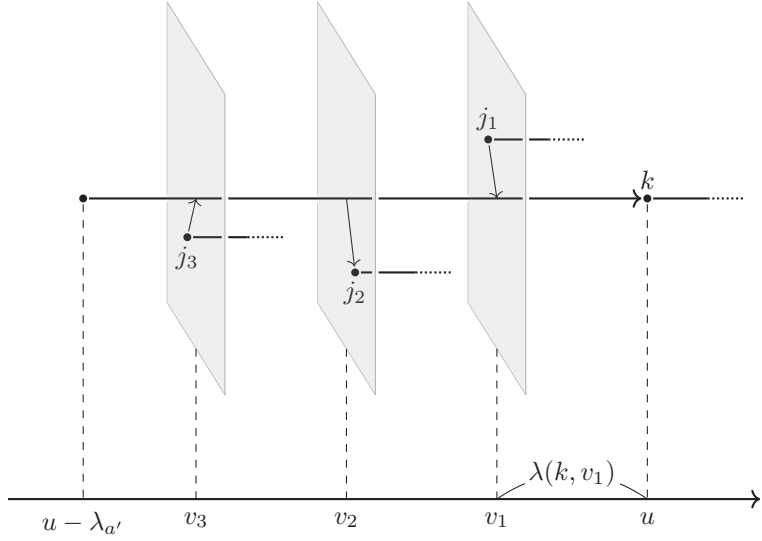

\begin{proposition}[{\cite[Section 5.5]{Nakb}}]
	\label{prop: Y-system}
	For any mutation loop $\gamma$,
	the family of elements $(Y_a (u))_{(a,u) \in R_{\gamma}}$ satisfy the following relation in $\PP$ for any $(a,u) \in R_{\gamma}$:
	\begin{align*}
	\prod_{b,p}   Y_b (u-p)^{ \check{n}_{ab;p}^{0}}
	&=
	\frac{ \prod_{b,p}  \bigl(1 \oplus Y_b (u-p)\bigr)^{ \check{n}_{ab;p}^{-}}}{ \prod_{b,p}  \bigl(1 \oplus Y_b (u-p)^{-1}\bigr)^{ \check{n}_{ab;p}^{+}}},
	\end{align*}
	where $\prod_{b,p}=\prod_{b=1}^r \prod_{p=0}^{\infty}$.
\end{proposition}

We call the family of relations in Proposition \ref{prop: Y-system} the \emph{Y-system} associated with $\gamma$,
and the triple of matrices $(N_{\gamma,0}^{\vee},N_{\gamma,+}^{\vee},N_{\gamma,-}^{\vee})$ the \emph{Y-system triple} of $\gamma$.
From \eqref{eq: NY0}, the left-hand side in the Y-system can be rewritten as 
\begin{align*}
	\prod_{b,p}   Y_b (u-p)^{ \check{n}_{ab;p}^{0}}
	= Y_a(u) Y_{a'} (u-\lambda_{a'}).
\end{align*}
If we define elements $P_a^{\pm}(u) \in \PP$ by
\begin{align}\label{eq:def of P}
P_a^+(u) = \frac{Y_a(u)}{1 \oplus Y_a(u)},\quad
P_a^-(u) = \frac{1}{1 \oplus Y_a(u)},
\end{align}
the relation in Proposition \ref{prop: Y-system} can be written in a simpler form as
\begin{align*}
	\prod_{b,p} P_b^+(u-p)^{\check{n}_{ab;p}^{0}-\check{n}_{ab;p}^{+}}
	=\prod_{b,p} P_b^-(u-p)^{\check{n}_{ab;p}^{0}-\check{n}_{ab;p}^{-}}.
\end{align*}
Figure \ref{fig:Y-system} is a schematic description of a Y-system.

\begin{figure}[t]
	\centering
	\begin{tikzpicture}[label distance=-0.75mm,every node/.style={scale=1.0}]
	\node[inner sep=2pt,label=left:$k$] (o) at (0,0,0) {};
	\node[] (x) at (1,0,0) {};
	\node[] (y) at (0,1,1) {};
	\node[] (z) at (0,1,0) {};
	\fill[black] (o) circle (1.5pt);
	\node[inner sep=2pt] (p) at (5,0,0) {};
	\fill[black] (p) circle (1.5pt);
	\draw[white,line width=3pt] (o) -- (p);
	\draw[black,thick,->] (o) -- (p);
	\node[inner sep=2pt,label=above:$j_1$] (j1) at (0,0.9,0.3) {};
	\draw[white,line width=2pt] (o) -- (j1);
	\node[inner sep=2pt,label=below:$j_2$] (j2) at (0,-1.1,-0.3) {};
	\draw[white,line width=2pt] (o) -- (j2);
	\node[inner sep=2pt,label=below:$j_3$] (j3) at (0,-0.4,0.3) {};
	\draw[white,line width=2pt] (o) -- (j3);
	\draw[black,->] (j1.center) -- (o);
	\draw[black,->] (o) --(j2.center) ;
	\draw[black,->] (o) -- (j3.center);
	\foreach \a in {0}{
		\draw[white,line width=2.5pt] (\a,-3,-1) -- ++(0,4,0);
		\draw[gray, draw opacity=0.7] (\a,-3,-1) --++ (0,4,0);
	}
	\begin{pgfonlayer}{background}
	\foreach \a in {0}{
		\draw[gray,draw opacity=0.4] (\a,-1,1) --++ (0,4,0);
	}
	\end{pgfonlayer}
	\foreach \a in {0}{
		\draw[gray,draw opacity=0.6] (\a,1,-1) --++ (0,2,2);
	}
	\foreach \a in {0}{
		\draw[gray,draw opacity=0.5] (\a,-3,-1) --++ (0,2,2);
	}
	\node[inner sep=2pt] (j1p) at (2,0.9,0.3) {};
	\node[inner sep=2pt] (j2p) at (4,-1.1,-0.3) {};
	\node[inner sep=2pt] (j3p) at (6,-0.4,0.3) {};
	\fill[black] (j1p) circle (1.5pt);
	\fill[black] (j2p) circle (1.5pt);
	\fill[black] (j3p) circle (1.5pt);
	\begin{pgfonlayer}{background}
	\draw[white,line width=3pt] (j1)++(-0.8,0,0) -- (j1p);
	\draw[black,thick,->] (j1)++(-0.8,0,0) -- (j1p);
	\draw[black,thick,densely dotted] (j1)++(-0.8,0,0) --++ (-0.5,0,0);
	\draw[white,line width=3pt] (j2)++(-0.8,0,0) -- (j2p);
	\draw[black,thick,->] (j2)++(-0.8,0,0) -- (j2p);
	\draw[black,thick,densely dotted] (j2)++(-0.8,0,0) --++ (-0.5,0,0);
	\draw[white,line width=3pt] (j3)++(-0.8,0,0) -- (j3p);
	\draw[black,thick,->] (j3)++(-0.8,0,0) -- (j3p);
	\draw[black,thick,densely dotted] (j3)++(-0.8,0,0) --++ (-0.5,0,0);
	\end{pgfonlayer}
	\draw[fill=gray,draw opacity=0,fill opacity=0.15] (0,-3,-1) -- ++(0,4,0) -- ++(0,2,2) -- ++(0,-4,0) -- cycle;
	\draw[black,thick,->] (-1.5,-4,0) -- (7,-4,0);
	\draw[black,dashed] (j1p) -- (2,-4+0.1,0.3);
	\draw[black,dashed] (j2p) -- (4,-4-0.1,-0.3);
	\draw[black,dashed] (j3p) -- (6,-4+0.1,0.3);
	\draw[black,dashed] (0,-2,0) -- (0,-4,0);
	\draw[black,dashed] (p) --++ (0,-4,0);
	\node[label=below:$u$] (u) at (0,-4,0) {};
	\node[label=below:$v_1$] (v1) at (2,-4+0.1,0.3) {};
	\node[label=below:$v_2$] at (4,-4-0.1,-0.3) {};
	\node[label=below:$v_3$] at (6,-4+0.1,0.3) {};
	\node[label=below:$u+\lambda_{a}$] at (5,-4,0) {};
	\draw[] (v1.center) to[bend right=35] node[inner sep=1.5pt,fill=white]{$\lambda(j_1,v_1)$} (u.center);
	\end{tikzpicture}
	\caption{A schematic description of a T-system. A black point represents a mutation point. An arrow in the plane from (resp. to) the mutation point $(k,u)$ to (resp. from) a right arrow that ends at $(j, v)$ indicates that $\maxzero{-B_{jk}(u)} \neq 0$ (resp. $\maxzero{B_{jk}(u)} \neq 0$).}
	\label{fig:T-system}
\end{figure}

Next we are going to describe T-systems.
Let $\gamma$ be a complete mutation loop.
For any seed $(B,y,x)$,
we have the following infinite length sequence
of seeds:
\begin{align}\label{eq:infinite seeds}
\begin{array}{lllll}
&
&\cdots
&(B(-1),y(-1),x(-1)) &\xmapsto{\mu_{\mathbf{i}(-1)}}\\
(B(0),y(0),x(0)) &\xmapsto{\mu_{\mathbf{i}(0)}}
&\cdots
&(B(t-1),y(t-1),x(t-1)) &\xmapsto{\mu_{\mathbf{i}(t-1)}} \\
(B(t),y(t),x(t)) &\xmapsto{\mu_{\mathbf{i}(t)}}
&\cdots &
\end{array}
\end{align}
where $(B(0),y(0),x(0)) = (B,y,x)$.
We define $Y_a(u) \in \PP$ and $T_a(u) \in \mathcal{F}$ for any $(a,u) \in R_{\gamma}$ by
\begin{align}
	\label{eq: def of T_a(u)}
	Y_a(u) = y_i (u), \quad T_a(u) = x_i(u),
\end{align}
where $i \in I$ is a unique index such that $(i,u) \in P_{\gamma}$ and $a=\pi(i,u)$.

We define a matrix $N_{\gamma,0} = (\sum_{p \in \ZZ} n_{ab;p}^{0} z^p )_{1\leq a,b \leq r} \in \mat_{r \times r}(\ZZ[z])$ by
\begin{align}\label{eq: NT0}
	\sum_{p \in \ZZ} n_{ba;p}^{0} z^p = \delta_{ab} + \delta_{\sigma(a)b} z^{\lambda_a}.
\end{align}
We also define two matrices $N_{\gamma,+}=(\sum_{p \in \ZZ} n_{ab;p}^{+} z^p)_{a,b \in [1,r]}$ and $N_{\gamma,-} = (\sum_{p \in \ZZ} n_{ab;p}^{-} z^p)_{a,b \in [1,r]}$ in $\mat_{r \times r}(\ZZ[z])$ by
\begin{align}
	\label{eq: NT+-}
	\sum_{p \in \ZZ} n_{ba;p}^{\pm} z^p &= \sum_{\substack{j \in I \\ \pi(s(j,u))=b }} \maxzero{ \mp B_{jk}(u) }
	z^{\lambda(j,u)},
\end{align}
where $(k,u) \in \pi^{-1}(a)$.
The definition of $N_\pm$ does not depend on the choices of $(k,u)$.

\begin{proposition}[{\cite[Section 5.5]{Nakb}}]
	\label{prop: T-system}
	For any complete mutation loop,
	the family of elements $(Y_a (u))_{(a,u) \in R_{\gamma}}$ and $(T_a(u))_{(a,u) \in R_{\gamma}}$ satisfy the following relation in $\mathcal{F}$ for any $(a,u) \in R_{\gamma}$:
	\begin{align*}
	\prod_{b,p} T_b(u+p)^{n_{ba;p}^{0}} = P_a^+(u) \prod_{b,p} T_b(u+p)^{n_{ba;p}^{-}} + P_a^-(u) \prod_{b,p} T_b(u+p)^{n_{ba;p}^{+}},
	\end{align*}
	where $\prod_{b,p}=\prod_{b=1}^r \prod_{p=0}^{\infty}$ and $P_a^{\pm}(u) \in \PP$ are defined by \eqref{eq:def of P}.
\end{proposition}

We call the family of relations in Proposition \ref{prop: T-system} the \emph{T-system} associated with $\gamma$,
and the triple of matrices $(N_{\gamma,0},N_{\gamma,+},N_{\gamma,-})$ the \emph{T-system triple} of $\gamma$.
From \eqref{eq: NT0}, the left-hand side in the T-system can be rewritten as 
\begin{align*}
\prod_{b,p}   T_b (u+p)^{ n_{ba;p}^{0}}
= T_a(u) T_{\sigma(a)} (u+\lambda_{a}).
\end{align*}
Figure \ref{fig:T-system} is a schematic description of a T-system.

\subsection{Relation between Y-systems and T-systems}
\label{section:relation betwenn Y and T}
Let $\gamma = (B,d,\mathbf{i},\nu)$ be a mutation loop.

\begin{lemma}\label{lemma: symmetizers for mutation loops}
	The family of positive integers $(d_i(u))_{(i,u) \in I \times \ZZ}$ defined by $d_i(u)=d_i$
	satisfies the following:
	\begin{enumerate}
		\item $B_{ij}(u) d_j(u) = - B_{ji}(u) d_i(u)$ for any $i,j\in I$ and $u \in \ZZ$,
		\item $d_i(u) = d_j(v)$ for any $(i,u),(j,v) \in P_{\gamma}$ such that $\pi(i,u)=\pi(j,v)$.
	\end{enumerate}
\end{lemma}
\begin{proof}
	(1) holds since mutations preserve a symmetrizer.
	(2) follows from $d=\nu(d)$.
\end{proof}

From Lemma \ref{lemma: symmetizers for mutation loops}, the positive integers $d'_1 ,\dots, d'_r$
defined by $d'_a=d_i(u)$ where $(i,u) \in \pi^{-1}(a)$ do not depend on the choices of $(i,u)$, and $d'_a = d_i(u)$ for any $(i,u) \in I \times \ZZ$ such that $\pi(s(i,u))=a$.
We denote by $D_{\gamma}$ the positive integer diagonal matrix $\diag(d'_1, \dots, d'_r)$.

\begin{proposition}[cf. Proposition 5.13, \cite{Nakb}]
	\label{proposition:Y/T duality}
	Let $(N_{\gamma,0}^{\vee},N_{\gamma,+}^{\vee},N_{\gamma,-}^{\vee})$
	be the Y-system triple
	and $(N_{\gamma,0},N_{\gamma,+},N_{\gamma,-})$
	be the T-system triple
	of a complete mutation loop $\gamma$.
	Then we have
	\begin{align*}
		N_{\gamma,0}^{\vee} = N_{\gamma,0}
	\end{align*}
	and
	\begin{align*}
		D_{\gamma}N_{\gamma,\varepsilon}^{\vee} = N_{\gamma,\varepsilon} D_{\gamma}
	\end{align*}
	for any $\varepsilon \in \{0,{+},{-}\}$.
\end{proposition}
\begin{proof}
	The first identity follows from
	\begin{align*}
		\sum_{p} \check{n}^{0}_{ab;p} z^p = \delta_{ab} + \delta_{a'b} z^{\lambda_{a'}}
		=\delta_{ba} + \delta_{\hat{b}a} z^{\lambda_b}
		=\sum_{p} n^{0}_{ab;p} z^p.
	\end{align*}
	The second identity for $\varepsilon=0$ follows from Lemma \ref{lemma: symmetizers for mutation loops}.
	We now the second identity for $\varepsilon=\pm$.
	Let $(k,u) \in \pi^{-1}(b)$ and $(k',u') \in \pi^{-1}(a)$.
	Then we have
	\begin{align*}
		\sum_{p} \check{n}^{\pm}_{ab;p} z^p &=
		\sum_{\substack{(j,v) \in P_{\gamma} \\ s(k',v)=(k',u'),\pi(j,v)=b}}
		\maxzero{\pm B_{jk'}(v)} z^{\lambda(k,v)} \\ 
		&=
		\sum_{\substack{j' \in I \\ \pi(s(j',u))=a}}
		\maxzero{\pm B_{kj'}(u)} z^{\lambda(j',u)}.
	\end{align*}
	On the other hand, Lemma \ref{lemma: symmetizers for mutation loops} implies that
	\begin{align*}
		\sum_{p} n^{\pm}_{ab;p} z^p &=
		\sum_{\substack{j \in I \\ \pi(s(j,u))=a}}
		\maxzero{\mp B_{jk}(u)} z^{\lambda(j,u)}\\
		&=
		\sum_{\substack{j \in I \\ \pi(s(j,u))=a}}
		\maxzero{\pm d_j(u) d_k(u)^{-1} B_{kj}(u)} z^{\lambda(j,u)}\\
		&=
		d'_a (d'_b)^{-1} \sum_{\substack{j \in I \\ \pi(s(j,u))=a}}
		\maxzero{\pm  B_{kj}(u)} z^{\lambda(j,u)}.
	\end{align*}
\end{proof}

For any matrix $A$,
we denote the transpose of $A$ by $A^{\mathsf{T}}$.
For any $A \in \mat_{r \times r}(\ZZ[z^{\pm1}])$,
we define a matrix $A^{\dagger} \in \mat_{r \times r}(\ZZ[z^{\pm1}])$
by
$A^{\dagger} = ({A}|_{z=z^{-1}})^{\mathsf{T}}$.
Clearly, we have $(A^{\dagger})^{\dagger}=A$
and $(AB)^{\dagger} = B^\dagger A^\dagger$
for any $A,B \in \mat_{r \times r}(\ZZ[z^{\pm1}])$.

Let $(A_{\gamma,+}^{\vee},A_{\gamma,-}^{\vee})$ and
$(A_{\gamma,+},A_{\gamma,-})$ be
the pairs of matrices defined by
$A_{\gamma,\pm}^{\vee} = N_{\gamma,0}^{\vee} - N_{\gamma,\pm}^{\vee}$
and
$A_{\gamma,\pm} = N_{\gamma,0} - N_{\gamma,\pm}$, respectively.
We call them the \emph{Y-system pair} and the \emph{T-system pair} of $\gamma$.
These pairs of matrices describe the following relation between the Y-system and the T-system:

\begin{theorem}\label{theorem: AYp ATm = AYm ATp}
	Let $\gamma$ be a complete mutation loop, and
	$(A_{\gamma,+}^{\vee},A_{\gamma,-}^{\vee})$ and $(A_{\gamma,+},A_{\gamma,-})$ be the Y-system pair and the T-system pair of $\gamma$, respectively.
	Then we have
	\begin{align*}
	A_{\gamma,+}^{\vee} A_{\gamma,-}^{\dagger}
	=A_{\gamma,-}^{\vee} A_{\gamma,+}^{\dagger}.
	\end{align*}
\end{theorem}
\begin{proof}
	The claim is equivalent to the following equality:
	\begin{align}\label{eq:NYNT}
		N_0^{\vee}N_-^{\dagger}-N_0^{\vee}N_+^{\dagger}-N_-^{\vee}N_0^{\dagger}+N_+^{\vee}N_0^{\dagger}
		=N_+^{\vee}N_-^{\dagger} - N_-^{\vee}N_+^{\dagger}.
	\end{align}
	Let $a,b \in [1,r]$ and $p \in \ZZ$.
	Let us choose an element $(i,u) \in \pi^{-1}(a)$.
	Let $a'=\sigma^{-1}(a)$ and $b'=\sigma^{-1}(b)$.
	Let $p_a = \lambda_{a'}$ and $p_b = \lambda_{b'}$.
	Let $v=u-p$, $u' = u-p_a$, and $v'=v-p_b$.
	Then the $(ab;p)$-th entry in the left-hand side in \eqref{eq:NYNT}
	is given by
	\begin{align*}
	&n_{ba;-p}^{-} + n_{ba';p_a-p}^{-} 
	-n_{ba;-p}^{+} - n_{ba';p_a-p}^{+}
	-\check{n}_{ab;p}^{-} - \check{n}_{ab';p+p_b}^{-}
	+\check{n}_{ab;p}^{+} + \check{n}_{ab';p+p_b}^{+}
	\\
	&=
	\left\{
	\begin{array}{lllll}
		+n_{ba;-p}^{-}&\!\!\!-n_{ba;-p}^{+}&\!\!\!
		-\check{n}_{ab';p+p_b}^{-} &\!\!\!+ \check{n}_{ab';p+p_b}^{+}
		& \text{if $u \leq v$ and $u' \leq v'$},\\
		+n_{ba;-p}^{-}&\!\!\!-n_{ba;-p}^{+}&\!\!\!
		+ n_{ba';p_a-p}^{-} &\!\!\!- n_{ba';p_a-p}^{+} 
		& \text{if $u \leq v$ and $v' \leq u'$},\\
		-\check{n}_{ab;p}^{-} &\!\!\!+ \check{n}_{ab;p}^{+}&\!\!\!
		-\check{n}_{ab';p+p_b}^{-}&\!\!\!+ \check{n}_{ab';p+p_b}^{+}
		& \text{if $v \leq u$ and $u' \leq v'$},\\
		-\check{n}_{ab;p}^{-} &\!\!\!+ \check{n}_{ab;p}^{+}&\!\!\!
		+n_{ba';p_a-p}^{-} &\!\!\!- n_{ba';p_a-p}^{+}
		& \text{if $v \leq u$ and $v' \leq u'$},
	\end{array} \right. \\
	&=
	\sum_{\substack{j \in I\\ (j,v) \in P_{\gamma}, \pi(j,v)=b}}
	\begin{cases}
		B_{ji}(u) + B_{ji}(v')
		& \text{if $u \leq v$ and $u' \leq v'$},\\
		B_{ji}(u) + B_{ji}(u')
		& \text{if $u \leq v$ and $v' \leq u'$},\\
		B_{ji}(v) + B_{ji}(v')
		& \text{if $v \leq u$ and $u' \leq v'$},\\
		B_{ji}(v) + B_{ji}(u')
		& \text{if $v \leq u$ and $v' \leq u'$},
	\end{cases}\\
	&=\sum_{\substack{j \in I \\ (j,v) \in P_{\gamma}, \pi(j,v)=b}} (B_{ji}(\min(u,v)) + B_{ji}(\max(u',v') )).
	\end{align*}
	On the other hand, the $(ab;p)$-th entry in the right-hand side in \eqref{eq:NYNT} is given by
	\begin{align*}
	&\sum_{c=1}^r \sum_{w \in \ZZ} 
	(\check{n}_{ac;u-w}^{+} n_{bc;v-w}^{-}
	-\check{n}_{ac;u-w}^{-} n_{bc;v-w}^{+}) \\
	&=\!\! \sum_{\substack{j \in I,(k,w) \in P_{\gamma}\\ (j,v) \in P_{\gamma},\pi(j,v) =b \\ \max(u',v')<w<\min(u,v) }}
	(\maxzero{B_{ki}(w)}\maxzero{B_{jk}(w)}
	-\maxzero{-B_{jk}(w)}\maxzero{-B_{ki}(w)} ).
	\end{align*}
	These two entries coincide since
	\begin{align*}
		&B_{ji}(\min(u,v)) + B_{ji}(\max(u',v') ) \\
		&=
		\sum_{\substack{(k,w) \in P_{\gamma} \\ \max(u',v')<w<\min(u,v) }}
		(\maxzero{B_{ki}(w)}\maxzero{B_{jk}(w)}
		-\maxzero{-B_{jk}(w)}\maxzero{-B_{ki}(w)} )
	\end{align*}
	by the rule of matrix mutations \eqref{eq: B mutation}.
\end{proof}

\section{Axiomatic approach to Y/T systems}
\subsection{T-data}
Let $r$ be a positive integer.
As in the last section,
we define an involution $\dagger: \mat_{r \times r}(\QQ(z)) \to \mat_{r \times r}(\QQ(z))$ by
$A^{\dagger}=(A|_{z=z^{-1}})^{\mathsf{T}}$.

For a triple $(N_0, N_+, N_-)$ of the matrices in $\mat_{r \times r}(\ZZ[z])$ whose entries are given by
\begin{align}\label{eq:entries of N}
N_{\varepsilon}=\biggl(\sum_{p \in \ZZ_{\geq 0}} n_{ab;p}^{\varepsilon} z^p \biggr)_{a,b \in [1,r]},
\end{align}
we consider the following conditions:
\begin{enumerate}[label=(N\arabic*),ref=N\arabic*,leftmargin=2cm]
	\item \label{item:N1}
	$n_{ab;p}^{0} = \delta_{ab} \delta_{p0} + \delta_{a\sigma(b)} \delta_{pp_a}$ for
	some $\sigma \in \mathfrak{S}_r$ and $p_a \in \ZZ_{>0}$,
	\item \label{item:N2}
	$n_{ab;p}^{+} \geq 0$ and $n_{ab;p}^{-} \geq 0$ for any $a,b,p$,
	\item \label{item:N3}
	$n_{ab;p}^{+} = 0$ and $n_{ab;p}^{-}=0$ unless $0<p<p_a$,
	\item \label{item:N4}
	$n_{ab;p}^{+} n_{ab;p}^{-} = 0$ for any $a,b,p$.
\end{enumerate}

\begin{definition}\label{def:T-datum}
	We say that a triple of matrices $\alpha=(A_+,A_-,D)$ is a \emph{T-datum} of size $r$ if
	$A_\pm$ can be written as $A_\pm=N_0 - N_\pm$ by a triple of matrices $(N_0,N_+N_-)$ in $\mat_{r \times r}(\ZZ[z])$ satisfying (\ref{item:N1})--(\ref{item:N4}), and $D$ is a positive integer diagonal matrix that satisfies the following conditions:
	\begin{itemize}
		\item $N_0 D= D N_0$,
		\item $D^{-1}N_{\pm}D \in \mat_{r \times r}(\ZZ[z])$,
		\item $A_+ D A_-^{\dagger} = A_- D A_+^{\dagger}$.
	\end{itemize}
\end{definition}

It is clear that the triple $(N_0,N_+,N_-)$ that satisfies the conditions (\ref{item:N1})--(\ref{item:N4}) is uniquely recovered from $(A_+, A-)$ as follows:
$N_0=\maxzero{A_+}=\maxzero{A_-}$, $N_+=\maxzero{-A_+}$, and
$N_-=\maxzero{-A_-}$, where we take $\maxzero{\;}$ for each coefficient.
Note that both matrices $A_+$ and $A_-$ have non-zero determinants since their determinants are monic polynomials with constant terms $1$, which follows from the conditions (\ref{item:N1}) and (\ref{item:N3}).

We say that the last equation
\begin{align}\label{eq: symplectic}
	A_+ D A_-^{\dagger} = A_- D A_+^{\dagger}
\end{align}
in Definition \ref{def:T-datum}
is the \emph{symplectic relation} due to the following lemma, which can be easily verified:

\begin{lemma}
	Let $A_+,A_- \in \mat_{r \times r}(\ZZ[z^{\pm1}])$
	be matrices with non-zero determinants,
	and $D$ be a positive integer diagonal matrix. Then the following conditions are equivalent:
	\begin{enumerate}
		\item $A_+ D A_-^{\dagger} = A_- D A_+^{\dagger}$.
		\item $A_+ D A_-^{\dagger}$ is a $\dagger$-invariant.
		\item $A_- D A_+^{\dagger}$ is a $\dagger$-invariant.
		\item $(A_-)^{-1} A_+ D$ is a $\dagger$-invariant.
		\item $(A_+)^{-1} A_- D$ is a $\dagger$-invariant.
		\item $D^{-1} (A_-)^{-1} A_+$ is a $\dagger$-invariant.
		\item $D^{-1} (A_+)^{-1} A_-$ is a $\dagger$-invariant.
		\item The rows of the $r \times 2r$ matrices $[ A_+ \; A_- ]$ are pairwise orthogonal with respect to the symplectic pairing $\langle \, ,\, \rangle: (\ZZ[z^{\pm1}])^{2r} \times (\ZZ[z^{\pm1}])^{2r} \to \ZZ[z^{\pm1}]$ defined by
		\begin{align*}
		\left\langle 
		\begin{bmatrix}
		f(z)\\
		g(z)
		\end{bmatrix},
		\begin{bmatrix}
		f'(z)\\
		g'(z)
		\end{bmatrix}
		\right\rangle 
		=
		\begin{bmatrix}
		f(z)^{\mathsf{T}} & g(z)^{\mathsf{T}}
		\end{bmatrix}
		\begin{bmatrix}
		O & D \\
		-D & O
		\end{bmatrix}
		\begin{bmatrix}
		f'(z^{-1}) \\
		g'(z^{-1})
		\end{bmatrix},
		\end{align*}
		where $f(z),g(z),f'(z),g'(z) \in (\ZZ[z^{\pm1}])^r$.
	\end{enumerate}
\end{lemma}

Let $\alpha=(A_+,A_-,D)$ be a T-datum.
Then the triple $\alpha^{\vee}=(A_+^{\vee},A_-^{\vee},D^{\vee})$
defined by $N_{\varepsilon}^{\vee} = D^{-1} N_{\varepsilon} D$ and $D^{\vee} =  \delta D^{-1}$ where $\delta$ is the product of the greatest common divisor and the least common multiple of all the entries in $D$, is also a T-datum.
The T-datum $\alpha^{\vee}$ is called the \emph{Langlands dual} of $\alpha$.
Clearly, we have $\alpha^{{\vee}{\vee}}=\alpha$.
We write the entries of $N_{\varepsilon}^{\vee}$ as
\begin{align}\label{eq:entries of N check}
	N_{\varepsilon}^{\vee}=\biggl(\sum_{p \in \ZZ_{\geq 0}} \check{n}_{ab;p}^{\varepsilon} z^p \biggr)_{a,b \in [1,r]}.
\end{align}

\begin{definition}\label{def:consistent subset}
	Let $\alpha=(A_+,A_-,D)$ be a T-datum of size $r$.
	We say that a subset $R \subseteq [1,r] \times \ZZ$
	is \emph{consistent} for $\alpha$
	if the following conditions are satisfied:
	\begin{enumerate}[label=(R\arabic*),ref=R\arabic*]
		\item \label{item:R1}
		If $(a,u) \in R$ and $n_{ab;p}^{0}$, $n_{ab;p}^{+}$, or $n_{ab;p}^{-} \neq 0$,
		then $(b,u-p) \in R$.
		\item \label{item:R2}
		If $(a,u) \in R$ and $n_{ba;p}^{0}$, $n_{ba;p}^{+}$, or $n_{ba;p}^{-} \neq 0$,
		then $(b,u+p) \in R$.
		\item \label{item:R3}
		There exists a positive integer $t$ such that $R=R^{(t)}$ and 
		\begin{align*}
			[1,r] \times \ZZ = \bigsqcup_{k=0}^{t-1} R^{(k)},
		\end{align*}
		where $R^{(k)}:= \{ (a,u+k) \mid (a,u) \in R \}$.
	\end{enumerate}
\end{definition}

For example, $[1,r] \times \ZZ$ itself is always consistent since (\ref{item:R1}) and (\ref{item:R2}) are obvious, and (\ref{item:R3}) is satisfied by setting $t=1$.
Note that the positive integer $t$ in (\ref{item:R3}) is uniquely determined from $R$.
In the conditions (\ref{item:R1}) and (\ref{item:R2}), we can replace $n$ with $\check{n}$.
From (\ref{item:N1}) together with (\ref{item:R1}) and (\ref{item:R2}), we have
\begin{align*}
	(a,u) \in R \quad 
	&\text{if and only if} \quad
	(\sigma(a),u+p_{\sigma(a)}) \in R\\
	&\text{if and only if} \quad
	(\sigma^{-1}(a),u-p_a) \in R.
\end{align*}

\begin{definition}\label{def:cap R equivalence}
	Let $(\alpha,R)$ and $(\alpha',R')$ be 
	pairs of T-data of size $r$ and consistent subsets for them.
	They are called \emph{equivalent} if there exists a permutation $\rho \in \mathfrak{S}_r$ such that $A'_{\pm} = \rho(A_{\pm})$, $D'=\rho(D)$,
	and $R'=\rho(R)$, where $\rho(R) = \{ (\rho(a),u)  \mid (a,u) \in R \}$.
\end{definition}

\begin{definition}
	Let $\alpha$ be a T-datum of size $r$,
	and $R \subseteq [1,r] \times \ZZ$ be a consistent subset for $\alpha$.
	We say that a family of elements $(Y_a(u))_{(a,u) \in R }$ is a \emph{solution of the Y-system} associated with $(\alpha,R)$ in a semifield $\PP$
	if $Y_a(u) \in \PP$ and the following relation holds in $\PP$ for any $(a,u) \in R$:
	\begin{align}\label{eq:standalone Y-system}
	\prod_{b,p} Y_b (u-p)^{ \check{n}_{ab;p}^{0}}
	&=
	\frac{ \prod_{b,p}  \bigl(1 \oplus Y_b (u-p)\bigr)^{ \check{n}_{ab;p}^{-}}}{ \prod_{b,p}  \bigl(1 \oplus Y_b (u-p)^{-1}\bigr)^{ \check{n}_{ab;p}^{+}}},
	\end{align}
	where $\prod_{b,p} = \prod_{b=1}^{r} \prod_{p=0}^{\infty}$.
\end{definition}

For any solution of the Y-system associated with $(\alpha,R)$ in $\PP$, we define elements $P_a^{\pm}(u) \in \PP$ $((a,u) \in R)$ by
\begin{align*}
P_a^+(u) = \frac{Y_a(u)}{1 \oplus Y_a(u)},\quad
P_a^-(u) = \frac{1}{1 \oplus Y_a(u)}.
\end{align*}

\begin{definition}
	Let $\alpha$ be a T-datum of size $r$,
	and $R \subseteq [1,r] \times \ZZ$ be a consistent subset for $\alpha$.
	Let $Y=(Y_a(u))_{(a,u) \in R }$ be a solution of the Y-system associated with $(\alpha,R)$ in a semifield $\PP$.
	Let $\mathscr{T}(\alpha,R,Y)$ be the $\ZZ\PP$-algebra generated by the indeterminates $(T_a(u)^{\pm1})_{(a,u) \in R }$ subject to the relation
	\begin{align}\label{eq:standalone T-system}
	\prod_{b,p} T_b(u+p)^{n_{ba;p}^{0}} = P_a^+(u) \prod_{b,p} T_b(u+p)^{n_{ba;p}^{-}} + P_a^-(u) \prod_{b,p} T_b(u+p)^{n_{ba;p}^{+}}
	\end{align}
	for any $(a,u) \in R$,
	together with $T_a(u) T_a(u)^{-1} = 1$.
	We define $\mathscr{T}^{\circ}(\alpha,R,Y)$ to be the subalgebra of $\mathscr{T}(\alpha,R,Y)$ generated by $(T_a(u))_{(a,u) \in R }$.
	We say that $\mathscr{T}^{\circ}(\alpha,R,Y)$ is the \emph{T-algebra} associated with $(\alpha,R,Y)$.
	We often denote $\mathscr{T}^{\circ}(\alpha,R,Y)$ by $\mathscr{T}^{\circ}(\alpha)$ when $R=[1,r] \times \ZZ$ and $\mathbb{P}$ is the trivial semifield.
\end{definition}

The family of relations \eqref{eq:standalone Y-system} is called the \emph{Y-system} associated with $(\alpha,R)$,
and the family of relations \eqref{eq:standalone T-system} is called the \emph{T-system} associated with $(\alpha,R,Y)$.

\begin{example}[Somos-4 recurrence]
	\label{example:cap somos4}
	The triple of $1\times 1$ matrices $\alpha=(A_+,A_-,D)$ defined by
	\begin{align*}
		A_+=
		\begin{bmatrix}
			1-2z^2+z^4
		\end{bmatrix} , \quad 
		A_-= 
		\begin{bmatrix}
			1-z-z^3+z^4
		\end{bmatrix}, \quad 
		D=
		\begin{bmatrix}
			1
		\end{bmatrix}
	\end{align*}
	is a T-datum,
	and the whole set $R=\{1\} \times \ZZ$ is consistent for $\alpha$.
	The family $Y=(Y(u))_{(1,u) \in R}$ defined by
	$Y(u) = c_1 c_2^{-1}$ for any $u \in \ZZ$
	is a solution of the Y-system associated with $\alpha$
	in $\mathrm{Trop}(c_1,c_2)$, where we denote $Y_1(u)$ by $Y(u)$.
	The family of relations
	\begin{align*}
		T(u) T(u+4) = c_1 T(u+1)T(u+3) + c_2 T(u+2)^2
	\end{align*}
	for $u \in \ZZ$ is the T-system associated with $(\alpha,R,Y)$, where we denote $T_1(u)$ by $T(u)$.
	This is called the \emph{Somos-4 recurrence} \cite{FZ_LP}.
\end{example}

\begin{example}[Bipartite belt]
	\label{example:cap b belt}
	Let $A=(2 \delta_{ab} - n_{ab})_{a,b \in [1,r]}$ be a symmetrizable generalized Cartan matrix,
	and $D$ be a right symmetrizer of $A$.
	Suppose that $A$ is bipartite, that is, there exists a function $\epsilon:[1,r] \to \{1,-1\}$ such that
	$n_{ab} > 0$ implies $\epsilon(a) \neq \epsilon(b)$
	for any $a,b \in [1,r]$.
	Let $N=2I_r - A$.
	Then the triple of $r \times r$ matrices $\alpha=(A_+,A_-,D)$ defined by
	\begin{align*}
		A_+ = (1+z^2) I_r, \quad
		A_- = (1+z^2) I_r - z N
	\end{align*}
	is a T-datum since
	\begin{align*}
		A_+ D A_-^{\dagger} - A_- D A_+^{\dagger}
		= (z+z^{-1}) (-DN^{\mathsf{T}} + ND) =0,
	\end{align*}
	and the set $R \subseteq [1,r] \times \ZZ$ defined by
	\begin{align*}
		R = \{ (a,u) \in [1,r] \times \ZZ \mid \epsilon(a) = (-1)^{u-1} \}
	\end{align*}
	is consistent for $\alpha$.
	The family of relations
	\begin{align*}
		Y_a(u) Y_a(u-2) = \prod_{b=1}^{r} (1 \oplus Y_b(u-1) )^{n_{ba}}
	\end{align*}
	for $(a,u) \in R$ is the Y-system associated with $\alpha$,
	and
	\begin{align*}
		T_a(u) T_a(u+2)
		=\frac{Y_a(u) \prod_{b=1}^{r} T_b(u+1)^{n_{ba}} + 1}{1 \oplus Y_a(u)}
	\end{align*}
	for $(a,u) \in R$ is the T-system associated with $(\alpha,R,Y)$.
	The discrete dynamical system given by these relations are called the \emph{bipartite belt} associated with $A$ \cite[Section 8]{FZ4}.
	
	\begin{table}[t]
		\scalebox{0.9}{\parbox{.5\linewidth}{
				\begin{align*}
				\begin{array}{|c|ccccccc|}
				\hline
				u & 0 & 1 & 2 & 3 & 4 & 5 & 6 \\
				\hline
				&&&&&&&
				\\
				Y_1(u) & y_1 && \displaystyle{\frac{1 \oplus y_2 \oplus y_1y_2}{y_1}} && \displaystyle{\frac{1}{y_2}}  && (1 \oplus y_1)y_2 \\
				Y_2(u) && (1 \oplus y_1) y_2 && \displaystyle{\frac{1 \oplus y_2}{y_1 y_2}} && y_1 &  \\
				&&&&&&&\\
				T_1(u) & x_1 && \displaystyle{\frac{y_1 x_2 + 1}{(1 \oplus y_1)x_1}} && \displaystyle{\frac{x_1+y_2}{(1 \oplus y_2)x_2}}  && x_2 \\
				T_2(u) && x_2 && \displaystyle{\frac{y_1 y_2 x_2 +x_1 + y_2}{(1 \oplus y_2 \oplus y_1y_2) x_1 x_2}} && x_1 &  \\
				&&&&&&&\\
				\hline 
				\end{array}
				\end{align*}
		}}
		\caption{The bipartite belt associated with the Cartan matrix of type $A_2$.}
		\label{table:A2 Y/T-system}
	\end{table}
	If $A$ is the Cartan matrix of type $A_2$ for instance,
	the triple of matrices in $\alpha$ is given by
	\begin{align*}
	A_+=
	\begin{bmatrix}
	1+z^2 & 0 \\
	0 & 1+z^2
	\end{bmatrix},\quad
	A_-=
	\begin{bmatrix}
	1+z^2 & -z \\
	-z & 1+z^2
	\end{bmatrix},\quad
	D=
	\begin{bmatrix}
	1 & 0\\
	0 & 1
	\end{bmatrix}.
	\end{align*}
	Table \ref{table:A2 Y/T-system} is the bipartite belt
	associated with the Cartan matrix of type $A_2$ with $\epsilon(1)=-1$ and $\epsilon(2)=1$,
	where $y_1$ and $y_2$ are arbitrary elements in the underlying semifield $\PP$, and we write $T_1(0)$ and $T_2(1)$ as $x_1$ and $x_2$, respectively.
\end{example}

\subsection{T-data from mutation loops}
\label{section:T-data from mutation loops}
Let us see that we can obtain T-data from mutation loops.
Let $\gamma$ be a complete mutation loop of length $r$.
Let $(N_{\gamma,0}, N_{\gamma,+}, N_{\gamma,-})$ be the T-system triple and $(A_{\gamma,+},A_{\gamma,-})$ be the T-system pair of $\gamma$, which are defined in Section \ref{section:T-system and Y-system}.
Let $D_{\gamma}$ be the positive integer diagonal matrix in Proposition \ref{proposition:Y/T duality}.

\begin{lemma}
	The triple $(N_{\gamma,0}, N_{\gamma,+}, N_{\gamma,-})$
	satisfies the conditions (\ref{item:N1})--(\ref{item:N4}).
\end{lemma}
\begin{proof}
	The condition (\ref{item:N1}) is satisfied if
	$p_a = \lambda_{\sigma^{-1}(a)}$ and $\sigma$ is as in \eqref{eq:def of sigma}.
	The condition (\ref{item:N2}) is obvious from the definition.
	The definition \eqref{eq: NT+-} implies (\ref{item:N3}) since
	$0<\lambda(j,u)<\lambda_{\sigma^{-1}(b)}$ if $\pi(s(j,u))=b$ and $(j,u) \notin P_{\gamma}$.
	The definition \eqref{eq: NT+-} also implies (\ref{item:N4}) since at least one of $\maxzero{b}$ and $\maxzero{-b}$ is zero for any integer $b$.
\end{proof}

\begin{proposition}\label{proposition:cap from mutation loop}
	The triple
	$\alpha_{\gamma}:=(A_{\gamma,+},A_{\gamma,-},D_{\gamma})$ is a T-datum.
\end{proposition}
\begin{proof}
	We have $D_{\gamma} N_{\gamma,0} =D_{\gamma} N_{\gamma,0}^{\vee} = N_{\gamma,0}D_{\gamma}$, and
	$D_{\gamma}^{-1} N_{\gamma,\pm} D_{\gamma} = N_{\gamma,\pm}^{\vee} \in \mat_{r \times r}(\ZZ[z])$ by Proposition \ref{proposition:Y/T duality}.
	We also have
	\begin{align*}
		A_{\gamma,+} D_{\gamma} A_{\gamma,-}^{\dagger} - A_{\gamma,-} D_{\gamma} A_{\gamma,+}^{\dagger}
		&= D_{\gamma} A_{\gamma,+}^{\vee} D_{\gamma}^{-1} D_{\gamma} A_{\gamma,-}^{\dagger}
		-D_{\gamma} A_{\gamma,-}^{\vee} D_{\gamma}^{-1} D_{\gamma} A_{\gamma,+}^{\dagger} \\
		&= D_{\gamma} ( A_{\gamma,+}^{\vee} A_{\gamma,-}^{\dagger}
		-A_{\gamma,-}^{\vee} A_{\gamma,+}^{\dagger}) \\
		&=0
	\end{align*}
	by Proposition \ref{proposition:Y/T duality} and Theorem \ref{theorem: AYp ATm = AYm ATp}.
\end{proof}

Let $D_{\gamma}^{\vee}$ the diagonal matrix defined by $D_{\gamma}^{\vee} =  \delta D_{\gamma}^{-1}$ where $\delta$ is the product of the greatest common divisor and the least common multiple of all the entries in $D_{\gamma}$.

\begin{corollary}
	The triple
	$\alpha_{\gamma}^{\vee}:=(A_{\gamma,+}^{\vee},A_{\gamma,-}^{\vee},D_{\gamma,-}^{\vee})$ is a T-datum, and it is the Langlands dual of $\alpha_{\gamma}$.
\end{corollary}

\begin{proposition}
	The subset $R_\gamma \subseteq [1,r] \times \ZZ$ defined in \eqref{eq:def of R gamma} is consistent
	for $\alpha_{\gamma}$.
\end{proposition}
\begin{proof}
	The conditions (\ref{item:R1}) and (\ref{item:R2}) follow from Proposition \ref{prop: Y-system} and \ref{prop: T-system}, respectively.
	The condition (\ref{item:R3}) is satisfied if we define $t$ in (\ref{item:R3}) as the length of the partition of $\mathbf{i}$.
\end{proof}

\subsection{Mutation loops from T-data}
\label{section:Mutation loops from T-data}
In this section, we prove all T-data can be obtained from mutation loops up to equivalence.

\begin{theorem}\label{theorem:mutation loop from T-datum}
	Suppose that $\alpha=(A_+,A_-,D)$ is a T-datum of size $r$,
	and $R \subseteq [1,r] \times \ZZ$ is consistent for $\alpha$.
	Then there exists a complete mutation loop $\gamma$ of length $r$ such that
	$(\alpha_{\gamma},R_{\gamma})$ and $(\alpha,R)$
	are equivalent,
	where $\alpha_{\gamma}=(A_{\gamma,+},A_{\gamma,-},D_{\gamma})$.
\end{theorem}

The rest of Section \ref{section:Mutation loops from T-data} is devoted to the proof of Theorem \ref{theorem:mutation loop from T-datum}.
Let $p_1 ,\dots, p_r$ be positive integers and $\sigma$ be the permutation of $[1,r]$ in (\ref{item:N1}).
Let $\psi : [1,r] \times \ZZ \to [1,r] \times \ZZ$ be the bijection defined by
$\psi(a,u) = (a,u+1)$.
We define a family of subsets $R^{(k)}$ $(k \in \ZZ)$ by $R^{(k)} = \psi^k(R)$ as in Definition \ref{def:consistent subset}.
We also define a subset $R^{(k)}(u) \subseteq R^{(k)}$ for any $u \in \ZZ$ by
$R^{(k)}(u) = \{ (a,u+p) \in R^{(k)} \mid 0 \leq p <p_a \}$.
We denote $R^{(0)}(u)$ by $R(u)$.
The map $\psi$
restricts to a bijection $\psi|_{R^{(k)}(u)} : R^{(k)}(u) \to R^{(k+1)}(u+1)$.
We will write this restriction simply $\psi$ when no confusion can arise.
Let $t$ be the integer in (\ref{item:R3}) in Definition \ref{def:consistent subset}.
Then we have $R^{(k)}(u)=R^{(k+t)}(u)$.
In particular, the map $\psi^t$ restricts to a bijection
$\psi^t|_{R^{(k)}(u)} : R^{(k)}(u) \to R^{(k)}(u+t)$.
We also define a family of bijections $\varphi_{u} : R(u) \to R(u+1)$ ($u \in \ZZ$) by
\begin{align}\label{eq:def of varphi}
	\varphi_{u}(a,u+p) = 
	\begin{cases}
		(a,u+p) &\text{if $p \neq 0$,}\\
		(\sigma(a),u+p_{\sigma(a)}) &\text{if $p=0$.}
	\end{cases}
\end{align}
It is easy to check that $\psi^t$ and $\varphi$ commute in the sense that $\psi^t \circ \varphi_{u} = \varphi_{u+t} \circ \psi^t$.
We define $R_0(u) \subseteq R(u)$ by
$R_0(u)  = \{ (a,u+p) \in R(u) \mid p=0 \}$, which is endowed with the linear order coming from the standard linear order on $[1,r]$.

For any $u \in \ZZ$, we define an $R(u) \times R(u)$ matrix $\bar{B}(u)$ by
\begin{equation}\label{eq:def of bar B(u)}
\begin{split}
	&\bar{B}_{(a,u+p)(b,u+q)}(u) = 
	-n_{ab;p-q}^{+} +n_{ab;p-q}^{-}
	+\check{n}_{ba;q-p}^{+}-\check{n}_{ba;q-p}^{-} \\
	&\quad\quad +\sum_{c=1}^{r} \sum_{v=0}^{\min(p,q)}
	\bigl( n_{ac;p-v}^{+} \check{n}_{bc;q-v}^{-}
	- n_{ac;p-v}^{-} \check{n}_{bc;q-v}^{+} \bigr),
\end{split}
\end{equation}
where $n^{\pm}$ and $\check{n}^{\pm}$ are defined in \eqref{eq:entries of N} and \eqref{eq:entries of N check}, respectively.
Note that $\bar{B}(u)$ and $\bar{B}(v)$ may be different matrices if $u \neq v$,
even though they have the same expression \eqref{eq:def of bar B(u)}.
Rather, these matrices are related by mutations, as the following lemma shows:

\begin{lemma}\label{lemma:B(u+1)=varphi mu B(u)}
$\bar{B}(u+1) = \varphi_{u} (\mu_{R_0(u)} (\bar{B}(u)))$ for any $u \in \ZZ$.
\end{lemma}
\begin{proof}
	Let
	$\bar{B}'(u) = \mu_{R_0(u)}(\bar{B}(u))$
	and $(a,u+p),(b,u+q) \in R(u)$.
	Then we have
	\begin{equation}\label{eq:bar B'}
		\begin{split}
			\bar{B}'_{(a,u+p)(b,u+q)}(u) 
			= 
			\begin{cases}
				-\bar{B}_{(a,u+p)(b,u+q)}(u)  & \text{if $p$ or $q=0$},\\
				\bar{B}_{(a,u+p)(b,u+q)}(u) - {\displaystyle \sum_{c=1}^{r} 
				\bigl( n_{ac;p}^{+} \check{n}_{bc;q}^{-}
				- n_{ac;p}^{-} \check{n}_{bc;q}^{+} \bigr)}
			& \text{if $p,q>0$},
			\end{cases}
		\end{split}
	\end{equation}
	since
	\begin{align*}
		&\sum_{\substack{ c \in [1,r] \\ (c,u) \in R_0(u)}}
		\bigl(
		\maxzero{\bar{B}_{(a,u+p)(c,u)}(u)}\maxzero{\bar{B}_{(c,u)(b,u+q)}(u)} \\
		&\quad\quad\quad-\maxzero{-\bar{B}_{(a,u+p)(c,u)}(u)}\maxzero{-\bar{B}_{(c,u)(b,u+q)}(u)} \bigr) \\
		&=\sum_{c=1}^{r} 
		\bigl( n_{ac;p}^{-} \check{n}_{bc;q}^{+}
		- n_{ac;p}^{+} \check{n}_{bc;q}^{-} \bigr)
	\end{align*}
	by (\ref{item:N2}), (\ref{item:N4}), and (\ref{item:R1}).
	
	The proof is divided into the following cases:
	\begin{enumerate*}[label=(\roman*)]
		\item $p,q>0$,
		\item $p=0$ and $q>0$,
		\item $p>0$ and $q=0$, and
		\item $p=q=0$.
	\end{enumerate*}
	For the case (i), we have
	\begin{align*}
	\bar{B}_{(a,u+p)(b,u+q)}(u)-\bar{B}_{(a,u+p)(b,u+q)}(u+1)
	= \sum_{c=1}^{r} 
	\bigl( n_{ac;p}^{+} \check{n}_{bc;q}^{-}
	- n_{ac;p}^{-} \check{n}_{bc;q}^{+} \bigr)
	\end{align*}
	by \eqref{eq:def of bar B(u)}, and this yields the desired equality since $\varphi_{u}(a,u+p)=(a,u+p)$ and $\varphi_{u}(b,u+q)=(b,u+q)$.
	For the case (ii), we have
	$\varphi_{u}(a,u)=(\hat{a},u+p_{\hat{a}})$
	where $\hat{a}=\sigma(a)$.
	Then we have
	\begin{align*}
		\bar{B}'_{(a,u)(b,u+q)}(u) = -\bar{B}_{(a,u)(b,u+q)}(u)=-\check{n}_{ba;q}^{+} + \check{n}_{ba;q}^{-}
	\end{align*}
	and
	\begin{align*}
		&\bar{B}_{(\hat{a},u+p_{\hat{a}})(b,u+q)}(u+1)
		= 
		-n_{\hat{a}b;p_{\hat{a}}-q}^{+} 
		+n_{\hat{a}b;p_{\hat{a}}-q}^{-}
		+\check{n}_{b\hat{a};q-p_{\hat{a}}}^{+} 
		-\check{n}_{b\hat{a};q-p_{\hat{a}}}^{-} \\
		&\quad\quad +\sum_{c=1}^{r} \sum_{v=1}^{\min(p_{\hat{a}},q)}
		\bigl( n_{\hat{a}c;p_{\hat{a}}-v}^{+} \check{n}_{bc;q-v}^{-}
		- n_{\hat{a}c;p_{\hat{a}}-v}^{-} \check{n}_{bc;q-v}^{+} \bigr).
	\end{align*}
	These coincide by the $(b\hat{a};q-p_{\hat{a}})$-th entry in the symplectic relation, together with (\ref{item:N1}) and (\ref{item:N3}).
	For the case (iii), we have
	$\varphi_{u}(b,u)=(\hat{b},u+p_{\hat{b}})$
	where $\hat{b}=\sigma(b)$.
	Then we have
	\begin{align*}
	\bar{B}'_{(a,u+p)(b,u)}(u) = -\bar{B}_{(a,u+p)(b,u)}(u)
	=n_{ab;p}^{+} - n_{ab;p}^{-}
	\end{align*}
	and
	\begin{align*}
	&\bar{B}_{(a,u+p)(\hat{b},u+p_{\hat{b}})}(u+1)
	= 
	-n_{a\hat{b};p-p_{\hat{b}}}^{+}
	+n_{a\hat{b};p-p_{\hat{b}}}^{-}
	+\check{n}_{\hat{b}a;p_{\hat{b}}-p}^{+}
	-\check{n}_{\hat{b}a;p_{\hat{b}}-p}^{-} \\
	&\quad\quad +\sum_{c=1}^{r} \sum_{v=1}^{\min(p,p_{\hat{b}})}
	\bigl( n_{ac;p-v}^{+} \check{n}_{\hat{b}c;p_{\hat{b}}-v}^{-}
	- n_{ac;p-v}^{-} \check{n}_{\hat{b}c;p_{\hat{b}}-v}^{+} \bigr)
	\end{align*}
	These coincide by the $(a\hat{b};p-p_{\hat{b}})$-th entry in the symplectic relation, together with (\ref{item:N1}) and (\ref{item:N3}).
	For the case (iv), we have
	$\varphi_{u}(a,u)=(\hat{a},u+p_{\hat{a}})$
	and
	$\varphi_{u}(b,u)=(\hat{b},u+p_{\hat{b}})$.
	Then we have
	\begin{align*}
	\bar{B}'_{(a,u)(b,u)}(u) = -\bar{B}_{(a,u)(b,u)}(u)=0.
	\end{align*}
	On the other hand, we have
	\begin{align*}
	&\bar{B}_{(\hat{a},u+p_{\hat{a}})(\hat{b},u+p_{\hat{b}})}(u+1)
	= 
	-n_{\hat{a}\hat{b};p_{\hat{a}}-p_{\hat{b}}}^{+} +n_{\hat{a}\hat{b};p_{\hat{a}}-p_{\hat{b}}}^{-}
	+\check{n}_{\hat{b}\hat{a};p_{\hat{b}}-p_{\hat{a}}}^{+}
	-\check{n}_{\hat{b}\hat{a};p_{\hat{b}}-p_{\hat{a}}}^{-} \\
	&\quad\quad +\sum_{c=1}^{r} \sum_{v=1}^{\min(p_{\hat{a}},p_{\hat{b}})}
	\bigl( n_{\hat{a}c;p_{\hat{a}}-v}^{+} \check{n}_{\hat{b}c;p_{\hat{b}}-v}^{-}
	- n_{\hat{a}c;p_{\hat{a}}-v}^{-} \check{n}_{\hat{b}c;p_{\hat{b}}-v}^{+} \bigr) =0
	\end{align*}
	by the $(\hat{a}\hat{b};p_{\hat{a}}-p_{\hat{b}})$-th entry in the symplectic relation, together with (\ref{item:N1}) and (\ref{item:N3}).
\end{proof}

\begin{lemma}
	$\bar{B}(u+t) = \psi^{t} (\bar{B}(u))$ for any $u \in \ZZ$.
\end{lemma}
\begin{proof}
	Since $R(u)=R(u+t)$,
	the lemma follows from the fact that 
	$\bar{B}(u+t)$ and $\bar{B}(u)$ have the same expression \eqref{eq:def of bar B(u)}.
\end{proof}

We define an index set $I$ by $I=R(0)$,
and define an $I \times I$ integer matrices $B$ by $B=\bar{B}(0)$, that is, $B=(B_{(a,p)(b,q)})_{(a,p),(b,q) \in R(0)}$ and 
\begin{equation}\label{eq:def of B from cap}
\begin{split}
	&B_{(a,p)(b,q)} = 
	-n_{ab;p-q}^{+} 
	+n_{ab;p-q}^{-}
	+\check{n}_{ba;q-p}^{+}
	-\check{n}_{ba;q-p}^{-} \\
	&\quad\quad +\sum_{c=1}^{r} \sum_{v=0}^{\min(p,q)}
	\bigl( n_{ac;p-v}^{+} \check{n}_{bc;q-v}^{-}
	- n_{ac;p-v}^{-} \check{n}_{bc;q-v}^{+} \bigr).
\end{split}
\end{equation}
We define a tuple of positive integer $d = (d_{a,u})_{(a,u)\in R(u)}$ by $d_{a,u} = d_a$, where $d_a$ is the $a$-th entry in $D$.
We also define $\mathbf{i}=\mathbf{i}(0) \mid \dots \mid \mathbf{i}(t-1)$ by
$\mathbf{i}(u) =(\varphi_{u-1} \circ \dots \circ \varphi_{0})^{-1} (R_0(u))$,
where each $\mathbf{i}(u)$ is endowed with the linear order coming from the linear order on $R_0(u)$.
Finally, we define $\nu = (\varphi_{t-1} \circ \dots \circ \varphi_{0})^{-1} \circ \psi^{t}$.

\begin{lemma}
	$\gamma=(B,d,\mathbf{i},\nu)$ is a complete mutation loop of length $r$.
\end{lemma}
\begin{proof}
	It is easy to check that
	$B$ is a skew-symmetrizable matrix with the symmetrizer $d$.
	Let us denote by $\vec{\varphi}_{u}$ the composition $\varphi_{u-1} \circ \dots \circ \varphi_{0}$.
	We now prove
	\begin{align}\label{eq:mutation induction}
	(\mu_{\mathbf{i}(u-1)} \circ \dots \circ \mu_{\mathbf{i}(0)})(B)
	=(\vec{\varphi}_{u})^{-1} (\bar{B}(u))
	\end{align}
	for any $u=0,\dots,t-1$ by induction on $u$.
	The equation \eqref{eq:mutation induction} 
	holds when $u=0$ since $B = \bar{B}(0)$ by definition.
	Suppose that $u>0$.
	By the induction hypothesis and Lemma \ref{lemma:B(u+1)=varphi mu B(u)}, we have
	\begin{align*}
	&\quad (\mu_{\mathbf{i}(u-1)} \circ \mu_{\mathbf{i}(u-2)} \circ \dots \circ \mu_{\mathbf{i}(0)}) (B) \\
	&= (\mu_{\mathbf{i}(u-1)} \circ (\vec{\varphi}_{u-1})^{-1}) (\bar{B}(u-1))\\
	&=((\vec{\varphi}_{u-1})^{-1} \circ \mu_{R_0(u)}) (\bar{B}(u-1))\\
	&=((\vec{\varphi}_{u-1})^{-1} \circ (\varphi_{u-1})^{-1}) (\bar{B}(u))\\
	&=(\vec{\varphi}_{u})^{-1} (\bar{B}(u)),
	\end{align*}
	and \eqref{eq:mutation induction} is proved.
	Applying \eqref{eq:mutation induction} for $u=t-1$ yields
	\begin{align*}
	\mu_{\mathbf{i}}(B)=(\vec{\varphi}_{t})^{-1} (\bar{B}(t))=((\vec{\varphi}_{t})^{-1} \circ \psi^{t}) (B).
	\end{align*}
	This shows that $\gamma$ is a mutation loop.
	By (\ref{item:R3}) in Definition \ref{def:consistent subset},
	we have
	\begin{align*}
		\{ (a,0) \mid a \in [1,r] \} = \bigsqcup_{u=0}^{t-1} \psi^{-u}(R_0(u))
	\end{align*}
	as a set.
	This implies that the length of $\gamma$ is $r$.
	The completeness follows from the fact that the latency of $((a,p),0) \in I \times \{ 0 \}$ is $p$.
\end{proof}

Now we complete the proof of Theorem \ref{theorem:mutation loop from T-datum} by showing the following lemma:

\begin{lemma}\label{lemma:equivalent}
	$(\alpha_{\gamma},R_{\gamma})$ and $(\alpha,R)$ are equivalent.
\end{lemma}
\begin{proof}
	By replacing $(\alpha,R)$ with a suitable equivalent one,
	we can assume that
	$u<v$ implies that $a<b$
	for any $(a,u) \in R_0(u)$ and $(b,v) \in R_0(v)$ such that $0\leq u,v \leq t-1$.
	Then the construction of $\gamma$ ensures that $(\alpha_{\gamma},R_{\gamma})=(\alpha,R)$.
\end{proof}

\subsection{Consequences}
\label{section:consequences}
For any complete mutation loop $\gamma$,
we denote by $F(\gamma)$ the pair
$(\alpha_{\gamma},R_{\gamma})$
defined in Section \ref{section:T-data from mutation loops}.
For any pair $(\alpha, R)$ of a T-datum and a consistent subset $R$ for $\alpha$,
we denote $G(\alpha,R)$ by
the complete mutation loop defined in Section \ref{section:Mutation loops from T-data}.

We define
\begin{align*}
	\mathrm{Ml}_r 
	= \{ [\gamma] \mid \text{$\gamma$ is a complete mutation loop of length $r$}  \},
\end{align*}
where $[\gamma]$ is the equivalence class of $\gamma$ (see Definition \ref{def:ml equivalence}).
We also define
\begin{align*}
	\mathrm{Td}_r
	=\{ \alpha \mid \text{$\alpha$ is a T-datum of size $r$} \}
\end{align*}
and
\begin{align*}
	\mathrm{Td}'_r
	=\{ [(\alpha,R)] \mid \text{$\alpha \in \mathrm{Td}_r$ and $R \subseteq [1,r] \times \ZZ$ is consistent for $\alpha$} \},
\end{align*}
where $[(\alpha,R)] $ is the equivalence class of $(\alpha,R)$ (see Definition \ref{def:cap R equivalence}).
We define
\begin{align*}
\hat{F}_r: \mathrm{Ml}_r \to \mathrm{Td}'_r, \quad
\hat{G}_r: \mathrm{Td}'_r \to \mathrm{Ml}_r 
\end{align*}
by $\hat{F}_r([\gamma]) = [F(\gamma)]$
and $\hat{G}_r([(\alpha,R)])= [G(\alpha,R)]$.

\begin{theorem}\label{theorem:Ml Td bijection}
	$\hat{F}_r \circ \hat{G}_r = \id$ and $\hat{G}_r \circ \hat{F}_r = \id$.
\end{theorem}
\begin{proof}
	We first show that $\hat{F}_r$ and $\hat{G}_r$ are well-defined.
	Let $\gamma$ and $\gamma'$ be equivalent complete
	mutation loops, and $\rho$ be the permutation in Definition \ref{def:ml equivalence}.
	Then we have $A_{\gamma',\pm}=\rho(A_{\gamma,\pm})$ and
	$R_{\gamma'}=\rho(R_{\gamma})$.
	Thus $\hat{F}_r$ is well-defined.
	
	Let $(\alpha,R)$ and $(\alpha',R')$ be equivalent pairs of T-data and consistent subsets for them.
	Let $\rho \in \mathfrak{S}_r$ be the permutation in Definition \ref{def:cap R equivalence}.
	Let $\gamma=(B,d,\mathbf{i},\nu)$ and
	$\gamma'=(B',d',\mathbf{i}',\nu')$ be the mutation loops given by $\gamma=G(\alpha,R)$ and $\gamma'=G(\alpha',R')$, respectively.
	Then the bijection
	$f_{\rho,u}:R(u) \to R'(u)$ defined by $f_{\rho,u}(a,u)=(\rho(a),u)$
	satisfies $B'=f_{\rho,0}(B)$, $d'=f_{\rho,0}(d)$ and $\nu'=f_{\rho,0} \circ \nu \circ f_{\rho,0}^{-1}$.
	Moreover, $R'_0(u)$ and $f_{\rho,u}(R_0(u))$ coincide as sets.
	Thus $\gamma$ and $\gamma'$ are equivalent, and $\hat{G}_r$ is well-defined.
	
	By Lemma \ref{lemma:equivalent}, we have $\hat{F}_r \circ \hat{G}_r = \id$.
	It remains to show that $\hat{G} \circ \hat{F} = \id$.
	Let $\gamma$ be any complete mutation loop,
	and $I$ be the index set of $\gamma$.
	Let
	$\gamma'= G(\alpha_{\gamma},R_\gamma)$.
	Then the index set $I'$ of $\gamma'$ is given by
	$I' = \{ (a,u) \in R_{\gamma} \mid 0 \leq u < p_a \}$.
	Let $f:I \to I'$ be the map defined by
	\begin{align*}
		f(i) = 
		\begin{cases}
			(\pi(i,0),0) & \text{if $(i,0) \in P_\gamma$},\\
			(\pi(s(i,0)),\lambda(i,0)) & \text{if $(i,0) \notin P_\gamma$},
		\end{cases}
	\end{align*}
	where $\pi$, $s$, and $\lambda$, $P_\gamma$ are defined in Section \ref{section:T-system and Y-system}.
	It is easy to check that $f$ is a bijection,
	and gives the equivalence between $\gamma$ and $\gamma'$. 
\end{proof}

For any consistent subset $R \subseteq [1,r] \times \ZZ$ for a T-datum of size $r$,
we define the set $R_{\mathrm{in}}$ (the subscript ``in'' stands for initial)
by $R_{\mathrm{in}} = \{ (a,p) \in R \mid 0 \leq p < p_a \}$.

\begin{theorem}
	\label{theorem:embedding into cluster algebra}
	Let $\alpha$ be a T-datum of size $r$.
	Suppose that $R \subseteq [1,r] \times \ZZ$ is consistent for $\alpha$.
	Let $\PP$ be a semifield, and $\mathcal{F}$ be a filed that is isomorphic to the field of rational functions over $\QQ\PP$ in $\lvert R_{\mathrm{in}} \rvert$ variables.
	Let $x=(x_{a,p})_{(a,p) \in R_{\mathrm{in}}}$ be an $R_{\mathrm{in}}$-tuple of elements in $\mathcal{F}$ forming a free generating set.
	Let $Y=(Y_a(u))_{(a,u) \in R }$ be a solution of the Y-system associated with $(\alpha,R)$ in $\PP$.
	Then there exists a unique $R_{\mathrm{in}}$-labeled Y-seed $(B,y)$ in $\PP$ such that
	\begin{enumerate}
		\item there exists a unique injective $\ZZ\PP$-algebra homomorphism $\iota: \mathscr{T}^{\circ}(\alpha,R,Y) \hookrightarrow \mathcal{A}(B,y,x)$ such that $\iota(T_a(p)) = x_{a,p}$ for any $(a,p) \in R_{\mathrm{in}}$,
		\item $\iota(T_a(u))$ is a cluster variable in $\mathcal{A}(B,y,x)$ for any $(a,u) \in R$,
		\item the image of the relation \eqref{eq:standalone T-system} by $\iota$ is an exchange relation in $\mathcal{A}(B,y,x)$ for any $(a,u) \in R$.
	\end{enumerate}
\end{theorem}
\begin{proof}
	Let $\gamma=(B,d,\mathbf{i},\nu)$ be the mutation loop given by $\gamma=G(\alpha,R)$.
	By Lemma \ref{lemma:equivalent}, we can assume that $\alpha=\alpha_{\gamma}$ and $R=R_{\gamma}$.
	We define a family of elements $y=(y_{a,p})_{(a,p) \in R_{\mathrm{in}}}$ in $\PP$ by
	\begin{align}\label{eq:y from Y}
		y_{a,p} = Y_a(p) \frac{\prod_{b=1}^{r} \prod_{q=0}^{p} \bigl( 1 \oplus Y_b(p-q)^{-1} \bigr)^{\check{n}_{ab;q}^{+}}}{\prod_{b=1}^{r} \prod_{q=0}^{p} \bigl( 1 \oplus Y_b(p-q) \bigr)^{\check{n}_{ab;q}^{-}}}.
	\end{align}
	Then we can define a family of seeds
	$(B(u),y(u),x(u))_{u \in \ZZ}$
	by \eqref{eq:infinite seeds},
	where $(B(0),y(0),x(0)) = (B,y,x)$.
	Note that $\mathbf{i}(u) = (\vec{\varphi}_u)^{-1}(R_0(u))$ for any $u \in \ZZ$ by the commutativity of $\psi^t$ and $\varphi$,
	where we define $\vec{\varphi}_u:R(0) \to R(u)$ by
	\begin{align}\label{eq:def of vec varphi}
		\vec{\varphi}_u =
		\begin{cases}
			\varphi_{u-1} \circ \dots \circ \varphi_0
			& \text{if $u\geq 0$},\\
			(\varphi_{-1} \circ \dots \circ \varphi_u)^{-1}
			& \text{if $u< 0$}.
		\end{cases}
	\end{align}
	We define $Y' = (Y'_a(u))_{(a,u) \in R_{\gamma}}$
	by \eqref{eq: def of Y_a(u)},
	where $Y_a(u)$ in \eqref{eq: def of Y_a(u)} is replaced with $Y'_a(u)$.
	Then we have $Y_a(p) = Y'_a(p)$ for any $(a,p) \in R_{\mathrm{in}}$ by \eqref{eq:y from Y}.
	Moreover, $Y$ and $Y'$ are solutions of the same Y-system,
	we have $Y_a(u) = Y'_a(u)$ for any $(a,u) \in R$.
	We define $T'_a(u)$ for any $(a,u) \in R_{\gamma}$ by \eqref{eq: def of T_a(u)},
	where $T_a(u)$ in \eqref{eq: def of T_a(u)} is replaced with $T'_a(u)$.
	Let $\mathcal{A}_{\gamma}(B,y,x)$ be the $\ZZ\PP$-subalgebra of $\mathcal{A}(B,y,x)$ generated by $(x_{a,p}(u))_{(a,p) \in R_{\mathrm{in}},u \in \ZZ}$.
	It is also generated by $(T'_a(u))_{(a,p) \in R_{\gamma}}$.
	We now show that $\mathcal{A}_{\gamma}(B,y,x)$ is isomorphic to $\mathscr{T}^{\circ}(\alpha,R,Y)$.
	Let $\bar{\iota}:\mathscr{T}^{\circ}(\alpha,R,Y) \to \mathcal{A}_{\gamma}(B,y,x)$ be the algebra homomorphism defined by $\bar{\iota}(T_a(u)) = T'_a(u)$.
	This is well-defined by Proposition \ref{prop: T-system}.
	To construct the inverse of $\bar{\iota}$, we define an algebra homomorphism
	$\kappa: \ZZ\PP[x_{a,p}^{\pm1}]_{(a,p) \in R_{\mathrm{in}}} \to \mathscr{T}(\alpha,R,Y)$ by $\kappa(x_{a,p}^{\pm1}) = T_a(p)^{\pm1}$.
	By the Laurent phenomenon of cluster algebras~\cite{FZ1},
	$\mathcal{A}_{\gamma}(B,y,x)$ is a subalgebra of $\ZZ\PP[x_{a,p}^{\pm1}]_{(a,p) \in R_{\mathrm{in}}}$.
	Then we have $\kappa(T'_a(u)) = T_a(u)$ since
	$(T'_a(u))_{(a,p) \in R_{\gamma}}$ and $(T_a(u))_{(a,p) \in R}$
	satisfy the same T-system.
	Thus we obtain the algebra homomorphism $\bar{\kappa}: \mathcal{A}_{\gamma}(B,y,x) \to \mathscr{T}^{\circ}(\alpha,R,Y)$
	as the restriction of $\kappa$,
	which is the inverse of $\bar{\iota}$.
	Therefore, the existence of a Y-seed satisfying (1)--(3) is proved.
	
	The uniqueness follows from the following facts, which hold for any skew-symmetrizable cluster algebra:
	\begin{enumerate*}[label=(\roman*)]
		\item given a cluster as a set, the exchange relations involving its elements are uniquely determined~\cite[Proposition 6.1]{cao2018enough},
		\item any two clusters that have $\lvert I \rvert-1$ common cluster variables are related by an exchange relation~\cite[Theorem 5]{GSV08}.
	\end{enumerate*}
\end{proof}

\begin{example}[Somos-4 recurrence]
	\label{example:somos4 ml}
	Let $\alpha$, $R$, and $Y$ be as in Example \ref{example:cap somos4}.
	Then the Y-seed given by Theorem \ref{theorem:embedding into cluster algebra} is
	\begin{align}\label{eq:y-seed somos4}
		\begin{tikzpicture}[scale=1.9,baseline=(base.base),vertex/.style={scale=1}]
			\node (base) at (0.5,-0.5) [opacity=1] {};
			\node (1) at (0,0)[vertex] {$(1,0)$};
			\node (2) at (1,0)[vertex] {$(1,1)$};
			\node (3) at (1,-1)[vertex] {$(1,2)$};
			\node (4) at (0,-1)[vertex] {$(1,3)$};
			\node (c1) at (1.6,-0.5)[vertex] {$\boxed{c_1}$};
			\node (c2) at (-0.6,-0.5)[vertex] {$\boxed{c_2}$};
			\foreach \from/\to in {2/1,4/3,4/1,1/c2,c2/4,c1/1,c1/2,3/c1,4/c1}
			\draw[arrows={-Stealth[scale=1.2]}] (\from)--(\to);
			\foreach \from/\to in {1/3,2/4}
			\draw[arrows={-Stealth[scale=1.2]Stealth[scale=1.2]}] (\from)--(\to);
			\foreach \from/\to in {3/2}
			\draw[arrows={-Stealth[scale=1.2]Stealth[scale=1.2]Stealth[scale=1.2]}] (\from)--(\to);
		\end{tikzpicture},
	\end{align}
	where we represent the Y-seed using a quiver with frozen vertices as in the cluster algebra literature (see e.g., \cite{fomin2016introduction}).
	For instance,
	$\begin{tikzpicture}[scale=2.1,baseline=(i.base),vertex/.style={scale=1}]
	\node (i) at (0,0)[vertex] {$i$};
	\node (c1) at (0.5,0)[vertex] {$\boxed{c_1}$};
	\node (c2) at (-0.5,0)[vertex] {$\boxed{c_2}$};
	\foreach \from/\to in {c1/i,i/c2}
	\draw[arrows={-Stealth[scale=1.2]}] (\from)--(\to);
	\end{tikzpicture}$
	means $y_i = c_1 c_2^{-1}$.
	The mutation loop $\gamma = (B,d,\mathbf{i},\nu)=G(\alpha,R)$ is given by
	\begin{itemize}
		\item $B=B(Q)$,
		\item $d_i=1$ for any $i \in I$,
		\item $\mathbf{i}= \mathbf{i}(0)$ with
		$\mathbf{i}(0) =((1,0))$,
		\item $\nu=((1,0) \, (1,1)\, (1,2)\, (1,3))$,
	\end{itemize}
	where $Q$ is an underlying quiver in \eqref{eq:y-seed somos4}, and
	$\nu$ is the cyclic
	permutation corresponding to the right $\pi/2$ rotation of the quiver.
	In fact, the quiver mutation $\mu_{(1,0)}$ is given by
	\begin{align*}
	\begin{tikzpicture}[scale=1.9,baseline=(base.base),vertex/.style={scale=1}]
	\node (base) at (0.5,-0.5) {};
	\node (1) at (0,0)[vertex] {$(1,0)$};
	\node (2) at (1,0)[vertex] {$(1,1)$};
	\node (3) at (1,-1)[vertex] {$(1,2)$};
	\node (4) at (0,-1)[vertex] {$(1,3)$};
	%
	\foreach \from/\to in {2/1,4/3,4/1}
	\draw[arrows={-Stealth[scale=1.2]}] (\from)--(\to);
	\foreach \from/\to in {1/3,2/4}
	\draw[arrows={-Stealth[scale=1.2]Stealth[scale=1.2]}] (\from)--(\to);
	\foreach \from/\to in {3/2}
	\draw[arrows={-Stealth[scale=1.2]Stealth[scale=1.2]Stealth[scale=1.2]}] (\from)--(\to);
	\end{tikzpicture}
	\xmapsto{\mu_{(1,0)}}
	\begin{tikzpicture}[scale=1.9,baseline=(base.base),vertex/.style={scale=1}]
	\node (base) at (0.5,-0.5) [opacity=1] {};
	\node (1) at (0,0)[vertex] {$(1,0)$};
	\node (2) at (1,0)[vertex] {$(1,1)$};
	\node (3) at (1,-1)[vertex] {$(1,2)$};
	\node (4) at (0,-1)[vertex] {$(1,3)$};
	%
	\foreach \from/\to in {1/2,3/2,1/4}
	\draw[arrows={-Stealth[scale=1.2]}] (\from)--(\to);
	\foreach \from/\to in {2/4,3/1}
	\draw[arrows={-Stealth[scale=1.2]Stealth[scale=1.2]}] (\from)--(\to);
	\foreach \from/\to in {4/3}
	\draw[arrows={-Stealth[scale=1.2]Stealth[scale=1.2]Stealth[scale=1.2]}] (\from)--(\to);
	\end{tikzpicture},
	\end{align*}
	and we have $\mu_{(1,0)}(Q)=\nu(Q)$.
\end{example}

\begin{example}[Bipartite belt]
	\label{example:b belt ml}
	Let $\alpha$ and $R$ be as in Example \ref{example:cap b belt}.
	The mutation loop $\gamma = (B,d,\mathbf{i},\nu) = G(\alpha,R)$ is given by
	\begin{itemize}
		\item $B=(B_{(a,p)(b,q)})_{(a,p),(b,q) \in I}$, where $B_{(a,p)(b,q)} = \epsilon(a) n_{ab}$,
		\item $d=(d_{a,p})_{(a,p)\in I}$, where $D=\diag (d_1 ,\dots, d_r)$ and $d_{a,p}=d_a$,
		\item $\mathbf{i} = \mathbf{i}(0) \mid \mathbf{i}(1)$ with $\mathbf{i}(0)=\{ (a,0) \mid \epsilon(a)=-1 \}$ and
		$\mathbf{i}(1)=\{ (a,1) \mid \epsilon(a)=1 \}$,
		\item 	$\nu = \id$,
	\end{itemize}
	where $I= \{ (a,0) \mid \epsilon(a)=-1 \} \sqcup \{ (a,1) \mid \epsilon(a)=1 \}$.
	If $A$ is the Cartan matrix of type $A_2$, for instance, the quiver $Q(B)$ is given by
	\begin{align*}
		Q(B)=
		\begin{tikzpicture}[scale=1.9,baseline=(1.base),vertex/.style={scale=1}]
		\node (1) at (0,0)[vertex] {$(1,0)$};
		\node (2) at (1,0)[vertex] {$(2,1)$};
		\foreach \from/\to in {2/1}
		\draw[arrows={-Stealth[scale=1.2]}] (\from)--(\to);
		\end{tikzpicture}.
	\end{align*}
	If $A$ is a Cartan matrix of finite type,
	$\mathcal{A}(B,y,x)$ is a finite type cluster algebra and
	the embedding $\iota:\mathscr{T}^{\circ}(\alpha,R,Y) \hookrightarrow \mathcal{A}(B,y,x)$ in Theorem \ref{theorem:embedding into cluster algebra} is an isomorphism~\cite[Proposition 11.1]{FZ4}.
\end{example}

\subsection{Tropical T-system}
\label{section:tropical T}
By Theorem \ref{theorem:embedding into cluster algebra} and the Laurent phenomenon of cluster algebras~\cite{FZ1}, we obtain the following:

\begin{corollary}\label{corollary:laurent phenomenon}
	Let $\mathscr{T}^{\circ}(\alpha,R,Y)$ be a T-algebra.
	Then
	$T_a(u) \in \mathscr{T}^{\circ}(\alpha,R,Y)$ can be written as a Laurent polynomial in $(T_c(p))_{(c,p) \in R_{\mathrm{in}}}$ with coefficients in $\ZZ\PP$, for any $(a,u) \in R$.
\end{corollary}

Let $R= [1,r] \times \ZZ$.
By Corollary \ref{corollary:laurent phenomenon},
any $T_a(u)$ can be uniquely written as
\begin{align}\label{eq:denominator}
	T_a(u) = \frac{N}{ \prod_{(c,p) \in R_{\mathrm{in}}} T_c(p)^{d_{c,p}} },
\end{align}
where $N$ is a polynomial in $(T_c(p))_{(c,p) \in R_{\mathrm{in}}}$
with coefficients in $\ZZ\PP$ which is not divisible by any $T_c(p)$ ($(c,p) \in R_{\mathrm{in}}$).
We denote by $\mathfrak{t}_{a}^{(c)}(u)$ the integer $d_{c,0}$ in \eqref{eq:denominator}.
The family of integers $(\mathfrak{t}_a^{(c)} (u))_{(a,u) \in [1,r] \times \ZZ}$ is uniquely
determined by the initial conditions
\begin{align}\label{eq:tropical T initial}
\mathfrak{t}_a^{(c)} (p)
=
\begin{cases}
-1 & \text{if $(a,p) = (c,0)$,}\\
0	& \text{if $(a,p) \neq (c,0)$ and $0 \leq p <p_a$,}
\end{cases}
\end{align}
together with the following recurrence relation for each $(a,u) \in [1,r] \times \ZZ$:
\begin{align}\label{eq:tropical T}
\sum_{b,p} n_{ba;p}^0 \mathfrak{t}_b^{(c)} (u+p) &=
\max \biggl(\sum_{b,p} n_{ba;p}^- \mathfrak{t}_b^{(c)} (u+p),\sum_{b,p} n_{ba;p}^+ \mathfrak{t}_b^{(c)} (u+p) \biggr) .
\end{align}
In particular, the integer $\mathfrak{t}_{a}^{(c)}(u)$ is independent of the choice of $Y$.
The family of relations \eqref{eq:tropical T} is called the \emph{tropical T-system} associated with $\alpha$.

We also define a family of integers $(\hat{\mathfrak{y}}_a^{(c)} (u))_{(a,u) \in [1,r] \times \ZZ}$ by
\begin{align}
\label{eq:tropical Y hat 1}
\hat{\mathfrak{y}}_a^{(c)}(u)
&=
\sum_{b,p} \big( n_{ba;p}^{-} \mathfrak{t}_b^{(c)} (u+p)
-n_{ba;p}^{+} \mathfrak{t}_b^{(c)} (u+p) \bigr) \\
\label{eq:tropical Y hat 2}
&=
\sum_{b,p} \bigl( (n_{ba;p}^{0}-n_{ba;p}^{+}) \mathfrak{t}_b^{(c)} (u+p)
-(n_{ba;p}^{0}-n_{ba;p}^{-}) \mathfrak{t}_b^{(c)} (u+p) \bigr).
\end{align}
By the relation \eqref{eq:tropical T}, we have
\begin{align}\label{eq:tropical Y hat pm}
\maxzero{ \pm \hat{\mathfrak{y}}_a^{(c)} (u)}
&= \sum_{b,p} (n_{ba;p}^{0}-n_{ba;p}^{\pm}) \mathfrak{t}_b^{(c)} (u+p).
\end{align}

The following lemma will be used in Section \ref{section:periodic yt}.

\begin{lemma}\label{lemma:tropical T and Y around initial}
	The following equalities hold for any T-datum:
	\begin{enumerate}
		\item For any $a \in [1,r]$, we have
		\begin{align*}
		\mathfrak{t}_a^{(c)} (p_a)
		=
		\begin{cases}
		1	& \text{if $a = \sigma(c)$,}\\
		0 & \text{otherwise.}
		\end{cases}
		\end{align*}
		\item For any $a \in [1,r]$ and $0 \leq p \leq p_c$,
		we have
		\begin{align*}
		\mathfrak{t}_a^{(c)} (-p)
		=
		\begin{cases}
		-1 & \text{if $(a,p)=(c,0)$},\\
		1	& \text{if $(a,p) = (\sigma^{-1}(c),p_c)$,}\\
		0 & \text{otherwise.}
		\end{cases}
		\end{align*}
		\item For any $a \in [1,r]$ and $0 \leq p \leq p_c$,
		we have
		\begin{align*}
		\maxzero{ \pm \hat{\mathfrak{y}}_a^{(c)}(-p) }
		= n_{ca;p}^{\pm}.
		\end{align*}
	\end{enumerate}
\end{lemma}
\begin{proof}
	(1) is clear from \eqref{eq:tropical T initial} and \eqref{eq:tropical T}.
	We prove (2) by induction on $p$.
	The case $p=0$ follows from \eqref{eq:tropical T initial}.
	Suppose that $p>0$. Then we have
	\begin{align*}
		\mathfrak{t}_a^{(c)} (-p) &= -\mathfrak{t}_{\sigma(a)}^{(c)}(-p+p_{\sigma(a)})
		+\max \biggl(\sum_{b,q} n_{ba;q}^- \mathfrak{t}_b^{(c)} (-p+q),\sum_{b,q} n_{ba;q}^+ \mathfrak{t}_b^{(c)} (-p+q) \biggr)\\
		&= \delta_{\sigma(a)c}\delta_{pp_c} + \max( -n_{ca;p}^{+},-n_{ca;p}^{-} ) \\
		&=\delta_{\sigma(a)c}\delta_{pp_c},
	\end{align*}
	and (2) is proved.
	We now prove (3).
	From \eqref{eq:tropical Y hat pm}, we have
	\begin{align*}
	\maxzero{ \pm \hat{\mathfrak{y}}_a^{(c)}(-p) }
	&= 
	\sum_{b,q} (n_{ba;q}^{0} - n_{ba;q}^{\pm}) \mathfrak{t}_b^{(c)} (-p+q)\\
	&=
	\mathfrak{t}_{a}^{(c)}(-p)
	+\mathfrak{t}_{\sigma(a)}^{(c)}(-p+p_{\sigma(a)})
	-n_{ca;p}^{\pm} \mathfrak{t}_{c}^{(c)}(0) \\
	&=
	\mathfrak{t}_{a}^{(c)}(-p)
	+\mathfrak{t}_{\sigma(a)}^{(c)}(-p+p_{\sigma(a)})
	+n_{ca;p}^{\pm}.
	\end{align*}
	By (1) and (2) in this lemma,
	we have
	\begin{align*}
		\mathfrak{t}_a^{(c)} (-p)
		=
		\begin{cases}
			1	& \text{if $(a,p) = (\sigma^{-1}(c),p_c)$},\\
			-1 & \text{if $(a,p)=(c,0)$},\\
			0 & \text{otherwise},
		\end{cases}
	\end{align*}
	and
	\begin{align*}
		\mathfrak{t}_{\sigma(a)}^{(c)} (-p+p_{\sigma(a)})
		=
		\begin{cases}
			1	& \text{if $(a,p) = (c,0)$},\\
			-1 & \text{if $(a,p) = (\sigma^{-1}(c),p_c)$},\\
			0 & \text{otherwise}.
		\end{cases}
	\end{align*}
	Thus we have
	$\mathfrak{t}_{a}^{(c)}(-p)+\mathfrak{t}_{\sigma(a)}^{(c)}(-p+p_{\sigma(a)}) = 0$.
	This completes the proof of (3).
\end{proof}

\subsection{Indecomposable T-data}
If $(A_+,A_-,D)$ and $(A'_+,A'_-,D')$ are T-data, the direct sum
\begin{align*}
\left( 
\begin{bmatrix}
A_{+} & O\\
O & A'_{+}
\end{bmatrix},
\begin{bmatrix}
A_{-} & O\\
O & A'_{-}
\end{bmatrix},
\begin{bmatrix}
D & O\\
O & D'
\end{bmatrix}
\right)
\end{align*}
is also a T-datum.
A T-datum $(A_+, A_-,D)$ is called \emph{decomposable} 
if it can be written as a nontrivial direct sum after reordering the indices of matrices.
A T-datum that is not decomposable is called \emph{indecomposable}.

We say that a skew-symmetrizable matrix $B$ is \emph{connected} if it cannot be written as a nontrivial direct sum.
We also say that a connected skew-symmetrizable matrix $B'=(B'_{ij})_{i,j \in I'}$ is a \emph{connected component} of a skew-symmetrizable matrix $B=(B_{ij})_{i,j \in I}$ if $I' \subseteq I$, $B'_{ij} = B_{ij}$ for any $i,j \in I'$, and $B_{ij}=0$ for any $i \in I'$ and $j \in I \setminus I'$.

\begin{proposition}\label{prop:R from connected component}
	Let $\alpha$ be an indecomposable T-datum.
	Let $I'$ be the index set of a connected component of $B$,
	where $B$ is the skew-symmetrizable matrix in the mutation loop $G(\alpha,[1,r] \times \ZZ)$.
	Then the set
	\begin{align*}
		R' := \bigcup_{u \in \ZZ} \vec{\varphi}_u(I')
	\end{align*}
	is consistent for $\alpha$,
	where $\vec{\varphi}_u$ is defined by \eqref{eq:def of vec varphi}.
	Moreover, the index set of $B'$ is $I'$, where $B'$ is the skew-symmetrizable matrix in the mutation loop $G(\alpha,R')$.
\end{proposition}
\begin{proof}
	Let $(B,d,\mathbf{i},\nu) = G(\alpha,[1,r]\times \ZZ)$.
 	The bijection $\nu$ is given by $\nu = \varphi_{0}^{-1} \circ \psi$.
	Let $I$ be the index set of $B$, that is,
	$I= \{ (a,p) \mid a\in [1,r] , 0\leq p <p_a \}$.
	We recursively define subsets $I'_k \subseteq I$ for $k\in \ZZ_{>0}$ by $I'_0 = I'$ and $I'_k = \nu|_{I'_{k-1}}(I'_{k-1})$.
	Then $B|_{I'_k}$ is a connected component of $B$ since mutations preserve connected components.
	We now prove (\ref{item:R1}) and (\ref{item:R2}) in Definition \ref{def:consistent subset} for $R'$.
	The proof of (\ref{item:R1}) and (\ref{item:R2}) are almost the same, we only prove (\ref{item:R1}).
	Suppose that $(a,u) \in R'$.
	If $\check{n}_{ab;p}^{0} \neq 0$, then $(b,u-p) \in R'$ by the definition of $R'$.
	Suppose that $\check{n}_{ab;p}^{+}$ or $\check{n}_{ab;p}^{-} \neq 0$.
	Then we have $(a,u),(b,u-p) \in \vec{\varphi}_{u-p}(I)$ and $\bar{B}_{(b,u-p)(a,u)}(u-p) \neq 0$, where $\bar{B}(u-p)$ is defined in \eqref{eq:def of bar B(u)}.
	This shows that $(a,p)$ and $(b,0)$ lie in the same connected component $B|_{I'_{p-u}}$.
	Thus we have $(b,u-p) \in R'$ since
	$\vec{\varphi}_{u-p}(\nu^{u-p}(b,0)) = (b,u-p)$ and
	$\nu^{u-p}(b,0) \in I'$.
	
	We next prove (\ref{item:R3}).
	Let $t$ be the smallest positive integer such that $I'_t = I'$.
	We now show that
	\begin{align}\label{eq:I'= sqcup I'k}
		I = \bigsqcup_{k=0}^{t-1} I'_k.
	\end{align}
	The equality $I = \bigcup_{k=0}^{t-1} I'_k$ follows from the fact that $\alpha$ is indecomposable. 
	Suppose that $I'_{k_1} \cap I'_{k_2} \neq \emptyset$ for some $0\leq k_1 < k_1 <t$.
	Then we have $I'_{k_1} = I'_{k_2}$ since $B|_{I'_{k_1}}$ and $B|_{I'_{k_2}}$ are connected components of $B$.
	But this implies $I' = I'_{k_2-k_1}$, which contradicts the minimality of $t$.
	Thus \eqref{eq:I'= sqcup I'k} is proved.
	We now have
	\begin{align*}
		&\bigcup_{k=0}^{t-1} \psi^{k} (R') = \bigcup_{k=0}^{t-1}\bigcup_{u \in \ZZ} (\psi^{k} \circ \vec{\varphi}_u) (I')
		=\bigcup_{k=0}^{t-1} \bigcup_{u \in \ZZ}  (\vec{\varphi}_{u+k} \circ \nu^{k}) (I')\\
		&=\bigcup_{k=0}^{t-1} \bigcup_{u \in \ZZ}  \vec{\varphi}_{u+k} (I'_k)
		=\bigcup_{k=0}^{t-1}\bigcup_{u \in \ZZ} \vec{\varphi}_{u} (I'_k)
		=\bigcup_{u \in \ZZ} \vec{\varphi}_{u} \biggl(\bigcup_{k=0}^{t-1} I'_k\biggr)\\
		&=\bigcup_{u \in \ZZ} \vec{\varphi}_{u} (I)
		=[1,r] \times \ZZ.
	\end{align*}
	It remains to prove the disjointness.
	Suppose that there exists a element $(a,u) \in \psi^{k_1}(R') \cap \psi^{k_2}(R')$ for some $0\leq k_1 <k_2 <t$.
	Then we have $((\vec{\varphi}_{u-k_i})^{-1} \circ \psi^{-k_i}) (a,u) \in I'$ for $i=1,2$.
	Thus we have $(\nu^{k_i} \circ (\vec{\varphi}_{u-k_i})^{-1} \circ \psi^{-k_i}) (a,u) \in I'_{k_i}$ for $i=1,2$.
	On the other hand, we have $\nu^{k_i} \circ (\vec{\varphi}_{u-k_i})^{-1} \circ \psi^{-k_i} = (\vec{\varphi}_u)^{-1}$.
	This implies that $(\vec{\varphi}_u)^{-1} (a,u) \in I'_{k_1} \cap I'_{k_2}$,
	which contradicts \eqref{eq:I'= sqcup I'k}.
\end{proof}

\begin{corollary}\label{cor:Npm=0}
	Let $\alpha=(A_+,A_-,D)$ be an indecomposable T-datum of size $r$.
	Suppose that there exists $a \in [1,r]$ such that
	both the $a$-th columns in $N_+$ and $N_-$ are zero vectors.
	Then both $N_+$ and $N_-$ are zero matrices.
\end{corollary}
\begin{proof}
	Let $(B,d,\mathbf{i},\nu)=G(\alpha,[1,r]\times \ZZ)$.
	By the assumption and \eqref{eq:def of B from cap},
	the set $I'=\{ (a,0) \}$ is the index set of a connected component of $B$.
	Then the set $R'$ defined in Proposition \ref{prop:R from connected component} is consistent for $\alpha$.
	Let $(B',d',\mathbf{i}',\nu')=G(\alpha,R')$.
	By Proposition \ref{prop:R from connected component},
	$B'$ is an $I' \times I'$ matrix.
	This implies that $B'=0$ since
	any skew-symmetrizable matrix of size $1$ should be the zero matrix.
	From \eqref{eq: NT+-} and Lemma \ref{lemma:equivalent}, we have $N_\pm = 0$.
\end{proof}

\section{Examples of T-data}
\label{section:examples}
\subsection{Period 1 quivers}
\label{section:period 1}

\begin{theorem}\label{theorem:T-datum of size 1}
	Let $p > 0$ be a positive integer,
	and $a(z) = 1 + n_1 z + \dots + n_{p-1} z^{p-1} + z^{p} \in \ZZ[z]$
	be a monic palindromic polynomial of degree $p$, that is,
	$n_{q} = n_{p-q}$ for any $0< q < p$.
	Let $d$ be any positive integer.
	Then the triple $\alpha = (A_+,A_-,[d])$ given
	by $A_\pm = N_0-N_\pm$ where
	\begin{align*}
		N_0 = 
		\begin{bmatrix}
			1 + z^p
		\end{bmatrix}, \quad
		N_+ = 
		\begin{bmatrix}
			{\displaystyle \sum_{q=1}^{p-1} \maxzero{n_q} z^q}
		\end{bmatrix}, \quad
		N_- = 
		\begin{bmatrix}
			{\displaystyle \sum_{q=1}^{p-1} \maxzero{-n_q} z^q}
		\end{bmatrix}, \quad
	\end{align*}
	is a T-datum of size $1$.
	Furthermore, any T-datum of size $1$ is of this form.
\end{theorem}
\begin{proof}
	The conditions (\ref{item:N1})--(\ref{item:N4}) follow immediately from the definition.
	Since $a(z)$ is a monic palindromic polynomial of degree $p$,
	both $z^{-p/2} A_+$ and $z^{-p/2} A_-$ are $\dagger$-invariant.
	This implies that $A_+ A_-^{\dagger} = A_- A_+^{\dagger}$.
	Thus $\alpha$ is a T-datum.
	
	Conversely, let $\alpha = (A_+,A_-,[d]) = (N_0-N_+,N_0-N_-,[d])$ be
	any T-datum of size $1$.
	We now identify $1 \times 1$ matrices with their entries. 
	By the condition (\ref{item:N1}), the matrix $N_0$ can be written as $N_0=1+z^p$ for some positive integer $p>0$.
	Then the matrices $N_\pm$ can be written as $N_\varepsilon =  \sum_{q=1}^{p-1} n_q^\varepsilon z^q$
	since the degrees of $N_\pm$ are greater that $0$ and less than $p$ by the condition (\ref{item:N3}).
	We also have $n_q^\varepsilon \in \ZZ_{\geq 0}$ by the condition (\ref{item:N2}).
	By the condition (\ref{item:N4}), these numbers can be written as $n_q^{\pm}=\maxzero{\pm n_q}$, where we define $n_q := n_q^{+}-n_q^{-}$.
	We now show by induction on $q$ that
	$n_{q}^{\pm}=n_{p-q}^{\pm}$ for any $0 \leq q \leq p$,
	where we set $n_0^{\pm}=n_p^{\pm}=0$.
	The case $q=0$ is obvious from the definition.
	Suppose that $q>0$.
	Let $m_q$ be the coefficient of $z^{p-q}$
	in the polynomial $N_0 N_+^{\dagger} +N_+ N_-^{\dagger} + N_- N_0^{\dagger}$, that is,
	\begin{align}
		\label{eq:m_q}
		m_q = n_q^+ + n_{p-q}^- + \sum_{k=0}^q n_{p-q+k}^+ n_k^-.
	\end{align}
	On the other hand, $m_q$ is also the coefficient of $z^{p-q}$ in the polynomial $N_0 N_-^{\dagger} +N_- N_+^{\dagger} + N_+ N_0^{\dagger}$ by the symplectic relation.
	Thus we obtain
	\begin{align}
		\label{eq:m_q 2}
		m_{q} = n_q^- + n_{p-q}^+ + \sum_{k=0}^q n_{p-k}^- n_{q-k}^+.
	\end{align}
	The sum parts in \eqref{eq:m_q} and \eqref{eq:m_q 2} coincide by the induction hypothesis, so we obtain $n_q^+ + n_{p-q}^- = n_q^- + n_{p-q}^+ $.
	Then we conclude from (\ref{item:N4}) that $n_q^{\pm} = n_{p-q}^{\pm}$.
\end{proof}

Let $\alpha$ be a T-datum of size $1$ given in Theorem \ref{theorem:T-datum of size 1}.
Let $\gamma=(B,d,\mathbf{i},\nu)$
be the complete mutation loop given by $\gamma=G(\alpha,\{1\} \times \ZZ)$.
Then the index set $I$ of $B$ is given by
$I=\{ (1,i) \in \{1\} \times \ZZ \mid  0 \leq i < p \}$.
We identify $I$ with the set $\{0,1, \dots , p-1\}$.
Then $\mathbf{i}=\mathbf{i}(0)$ with $\mathbf{i}(0) = (0)$, and $\nu$ is the cyclic permutation given by
$\nu(i) =i+1 \pmod t$.
The exchange matrix $B=(B_{ij})_{i,j \in I}$ can be computed from the formula \eqref{eq:def of B from cap} as follows:
\begin{align}\label{eq:period 1 general solution}
	B_{ij} = -n_{i-j} + n_{j-i}
	+\sum_{k=0}^{\min(i,j)} \bigl( n_{i-k}^+ n_{j-k}^- -n_{i-k}^- n_{j-k}^+ \bigr),
\end{align}
where $n_i^{\pm} := \maxzero{\pm n_i}$ and $n_i:=0$ unless $0<i<p$.

\begin{remark}
	The formula \eqref{eq:period 1 general solution} is precisely the general solution of \emph{period $1$ quivers} given by Fordy and Marsh~\cite[Theorem 6.1]{FordyMarsh}.
	We can regard Theorem \ref{theorem:T-datum of size 1} as another proof of the classification for period 1 quivers, which was also given in~\cite[Theorem 6.1]{FordyMarsh}.
\end{remark}

In Example \ref{example:cap somos4} and \ref{example:somos4 ml} (the Somos-4 recurrence), we give an example of a T-datum of size 1 and a period $1$ quiver.

\subsection{Commuting Cartan matrices}
\label{section:commuting Cartan}
In this section, we give T-data associated with pairs of Cartan matrices, which are generalization of T-data associated with bipartite belts in Example \ref{example:cap b belt} and \ref{example:b belt ml}.
\begin{definition}
	A matrix $C=(c_{ab})_{a,b \in [1,r]} \in \mat_{r \times r}(\ZZ)$
	is called a \emph{symmetrizable weak generalized Cartan matrix}
	if
	\begin{enumerate}
		\item $c_{aa} \leq 2$ for any $a$,
		\item $c_{ab} \geq 0$ for any $a,b$,
		\item there exists a positive integer diagonal matrix $D$ such that $CD$ is a symmetric matrix.
	\end{enumerate}
\end{definition}

The diagonal matrix $D$ is called
a (right) \emph{symmetrizer} of $C$. 
Note that a symmetrizable generalized Cartan matrix is a symmetrizable weak generalized Cartan matrix satisfying $c_{aa}=2$ for any $a \in [1,r]$.

\begin{proposition}\label{prop:double weak Cartan}
	Let $A$ and $A'$ be symmetrizable weak generalized Cartan matrices that have a common symmetrizer $D$.
	Let $N = (n_{ab})_{a,b \in [1,r]} := 2I_r - A$ and $N' =(n'_{ab})_{a,b \in [1,r]} := 2 I_r - A'$.
	Then the triple $\alpha = (A_+,A_-,D)$ defined by
	\begin{align*}
		N_0 = (1+z^2) I_r ,\quad
		N_+ =  z N, \quad
		N_- =  z N'
	\end{align*}
	is a T-datum if and only if $A A'= A' A$ and $n_{ab} n'_{ab}=0$ for any $a,b \in [1,r]$.
\end{proposition}
\begin{proof}
	The conditions (\ref{item:N1})--(\ref{item:N3}) are obvious by the definition.
	We also have $N_0D = D N_0$ and $D^{-1} N_\pm D \in \mat_{r\times r}(\ZZ[z])$ by the definition.
	The condition (\ref{item:N4}) is equivalent to $n_{ab} n'_{ab}=0$ for any $a,b \in [1,r]$.
	Thus it is sufficient to show that
	the symplectic relation is equivalent to $A A' = A' A$.
	Clearly, $A A' = A' A$ if and only if $N N' = N' N$.
	We now have
	\begin{align*}
		&A_+ D A_-^{\dagger} - A_- D A_+^{\dagger}\\
		&= \bigl((1+z^2)I_r - zN \bigr) D \bigl((1+z^{-2})I_r - z^{-1}N'^{\mathsf{T}}\bigr) \\
		&\quad -\bigl((1+z^2)I_r - zN' \bigr) D \bigl((1+z^{-2})I_r - z^{-1}N^{\mathsf{T}}\bigr)\\
		&= (z+z^{-1}) ( -N D - D N'^{\mathsf{T}} + DN^{\mathsf{T}} + N'^{\mathsf{T}}D )
		+N D N'^{\mathsf{T}} - N' D N^{\mathsf{T}} \\
		&=N D N'^{\mathsf{T}} - N' D N^{\mathsf{T}}\\
		&=(N N' - N' N) D,
	\end{align*}
	which completes the proof.
\end{proof}

\begin{proposition}\label{prop:bipartite double weak Cartan}
	Suppose that $\alpha$ given in Proposition \ref{prop:double weak Cartan} is a T-datum.
	Suppose further that there exists a function $\epsilon: [1,r] \to \{1,-1\}$ such that $n_{ab}$ or $n'_{ab} > 0$ implies $\epsilon(a) \neq \epsilon(b)$ for any $a,b \in [1,r]$.
	Then the set
	\begin{align*}
		R_{\epsilon} := \{  (a,u) \in [1,r] \times \ZZ \mid \epsilon(a)=(-1)^{u-1} \}
	\end{align*}
	is consistent for $\alpha$.
\end{proposition}
\begin{proof}
	The conditions (\ref{item:R1}) and (\ref{item:R2}) follow from the assumption on the function $\epsilon$.
	The condition (\ref{item:R3}) is satisfied by setting $t=2$.
\end{proof}

Let $\gamma=(B,d,\mathbf{i},\nu)=G(\alpha,R_{\epsilon})$
be the mutation loop obtained from the data given in Proposition \ref{prop:bipartite double weak Cartan}.
The index set of $B$ is given by
\begin{align*}
	I= \{ (a,0) \mid \epsilon(a)=-1 \} \sqcup \{ (a,1) \mid \epsilon(a)=1 \},
\end{align*}
and we identify it with $[1,r]$ by taking the first components.
Then $B = (B_{ab})_{a,b \in I}$ is given by
\begin{align}\label{eq:bipartite recurrent}
	B_{ab} = 
	\begin{cases}
		-\epsilon(a) n_{ab} &\text{if $n'_{ab}=0$},\\
		+\epsilon(a) n'_{ab} &\text{if $n_{ab}=0$}.
	\end{cases}
\end{align}
The symmetrizer $d$ is given by $d=(d_a)_{a\in I}$,
where $d_a$ is the $a$-th entry in the common symmetrizer $D$.
The sequence $\mathbf{i}$ is given by
$\mathbf{i} = \mathbf{i}(0) \mid \mathbf{i}(1)$ with
$\mathbf{i}(0) = \{ a \in I \mid \epsilon(a)=-1 \}$
and $\mathbf{i}(1) = \{ a \in I \mid \epsilon(a)=1 \}$.
The permutation $\nu$ is the trivial permutation.

\begin{remark}
	If $B$ is skew-symmetric, the corresponding quiver $Q(B)$ is called a \emph{bipartite recurrent quiver}~\cite{GalashinPylyavskyy}.
	Galashin and Pylyavskyy developed the classification theory of bipartite recurrent quivers~\cite{GalashinPylyavskyy,galashin2016quivers,galashin2017quivers}.
	In particular, they gave a complete classification of bipartite recurrent quivers with which the associated T-system is periodic.
\end{remark}

\begin{example}[Tensor product construction]
	\label{example:tensor}
	Let $\bar{A}$ and $\bar{A}'$ be symmetrizable weak generalized Cartan matrices of size $\bar{r}$ and $\bar{r}'$, respectively.
	Suppose that one of them is non-weak.
	Let $D$ and $D'$ be right symmetrizers of $\bar{A}$ and $\bar{A}'$, respectively.
	Let $A = \bar{A} \otimes I_{\bar{r}'}$, $A' = I_{\bar{r}} \otimes \bar{A}'$, and $D= \bar{D} \otimes \bar{D}'$.
	The matrices $A$ and $A'$ are symmetrizable weak generalized Cartan matrices that have the common symmetrizer $D$.
	Then the triple $\alpha = (A_+,A_-,D)$ given in Proposition \ref{prop:double weak Cartan}
	is a T-datum of size $\bar{r}\bar{r}'$ since
	$(\bar{A} \otimes I_{\bar{r}'})(I_{\bar{r}} \otimes \bar{A}') = 
	(I_{\bar{r}} \otimes \bar{A}') (\bar{A} \otimes I_{\bar{r}'})$
	and $n_{ab} \delta_{a'b'} \delta_{ab} n'_{a'b'} = 0$
	for any $a,b \in [1,\bar{r}]$ and $a',b' \in [1,\bar{r}']$.
	
	Suppose further that both $\bar{A}$ and $\bar{A}'$ are bipartite by functions $\bar{\epsilon}$ and $\bar{\epsilon}$, respectively.
	Then the function $\epsilon: [1,\bar{r}] \times [1,\bar{r}'] \to \{ 1, -1 \}$ defined by $\epsilon(a,a') = \bar{\epsilon}(a)\bar{\epsilon}'(a')$ satisfies the assumption in Proposition \ref{prop:bipartite double weak Cartan}.
	Thus we get the consistent subset $R_{\epsilon}$ for $\alpha$.
	For example, let
	\begin{align*}
		A=
		\begin{bmatrix}
			2 & -1 & 0 \\
			-1 & 2 & -1 \\
			0 & -1 & 2
		\end{bmatrix},\quad
		A'=
		\begin{bmatrix}
			2 & -1 \\
			-1 & 2
		\end{bmatrix}
	\end{align*}
	be Cartan matrices of types $A_3$ and $A_2$, respectively.
	Define $\bar{\epsilon}$ and $\bar{\epsilon}'$ by
	$\bar{\epsilon}(1)=\bar{\epsilon}(3) = \bar{\epsilon}'(2) = 1$ and 
	$\bar{\epsilon}(2)= \bar{\epsilon}'(1) = -1$.
	Then the bipartite recurrent quiver $Q(B)$ is given by
	\begin{align*}
		Q(B)=
		\begin{tikzpicture}[scale=2,every node/.style={scale=0.9},baseline=(base.base),vertex/.style={scale=1}]
		\node (base) at (1,1.35) {};
		\node (11) at (1,1)[vertex] {$(11,0)$};
		\node (21) at (2,1)[vertex] {$(21,1)$};
		\node (31) at (3,1)[vertex] {$(31,0)$};
		\node (12) at (1,1.7)[vertex] {$(12,1)$};
		\node (22) at (2,1.7)[vertex] {$(22,0)$};
		\node (32) at (3,1.7)[vertex] {$(32,1)$};
		\foreach \from/\to in {12/11,11/21,21/22,22/12,22/32,32/31,31/21}
		\draw[arrows={-Stealth[scale=1.2]}] (\from)--(\to);
		\end{tikzpicture},
	\end{align*}
	where we denote $(a,a')$ by $aa'$.
\end{example}

\begin{example}[Tadpole type]
	\label{example:tadpole}
	Although Proposition \ref{prop:bipartite double weak Cartan} is for bipartite Cartan matrices,
	non-bipartite Cartan matrices are sometimes interesting.
	Let $A=2 I_r$ and $A'= (2 \delta_{ab} - n'_{ab})_{a,b \in [1,r]}$
	where 
	\begin{align}\label{eq:def of tadpole}
		n'_{ab} =
		\begin{cases}
			1 & \text{if $\lvert a-b \rvert = 1$},\\
			1 & \text{if $a=b=r$},\\
			0 & \text{otherwise}.
		\end{cases}
	\end{align}
	The matrix $A'$ is called the Cartan matrix of the \emph{tadpole type} $T_r$.
	The tadpole type is non-bipartite since \eqref{eq:def of tadpole} has a non-zero entry in the diagonal.
	Let $D=I_r$.
	Then $\alpha$ in Proposition \ref{prop:double weak Cartan}
	is a T-datum.
	For example, $\alpha=(A_+,A_-,I_r)$ for $r=3$ is given by
	\begin{align*}
		A_+ = 
		\begin{bmatrix}
		1+z^2 & 0 & 0\\
		0& 1+z^2 & 0\\
		0 & 0 & 1+z^2
		\end{bmatrix},\quad
		A_- = 
		\begin{bmatrix}
			1+z^2 & -z & 0\\
			-z& 1+z^2 & -z\\
			0 & -z & 1-z+z^2
		\end{bmatrix}.
	\end{align*}
	Let $\gamma = (B,d,\mathbf{i},\nu)$ be the mutation loop
	given by $\gamma = G(\alpha, [1,r]\times \ZZ)$.
	Then the quiver mutation
	$Q(B) \xmapsto{\mu_{\mathbf{i}}} \mu_{\mathbf{i}}(Q(B)) = \nu(Q(B))$
	is given as follows (we set $r=3$ for simplicity) :
	\begin{align*}
		\begin{tikzpicture}[scale=1.7,every node/.style={scale=0.9},baseline=(base.base),vertex/.style={scale=1}]
		\node (base) at (1,0.6) {};
		\node (10) at (1,1)[vertex] {$(1,0)$};
		\node (21) at (2,1)[vertex] {$(2,1)$};
		\node (30) at (3,1)[vertex] {$(3,0)$};
		\node (31) at (3,0.2)[vertex] {$(3,1)$};
		\node (20) at (2,0.2)[vertex] {$(2,0)$};
		\node (11) at (1,0.2)[vertex] {$(1,1)$};
		\foreach \from/\to in {21/10,21/30,31/30,31/20,11/20}
		\draw[arrows={-Stealth[scale=1.2]}] (\from)--(\to);
		\end{tikzpicture} \xmapsto{\mu_{\mathbf{i}}}
		\begin{tikzpicture}[scale=1.7,every node/.style={scale=0.9},baseline=(base.base),vertex/.style={scale=1}]
		\node (base) at (1,0.6) {};
		\node (10) at (1,1)[vertex] {$(1,0)$};
		\node (21) at (2,1)[vertex] {$(2,1)$};
		\node (30) at (3,1)[vertex] {$(3,0)$};
		\node (31) at (3,0.2)[vertex] {$(3,1)$};
		\node (20) at (2,0.2)[vertex] {$(2,0)$};
		\node (11) at (1,0.2)[vertex] {$(1,1)$};
		\foreach \from/\to in {10/21,30/21,30/31,20/31,20/11}
		\draw[arrows={-Stealth[scale=1.2]}] (\from)--(\to);
		\end{tikzpicture},
	\end{align*}
	where $\mathbf{i}=\mathbf{i}(0)=\{(a,0) \mid a \in [1,r]\}$, and $\nu(a,p) = (a,1-p)$ is the permutation that swaps vertices at the same position in the top and bottom rows.
	Intuitively, this mutation loop is explained as follows.
	If we identify the vertices lying in the same $\nu$-orbits and forget the orientation of the quiver,
	we obtain the following graph:
	\begin{align*}
		\begin{tikzpicture}[scale=1.7,every node/.style={scale=0.9},baseline=(1.base),vertex/.style={scale=1}]
		\node (1) at (1,1)[vertex] {$1$};
		\node (2) at (2,1)[vertex] {$2$};
		\node (3) at (3,1)[vertex] {$3$};
		\foreach \from/\to in {1/2,2/3}
		\draw (\from)--(\to);
		\draw (3) to[in=30,out=-30,loop]  (3);
		\end{tikzpicture}.
	\end{align*}
	This is the Dynkin diagram of type $T_r$.
	Therefore, the mutation loop involves the ``folding method'' that constructs $T_r$ diagram from $A_{2r}$ diagram.
	In general,
	one advantage of the strategy of constructing mutation loops from T-data is that it can ``automatically'' perform such a folding method.
\end{example}

\subsection{T-systems associated with quantum affinizations}
\label{section:affinization}
In this section,
we assign T-data to generalized Cartan matrices that satisfy a certain condition, including all finite and affine types.
The T-data in this section are
different from that in Section \ref{section:commuting Cartan} even though both use Cartan matrices.

Fix a positive integer $n$.
Let $C=(c_{ab})_{1 \leq a,b \leq n}$ be a symmetrizable generalized Cartan matrix.
We assume that $C$ is indecomposable.
Let $\diag(c_1, \dots, c_n)$ be a \textit{left} symmetrizer of $C$.
We define integers $t_a$ ($1 \leq a  \leq n$) by
\begin{align*}
	t_a = c_a^{-1} \lcm(c_1, \dots, c_n).
\end{align*}
We also define integers $t_{ab}$ ($1 \leq a,b \leq n$) by
\begin{align*}
	t_{ab} =c_a^{-1} \lcm(c_a,c_b).
\end{align*}
These integers do not depend on the choice of a symmetrizer.
Let $[k]_z \in \ZZ[z^{\pm1}]$ be the $z$-integer defined by
\begin{align*}
	[k]_{z} &= \frac{z^{k} - z^{-k}}{z-z^{-1}} \\
	&= z^{k-1}+z^{k-3}+ \dots + z^{-(k-3)} + z^{-(k-1)}.
\end{align*}
We denote $[k]_{z^{c_a}}$ by $[k]_{z_a}$.

Let $\level$ be an integer with $\level \geq 2$.
Let $H$ be the index set defined by
\begin{align*}
	H= \{ (a,m) \mid 1 \leq a \leq n, 1 \leq m \leq t_a \level -1 \}.
\end{align*}
We often denote an element $(a,m) \in H$ by $am$.
For any $(a,m),(b,k) \in H$,
we define polynomials $\tilde{n}_{am,bk}^0,\tilde{n}_{am,bk}^+,\tilde{n}_{am,bk}^- \in \ZZ[z^{\pm1}]$ by
\begin{align*}
	\tilde{n}_{am,bk}^0 
	&=  [2]_{z_a} \delta_{ab}, \\
	\tilde{n}_{am,bk}^+ 
	&= \delta_{ab} (\delta_{m,k+1}+\delta_{m,k-1}), \\
	\tilde{n}_{am,bk}^- 
	&=
	\begin{cases}
		t_{ab}^{-1} \lvert c_{ab} \rvert [t_{ba} - \lvert p-k \rvert]_{z_b} & \text{if $a\sim b$, $p \in \ZZ$, and $\lvert p-k \rvert <t_{ba}$},\\
		0 &\text{otherwise},
	\end{cases}
\end{align*}
where we write $a \sim b$ if $c_{ab}<0$, and $p=m t_{ab}^{-1} t_{ba}$.
We define two $H \times H$-matrices $\widetilde{A}_+$ and $\widetilde{A}_-$ by
\begin{align*}
	\widetilde{A}_{\pm} = \bigl(\tilde{n}_{am,bk}^0 - \tilde{n}_{am,bk}^{\pm}\bigr)_{am,bk \in H}.
\end{align*}

To illustrate the pair of matrices $(\widetilde{A}_+,\widetilde{A}_-)$, it is useful to consider the graph $\Gamma(\widetilde{A}_+,\widetilde{A}_-)$ defined as follows:
\begin{itemize}
	\item the set of vertices of $\Gamma(\widetilde{A}_+,\widetilde{A}_-)$ is $H$,
	\item for any pair of vertices $(a,m),(b,k) \in H$,
	we draw a \textcolor{cyan}{blue edge} equipped with the pair of polynomials $(f_+,g_+) := (\tilde{n}_{am,bk}^+, \tilde{n}_{bk,am}^+ )$, and a \textcolor{red}{red edge} equipped with the pair of polynomials $(f_-,g_-) := ( \tilde{n}_{am,bk}^-, \tilde{n}_{bk,am}^-)$:
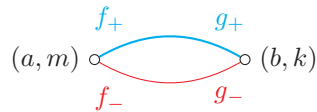
	\begin{align*}
	\begin{tikzpicture}[scale=2,vertex/.style={circle,scale=0.40},red edge/.style={red},blue edge/.style={cyan, thick}]
	\node[vertex] (am) at (1,1) {};
	\node[vertex] (bk) at (2,1) {};
	\draw[blue edge] (am) to[bend left=30] node[pos=0.25,auto=left] {$f_+$} node[pos=0.75,auto=left] {$g_+$}  (bk) ;
	\draw[red edge] (am) to[bend left=-30] node[pos=0.25,auto=right] {$f_-$} node[pos=0.75,auto=right] {$g_-$} (bk) ;
	\node[vertex,draw,fill=white][label=left:${(a,m)}$] at (1,1) {};
	\node[vertex,draw,fill=white][label=right:${(b,k)}$] at (2,1) {};
	\end{tikzpicture}.
	\end{align*}
\end{itemize}

For a red edge
\begin{align*}
\begin{tikzpicture}[scale=2,vertex/.style={circle,scale=0.40},red edge/.style={red},blue edge/.style={cyan, thick}]
\node[vertex] (am) at (1,1) {};
\node[vertex] (bk) at (2,1) {};
\draw[red edge] (am) to node[pos=0.25,auto=left] {$f_-$} node[pos=0.75,auto=left] {$g_-$}  (bk) ;
\node[vertex,draw,fill=white][label=left:${(a,m)}$]  at (1,1) {};
\node[vertex,draw,fill=white][label=right:${(b,k)}$]  at (2,1) {};
\end{tikzpicture},
\end{align*}
we use the following abbreviations:
\begin{align*}
&\begin{tikzpicture}[scale=1.5, baseline=(am.south),vertex/.style={circle,scale=0.40},red edge/.style={red},dashed red edge/.style={red, dash pattern=on 2pt off 1pt},blue edge/.style={cyan, thick}]
\node[vertex] (am) at (1,1/1) {};
\node[vertex] (bk) at (2,1/1) {};
\node[vertex,draw,fill=white] at (1,1/1) {};
\node[vertex,draw,fill=white] at (2,1/1) {};
\end{tikzpicture}
\quad \text{if $(f_-,g_-)=(0,0)$,} \\
&\begin{tikzpicture}[scale=1.5, baseline=(am.south),vertex/.style={circle,scale=0.40},red edge/.style={red},dashed red edge/.style={red, dash pattern=on 2pt off 1pt},blue edge/.style={cyan, thick}]
\node[vertex] (am) at (1,1/1) {};
\node[vertex] (bk) at (2,1/1) {};
\draw[red edge]  (am) to   (bk);
\node[vertex,draw,fill=white] at (1,1/1) {};
\node[vertex,draw,fill=white] at (2,1/1) {};
\end{tikzpicture}
\quad \text{if $(f_-,g_-)=(1,1)$,} \\
&\begin{tikzpicture}[scale=1.5, baseline=(am.south),vertex/.style={circle,scale=0.40},red edge/.style={red},blue edge/.style={cyan, thick}]
\node[vertex] (am) at (1,1) {};
\node[vertex]  (bk) at (2,1) {};
\draw[red edge]  (am.south) to  (bk.south);
\draw[red edge,draw opacity=0,text opacity=1]  (am) to [edge node={node {\scalebox{1.2}{$>$}}},sloped] (bk);
\draw[red edge]  (am.north) to  (bk.north);
\node[vertex,draw,fill=white] at (1,1/1) {};
\node[vertex,draw,fill=white] at (2,1/1) {};
\end{tikzpicture}
\quad \text{if $(f_-,g_-)=(1,[2]_{z_a})$,} \\
&\begin{tikzpicture}[scale=1.5, baseline=(am.south),vertex/.style={circle,scale=0.40},red edge/.style={red},blue edge/.style={cyan, thick}]
\node[vertex] (am) at (1,1) {};
\node[vertex]  (bk) at (2,1) {};
\draw[red edge]  (am.south) to  (bk.south);
\draw[red edge]  (am) to [edge node={node {\scalebox{1.2}{$>$}}},sloped] (bk);
\draw[red edge]  (am.north) to  (bk.north);
\node[vertex,draw,fill=white] at (1,1/1) {};
\node[vertex,draw,fill=white] at (2,1/1) {};
\end{tikzpicture}
\quad \text{if $(f_-,g_-)=(1,[3]_{z_a})$,}\\
&\begin{tikzpicture}[scale=1.5, baseline=(am.south),vertex/.style={circle,scale=0.40},dashed red edge/.style={red, dash pattern=on 2pt off 1pt},blue edge/.style={cyan, thick}]
\node[vertex] (am) at (1,1) {};
\node[vertex]  (bk) at (2,1) {};
\draw[dashed red edge]  (am) to [edge node={node {\scalebox{1.2}{$>$}}},sloped] (bk);
\node[vertex,draw,fill=white] at (1,1/1) {};
\node[vertex,draw,fill=white] at (2,1/1) {};
\end{tikzpicture}
\quad \text{if $(f_-,g_-)=(0,1)$,} \\
&\begin{tikzpicture}[scale=1.5, baseline=(am.south),vertex/.style={circle,scale=0.40},dashed red edge/.style={red, dash pattern=on 2pt off 1pt},blue edge/.style={cyan, thick}]
\node[vertex] (am) at (1,1) {};
\node[vertex]  (bk) at (2,1) {};
\draw[dashed red edge]  (am.south) to  (bk.south);
\draw[dashed red edge,draw opacity=0,text opacity=1]  (am) to [edge node={node {\scalebox{1.2}{$>$}}},sloped] (bk);
\draw[dashed red edge]  (am.north) to  (bk.north);
\node[vertex,draw,fill=white] at (1,1/1) {};
\node[vertex,draw,fill=white] at (2,1/1) {};
\end{tikzpicture}
\quad \text{if $(f_-,g_-)=(0,[2]_{z_a})$,} \\
&\begin{tikzpicture}[scale=1.5, baseline=(am.south),vertex/.style={circle,scale=0.40},dashed red edge/.style={red, dash pattern=on 2pt off 1pt},blue edge/.style={cyan, thick}]
\node[vertex] (am) at (1,1) {};
\node[vertex]  (bk) at (2,1) {};
\draw[dashed red edge]  (am.south) to  (bk.south);
\draw[dashed red edge]  (am) to [edge node={node {\scalebox{1.2}{$>$}}},sloped] (bk);
\draw[dashed red edge]  (am.north) to  (bk.north);
\node[vertex,draw,fill=white] at (1,1/1) {};
\node[vertex,draw,fill=white] at (2,1/1) {};
\end{tikzpicture}
\quad \text{if $(f_-,g_-)=(0,[3]_{z_a})$.} 
\end{align*}
We may use the same abbreviations for blue edges, but these are not needed here.
When we use these abbreviations, we keep in mind the symmetrizer.

\begin{example}\label{example:F4}
Consider a Cartan matrix of type $F_4$:
\begin{align*}
C=
\left[
\begin{array}{rrrr}
2 & -1 & 0 & 0 \\
-1 &2 & -1 & 0 \\
0 & -2 & 2 & -1 \\
0 & 0 & -1 & 2
\end{array}
\right],
\end{align*}
and choose a symmetrizer as $\diag(2,2,1,1)$.
When $\level=2$,
the index set $H$ is given by
\begin{align*}
H=\{ (1, 1), (2, 1), (3, 1), (3, 2), (3, 3), (4, 1), (4, 2), (4, 3) \} ,
\end{align*}
and the matrices $\widetilde{A}_+$ and $\widetilde{A}_-$ are given by
\begin{align*}
\widetilde{A}_+ &=
\begin{bmatrix*}
[2]_{z^2} & 0 & 0 & 0 & 0 & 0 & 0 & 0 \\
0 & [2]_{z^2} & 0 & 0 & 0 & 0 & 0 & 0 \\
0 & 0 & [2]_{z} & -1 & 0 & 0 & 0 & 0 \\
0 & 0 & -1 & [2]_{z} & -1 & 0 & 0 & 0 \\
0 & 0 & 0 & -1 & [2]_{z} & 0 & 0 & 0 \\
0 & 0 & 0 & 0 & 0 & [2]_{z} & -1 & 0 \\
0 & 0 & 0 & 0 & 0 & -1 & [2]_{z} & -1 \\
0 & 0 & 0 & 0 & 0 & 0 & -1 & [2]_{z}
\end{bmatrix*},\\
\widetilde{A}_- &=
\begin{bmatrix*}
[2]_{z^2} & -1 & 0 & 0 & 0 & 0 & 0 & 0 \\
-1 & [2]_{z^2} & -1 & -[2]_{z} & -1 & 0 & 0 & 0 \\
0 & 0 & [2]_{z} & 0 & 0 & -1 & 0 & 0 \\
0 & -1 & 0 & [2]_{z} & 0 & 0 & -1 & 0 \\
0 & 0 & 0 & 0 & [2]_{z} & 0 & 0 & -1 \\
0 & 0 & -1 & 0 & 0 & [2]_{z} & 0 & 0 \\
0 & 0 & 0 & -1 & 0 & 0 & [2]_{z} & 0 \\
0 & 0 & 0 & 0 & -1 & 0 & 0 & [2]_{z}
\end{bmatrix*}.
\end{align*}
The diagram $\Gamma(\widetilde{A}_+,\widetilde{A}_-)$ is given by
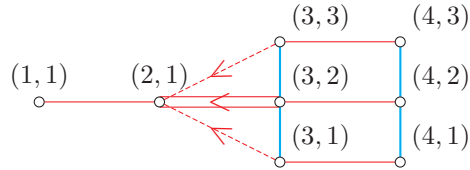
\begin{align*}
\begin{tikzpicture}[scale=1.6,vertex/.style={circle,scale=0.40},red edge/.style={red},dashed red edge/.style={red, dash pattern=on 2pt off 1pt},blue edge/.style={cyan, thick},label distance=-0.35mm]
\node[vertex][label=above:${(1,1)}$] (01) at (1,1/1) {};
\node[vertex][label=above:${(2,1)}$] (11) at (2,1/1) {};
\node[vertex][label=above right:${(3,1)}$] (21) at (3,1/2) {};
\node[vertex][label=above right:${(3,2)}$] (22) at (3,2/2) {};
\node[vertex][label=above right:${(3,3)}$] (23) at (3,3/2) {};
\node[vertex][label=above right:${(4,1)}$] (31) at (4,1/2) {};
\node[vertex][label=above right:${(4,2)}$] (32) at (4,2/2) {};
\node[vertex][label=above right:${(4,3)}$] (33) at (4,3/2) {};
\draw[red edge]  (01) to   (11);
\draw[dashed red edge]  (11) to [edge node={node {\scalebox{1.2}{$<$}}},sloped] (21);
\draw[red edge]  (11.south) to  (22.south);
\draw[red edge,draw opacity=0,text opacity=1]  (11) to [edge node={node {\scalebox{1.2}{$<$}}},sloped] (22);
\draw[red edge]  (11.north) to   (22.north);
\draw[dashed red edge]  (11) to [edge node={node {\scalebox{1.2}{$<$}}},sloped] (23);
\draw[red edge]  (21) to   (31);
\draw[red edge]  (22) to   (32);
\draw[red edge]  (23) to   (33);
\draw[blue edge] (21) to (22);
\draw[blue edge] (22) to (23);
\draw[blue edge] (31) to (32);
\draw[blue edge] (32) to (33);
\node[vertex,draw] at (1,1/1) {};
\node[vertex,draw] at (2,1/1) {};
\node[vertex,draw] at (3,1/2) {};
\node[vertex,draw] at (3,2/2) {};
\node[vertex,draw] at (3,3/2) {};
\node[vertex,draw] at (4,1/2) {};
\node[vertex,draw] at (4,2/2) {};
\node[vertex,draw] at (4,3/2) {};
\end{tikzpicture}\quad .
\end{align*}
\end{example}

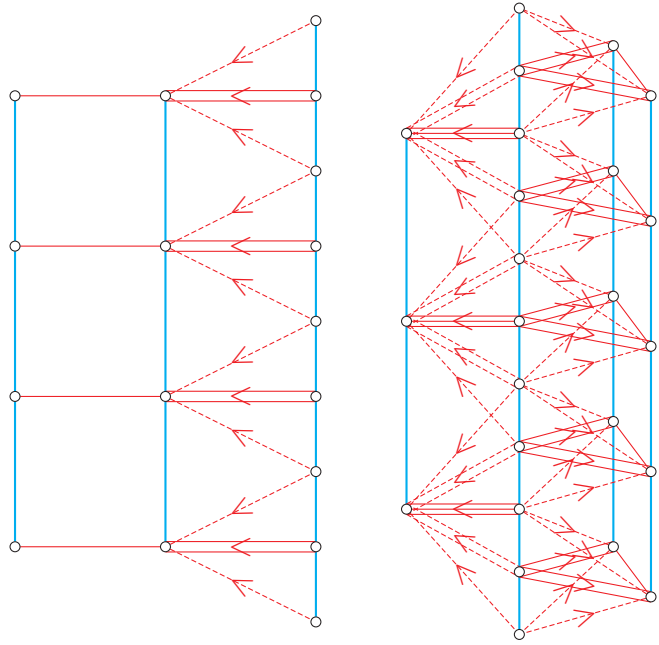
\begin{figure}
\begin{align*}
\begin{tikzpicture}[scale=2,baseline=(25.base),vertex/.style={circle,scale=0.40},red edge/.style={red},dashed red edge/.style={red, dash pattern=on 2pt off 1pt},blue edge/.style={cyan, thick}]
\node[vertex] (01) at (1,1/1) {};
\node[vertex] (02) at (1,2/1) {};
\node[vertex] (03) at (1,3/1) {};
\node[vertex] (04) at (1,4/1) {};
\node[vertex] (11) at (2,1/1) {};
\node[vertex] (12) at (2,2/1) {};
\node[vertex] (13) at (2,3/1) {};
\node[vertex] (14) at (2,4/1) {};
\node[vertex] (21) at (3,1/2) {};
\node[vertex] (22) at (3,2/2) {};
\node[vertex] (23) at (3,3/2) {};
\node[vertex] (24) at (3,4/2) {};
\node[vertex] (25) at (3,5/2) {};
\node[vertex] (26) at (3,6/2) {};
\node[vertex] (27) at (3,7/2) {};
\node[vertex] (28) at (3,8/2) {};
\node[vertex] (29) at (3,9/2) {};
\draw[blue edge] (01) to (02);
\draw[blue edge] (02) to (03);
\draw[blue edge] (03) to (04);
\draw[blue edge] (11) to (12);
\draw[blue edge] (12) to (13);
\draw[blue edge] (13) to (14);
\draw[blue edge] (21) to (22);
\draw[blue edge] (22) to (23);
\draw[blue edge] (23) to (24);
\draw[blue edge] (24) to (25);
\draw[blue edge] (25) to (26);
\draw[blue edge] (26) to (27);
\draw[blue edge] (27) to (28);
\draw[blue edge] (28) to (29);
\draw[red edge]  (01) to   (11);
\draw[red edge]  (02) to   (12);
\draw[red edge]  (03) to   (13);
\draw[red edge]  (04) to   (14);
\draw[dashed red edge]  (11) to [edge node={node {\scalebox{1.2}{$<$}}},sloped] (21);
\draw[red edge]  (11.south) to  (22.south);
\draw[red edge,draw opacity=0,text opacity=1]  (11) to [edge node={node {\scalebox{1.2}{$<$}}},sloped] (22);
\draw[red edge]  (11.north) to   (22.north);
\draw[dashed red edge]  (11) to [edge node={node {\scalebox{1.2}{$<$}}},sloped] (23);
\draw[dashed red edge]  (12) to [edge node={node {\scalebox{1.2}{$<$}}},sloped] (23);
\draw[red edge]  (12.south) to  (24.south);
\draw[red edge,draw opacity=0,text opacity=1]  (12) to [edge node={node {\scalebox{1.2}{$<$}}},sloped] (24);
\draw[red edge]  (12.north) to   (24.north);
\draw[dashed red edge]  (12) to [edge node={node {\scalebox{1.2}{$<$}}},sloped] (25);
\draw[dashed red edge]  (13) to [edge node={node {\scalebox{1.2}{$<$}}},sloped] (25);
\draw[red edge]  (13.south) to  (26.south);
\draw[red edge,draw opacity=0,text opacity=1]  (13) to [edge node={node {\scalebox{1.2}{$<$}}},sloped] (26);
\draw[red edge]  (13.north) to   (26.north);
\draw[dashed red edge]  (13) to [edge node={node {\scalebox{1.2}{$<$}}},sloped] (27);
\draw[dashed red edge]  (14) to [edge node={node {\scalebox{1.2}{$<$}}},sloped] (27);
\draw[red edge]  (14.south) to  (28.south);
\draw[red edge,draw opacity=0,text opacity=1]  (14) to [edge node={node {\scalebox{1.2}{$<$}}},sloped] (28);
\draw[red edge]  (14.north) to   (28.north);
\draw[dashed red edge]  (14) to [edge node={node {\scalebox{1.2}{$<$}}},sloped] (29);
\node[vertex,draw,fill=white] at (1,1/1) {};
\node[vertex,draw,fill=white] at (1,2/1) {};
\node[vertex,draw,fill=white] at (1,3/1) {};
\node[vertex,draw,fill=white] at (1,4/1) {};
\node[vertex,draw,fill=white] at (2,1/1) {};
\node[vertex,draw,fill=white] at (2,2/1) {};
\node[vertex,draw,fill=white] at (2,3/1) {};
\node[vertex,draw,fill=white] at (2,4/1) {};
\node[vertex,draw,fill=white] at (3,1/2) {};
\node[vertex,draw,fill=white] at (3,2/2) {};
\node[vertex,draw,fill=white] at (3,3/2) {};
\node[vertex,draw,fill=white] at (3,4/2) {};
\node[vertex,draw,fill=white] at (3,5/2) {};
\node[vertex,draw,fill=white] at (3,6/2) {};
\node[vertex,draw,fill=white] at (3,7/2) {};
\node[vertex,draw,fill=white] at (3,8/2) {};
\node[vertex,draw,fill=white] at (3,9/2) {};
\end{tikzpicture}
\quad\quad\quad
\begin{tikzpicture}[scale=5,baseline=(02.base),vertex/.style={circle,scale=0.40},red edge/.style={red},dashed red edge/.style={red, dash pattern=on 2pt off 1pt},blue edge/.style={cyan,thick}]
\node[vertex] (01) at (1.70000000000000,1/2) {};
\node[vertex] (02) at (1.70000000000000,2/2) {};
\node[vertex] (03) at (1.70000000000000,3/2) {};
\node[vertex] (11) at (2,1/6) {};
\node[vertex] (12) at (2,2/6) {};
\node[vertex] (13) at (2,3/6) {};
\node[vertex] (14) at (2,4/6) {};
\node[vertex] (15) at (2,5/6) {};
\node[vertex] (16) at (2,6/6) {};
\node[vertex] (17) at (2,7/6) {};
\node[vertex] (18) at (2,8/6) {};
\node[vertex] (19) at (2,9/6) {};
\node[vertex] (110) at (2,10/6) {};
\node[vertex] (111) at (2,11/6) {};
\node[vertex] (21) at (2.35000000000000,0.800000000000000/3) {};
\node[vertex] (22) at (2.35000000000000,1.80000000000000/3) {};
\node[vertex] (23) at (2.35000000000000,2.80000000000000/3) {};
\node[vertex] (24) at (2.35000000000000,3.80000000000000/3) {};
\node[vertex] (25) at (2.35000000000000,4.80000000000000/3) {};
\node[vertex] (31) at (2.25000000000000,1.20000000000000/3) {};
\node[vertex] (32) at (2.25000000000000,2.20000000000000/3) {};
\node[vertex] (33) at (2.25000000000000,3.20000000000000/3) {};
\node[vertex] (34) at (2.25000000000000,4.20000000000000/3) {};
\node[vertex] (35) at (2.25000000000000,5.20000000000000/3) {};
\draw[blue edge] (01) to (02);
\draw[blue edge] (02) to (03);
\draw[blue edge] (11) to (12);
\draw[blue edge] (12) to (13);
\draw[blue edge] (13) to (14);
\draw[blue edge] (14) to (15);
\draw[blue edge] (15) to (16);
\draw[blue edge] (16) to (17);
\draw[blue edge] (17) to (18);
\draw[blue edge] (18) to (19);
\draw[blue edge] (19) to (110);
\draw[blue edge] (110) to (111);
\draw[blue edge] (21) to (22);
\draw[blue edge] (22) to (23);
\draw[blue edge] (23) to (24);
\draw[blue edge] (24) to (25);
\draw[blue edge] (31) to (32);
\draw[blue edge] (32) to (33);
\draw[blue edge] (33) to (34);
\draw[blue edge] (34) to (35);
\draw[dashed red edge]  (01) to [edge node={node {\scalebox{1.2}{$<$}}},sloped] (11);
\draw[dashed red edge]  (01.south) to  (12.south);
\draw[dashed red edge,draw opacity=0,text opacity=1]  (01) to [edge node={node {\scalebox{1.2}{$<$}}},sloped] (12);
\draw[dashed red edge]  (01.north) to   (12.north);
\draw[red edge]  (01.south) to  (13.south);
\draw[red edge]  (01) to [edge node={node {\scalebox{1.2}{$<$}}},sloped] (13);
\draw[red edge]  (01.north) to   (13.north);
\draw[dashed red edge]  (01.south) to  (14.south);
\draw[dashed red edge,draw opacity=0,text opacity=1]  (01) to [edge node={node {\scalebox{1.2}{$<$}}},sloped] (14);
\draw[dashed red edge]  (01.north) to   (14.north);
\draw[dashed red edge]  (01) to [edge node={node {\scalebox{1.2}{$<$}}},sloped] (15);
\draw[dashed red edge]  (02) to [edge node={node {\scalebox{1.2}{$<$}}},sloped] (14);
\draw[dashed red edge]  (02.south) to  (15.south);
\draw[dashed red edge,draw opacity=0,text opacity=1]  (02) to [edge node={node {\scalebox{1.2}{$<$}}},sloped] (15);
\draw[dashed red edge]  (02.north) to   (15.north);
\draw[red edge]  (02.south) to  (16.south);
\draw[red edge]  (02) to [edge node={node {\scalebox{1.2}{$<$}}},sloped] (16);
\draw[red edge]  (02.north) to   (16.north);
\draw[dashed red edge]  (02.south) to  (17.south);
\draw[dashed red edge,draw opacity=0,text opacity=1]  (02) to [edge node={node {\scalebox{1.2}{$<$}}},sloped] (17);
\draw[dashed red edge]  (02.north) to   (17.north);
\draw[dashed red edge]  (02) to [edge node={node {\scalebox{1.2}{$<$}}},sloped] (18);
\draw[dashed red edge]  (03) to [edge node={node {\scalebox{1.2}{$<$}}},sloped] (17);
\draw[dashed red edge]  (03.south) to  (18.south);
\draw[dashed red edge,draw opacity=0,text opacity=1]  (03) to [edge node={node {\scalebox{1.2}{$<$}}},sloped] (18);
\draw[dashed red edge]  (03.north) to   (18.north);
\draw[red edge]  (03.south) to  (19.south);
\draw[red edge]  (03) to [edge node={node {\scalebox{1.2}{$<$}}},sloped] (19);
\draw[red edge]  (03.north) to   (19.north);
\draw[dashed red edge]  (03.south) to  (110.south);
\draw[dashed red edge,draw opacity=0,text opacity=1]  (03) to [edge node={node {\scalebox{1.2}{$<$}}},sloped] (110);
\draw[dashed red edge]  (03.north) to   (110.north);
\draw[dashed red edge]  (03) to [edge node={node {\scalebox{1.2}{$<$}}},sloped] (111);
\draw[dashed red edge]  (21) to [edge node={node {\scalebox{1.2}{$>$}}},sloped] (11);
\draw[red edge]  (21.south) to  (12.south);
\draw[red edge,draw opacity=0,text opacity=1]  (21) to [edge node={node {\scalebox{1.2}{$>$}}},sloped] (12);
\draw[red edge]  (21.north) to   (12.north);
\draw[dashed red edge]  (21) to [edge node={node {\scalebox{1.2}{$>$}}},sloped] (13);
\draw[dashed red edge]  (22) to [edge node={node {\scalebox{1.2}{$>$}}},sloped] (13);
\draw[red edge]  (22.south) to  (14.south);
\draw[red edge,draw opacity=0,text opacity=1]  (22) to [edge node={node {\scalebox{1.2}{$>$}}},sloped] (14);
\draw[red edge]  (22.north) to   (14.north);
\draw[dashed red edge]  (22) to [edge node={node {\scalebox{1.2}{$>$}}},sloped] (15);
\draw[dashed red edge]  (23) to [edge node={node {\scalebox{1.2}{$>$}}},sloped] (15);
\draw[red edge]  (23.south) to  (16.south);
\draw[red edge,draw opacity=0,text opacity=1]  (23) to [edge node={node {\scalebox{1.2}{$>$}}},sloped] (16);
\draw[red edge]  (23.north) to   (16.north);
\draw[dashed red edge]  (23) to [edge node={node {\scalebox{1.2}{$>$}}},sloped] (17);
\draw[dashed red edge]  (24) to [edge node={node {\scalebox{1.2}{$>$}}},sloped] (17);
\draw[red edge]  (24.south) to  (18.south);
\draw[red edge,draw opacity=0,text opacity=1]  (24) to [edge node={node {\scalebox{1.2}{$>$}}},sloped] (18);
\draw[red edge]  (24.north) to   (18.north);
\draw[dashed red edge]  (24) to [edge node={node {\scalebox{1.2}{$>$}}},sloped] (19);
\draw[dashed red edge]  (25) to [edge node={node {\scalebox{1.2}{$>$}}},sloped] (19);
\draw[red edge]  (25.south) to  (110.south);
\draw[red edge,draw opacity=0,text opacity=1]  (25) to [edge node={node {\scalebox{1.2}{$>$}}},sloped] (110);
\draw[red edge]  (25.north) to   (110.north);
\draw[dashed red edge]  (25) to [edge node={node {\scalebox{1.2}{$>$}}},sloped] (111);
\draw[dashed red edge]  (31) to [edge node={node {\scalebox{1.2}{$>$}}},sloped] (11);
\draw[red edge]  (31.south) to  (12.south);
\draw[red edge,draw opacity=0,text opacity=1]  (31) to [edge node={node {\scalebox{1.2}{$>$}}},sloped] (12);
\draw[red edge]  (31.north) to   (12.north);
\draw[dashed red edge]  (31) to [edge node={node {\scalebox{1.2}{$>$}}},sloped] (13);
\draw[red edge]  (31) to   (21);
\draw[dashed red edge]  (32) to [edge node={node {\scalebox{1.2}{$>$}}},sloped] (13);
\draw[red edge]  (32.south) to  (14.south);
\draw[red edge,draw opacity=0,text opacity=1]  (32) to [edge node={node {\scalebox{1.2}{$>$}}},sloped] (14);
\draw[red edge]  (32.north) to   (14.north);
\draw[dashed red edge]  (32) to [edge node={node {\scalebox{1.2}{$>$}}},sloped] (15);
\draw[red edge]  (32) to   (22);
\draw[dashed red edge]  (33) to [edge node={node {\scalebox{1.2}{$>$}}},sloped] (15);
\draw[red edge]  (33.south) to  (16.south);
\draw[red edge,draw opacity=0,text opacity=1]  (33) to [edge node={node {\scalebox{1.2}{$>$}}},sloped] (16);
\draw[red edge]  (33.north) to   (16.north);
\draw[dashed red edge]  (33) to [edge node={node {\scalebox{1.2}{$>$}}},sloped] (17);
\draw[red edge]  (33) to   (23);
\draw[dashed red edge]  (34) to [edge node={node {\scalebox{1.2}{$>$}}},sloped] (17);
\draw[red edge]  (34.south) to  (18.south);
\draw[red edge,draw opacity=0,text opacity=1]  (34) to [edge node={node {\scalebox{1.2}{$>$}}},sloped] (18);
\draw[red edge]  (34.north) to   (18.north);
\draw[dashed red edge]  (34) to [edge node={node {\scalebox{1.2}{$>$}}},sloped] (19);
\draw[red edge]  (34) to   (24);
\draw[dashed red edge]  (35) to [edge node={node {\scalebox{1.2}{$>$}}},sloped] (19);
\draw[red edge]  (35.south) to  (110.south);
\draw[red edge,draw opacity=0,text opacity=1]  (35) to [edge node={node {\scalebox{1.2}{$>$}}},sloped] (110);
\draw[red edge]  (35.north) to   (110.north);
\draw[dashed red edge]  (35) to [edge node={node {\scalebox{1.2}{$>$}}},sloped] (111);
\draw[red edge]  (35) to   (25);
\node[vertex,draw,fill=white] at (1.70000000000000,1/2) {};
\node[vertex,draw,fill=white] at (1.70000000000000,2/2) {};
\node[vertex,draw,fill=white] at (1.70000000000000,3/2) {};
\node[vertex,draw,fill=white] at (2,1/6) {};
\node[vertex,draw,fill=white] at (2,2/6) {};
\node[vertex,draw,fill=white] at (2,3/6) {};
\node[vertex,draw,fill=white] at (2,4/6) {};
\node[vertex,draw,fill=white] at (2,5/6) {};
\node[vertex,draw,fill=white] at (2,6/6) {};
\node[vertex,draw,fill=white] at (2,7/6) {};
\node[vertex,draw,fill=white] at (2,8/6) {};
\node[vertex,draw,fill=white] at (2,9/6) {};
\node[vertex,draw,fill=white] at (2,10/6) {};
\node[vertex,draw,fill=white] at (2,11/6) {};
\node[vertex,draw,fill=white] at (2.35000000000000,0.800000000000000/3) {};
\node[vertex,draw,fill=white] at (2.35000000000000,1.80000000000000/3) {};
\node[vertex,draw,fill=white] at (2.35000000000000,2.80000000000000/3) {};
\node[vertex,draw,fill=white] at (2.35000000000000,3.80000000000000/3) {};
\node[vertex,draw,fill=white] at (2.35000000000000,4.80000000000000/3) {};
\node[vertex,draw,fill=white] at (2.25000000000000,1.20000000000000/3) {};
\node[vertex,draw,fill=white] at (2.25000000000000,2.20000000000000/3) {};
\node[vertex,draw,fill=white] at (2.25000000000000,3.20000000000000/3) {};
\node[vertex,draw,fill=white] at (2.25000000000000,4.20000000000000/3) {};
\node[vertex,draw,fill=white] at (2.25000000000000,5.20000000000000/3) {};
\end{tikzpicture}
\end{align*}
\caption{Examples of the diagram $\Gamma(\widetilde{A}_+,\widetilde{A}_-)$.}
\label{fig: tamely cap examples}
\end{figure}

More complicated examples are given in Figure \ref{fig: tamely cap examples}.
The left diagram is of
\begin{align*}
C=
\begin{bmatrix*}[r]
2 & -1 & 0 \\
-1 & 2 & -1 \\
0 & -2 & 2
\end{bmatrix*}
\end{align*}
(the Cartan matrix of type $B_3$) and $\level = 5$,
and the right diagram is of
\begin{align*}
C=
\begin{bmatrix*}[r]
2 & -1 & 0 & 0 \\
-3 & 2 & -2 & -2 \\
0 & -1 & 2 & -1 \\
0 & -1 & -1 & 2
\end{bmatrix*}
\end{align*}
and $\level = 2$.

\begin{proposition}\label{prop:contains Dynkin}
	The red part of the diagram $\Gamma(\widetilde{A}_+,\widetilde{A}_-)$ contains the Dynkin diagram of the transpose of $C$.
	More precisely,
	we have
	\begin{align*}
		(\widetilde{A}_-|_{z=1})_{am,bk} = c_{ba}
	\end{align*}
	if $t_{ba} m = t_{ab} k$.
\end{proposition}
\begin{proof}
	Suppose that $(a,m),(b,k) \in H$ satisfy
	$t_{ba} m = t_{ab} k$.
	Note that such pairs exist for any $a,b$ since
	$m:=t_{ab} \leq t_a \leq t_a \level -1 $ and
	$k:=t_{ba} \leq t_b \leq t_b \level -1 $
	satisfy the condition.
	If $a=b$, we have
	\begin{align*}
		(\widetilde{A}_-|_{z=1})_{am,bk}
		=(\tilde{n}_{am,am}^0)|_{z=1}
		= 2,
	\end{align*}
	and this is equal to $c_{aa}$.
	If $a \neq b$, we have
	\begin{align*}
		(\widetilde{A}_-|_{z=1})_{am,bk}
		=-(\tilde{n}_{am,bk}^- )|_{z=1}
		= t_{ab}^{-1}  c_{ab}  t_{ba}
		= c_a c_{ab} c_b^{-1}
		= c_{ba}.
	\end{align*}
\end{proof}

\begin{lemma}\label{lemma:symmetric cartan level}
	The matrix $\widetilde{A}_+ \widetilde{A}_-^{\mathsf{T}}$ is symmetric.
\end{lemma}
\begin{proof}	
	The matrix $\widetilde{A}_+ \widetilde{A}_-^{\mathsf{T}}$ is symmetric if and only if
	\begin{equation}\label{eq: nnn-nnn=0}
	\begin{split}
		\sum_{(c,l) \in H} (
		&\tilde{n}_{am,cl}^0\tilde{n}_{bk,cl}^+
		+\tilde{n}_{am,cl}^+\tilde{n}_{bk,cl}^-
		+\tilde{n}_{am,cl}^-\tilde{n}_{bk,cl}^0 \\
		&-\tilde{n}_{am,cl}^0\tilde{n}_{bk,cl}^-
		-\tilde{n}_{am,cl}^-\tilde{n}_{bk,cl}^+
		-\tilde{n}_{am,cl}^+\tilde{n}_{bk,cl}^0)=0
	\end{split}
	\end{equation}
	for any $(a,m),(b,k) \in H$.
	Let $X$ be the left-hand side in \eqref{eq: nnn-nnn=0}.
	Then $X=0$ is trivial except for the following cases:
	\begin{enumerate}
		\item[(i)] $a\sim b$ and $t_{ba}m = t_{ab} k$,
		\item[(ii)] $a \sim b$, $p=mt_{ab}^{-1}t_{ba} \in \ZZ$, and $0< \vert p-k \rvert < t_{ba}$,
		\item[(ii')] $a \sim b$, $p'=kt_{ba}^{-1}t_{ab} \in \ZZ$, and $0< \vert p'-m \rvert < t_{ab}$,
		\item[(iii)] $a \sim b$, $p=mt_{ab}^{-1}t_{ba} \in \ZZ$, and $\vert p-k \rvert = t_{ba}$,
		\item[(iii')] $a \sim b$, $p'=kt_{ba}^{-1}t_{ab} \in \ZZ$, and $\vert p'-m \rvert = t_{ab}$.
	\end{enumerate}
	Moreover, the cases (ii') and (iii') reduce to the cases (ii) and (iii), respectively, since the left-hand side in \eqref{eq: nnn-nnn=0} is skew-symmetric under $am \leftrightarrow bk$.
	For the case (i), we have
	\begin{align*}
	X &=
	2 t_{ba}^{-1} \lvert c_{ba} \rvert [t_{ab}-1]_{z_a}
	+t_{ab}^{-1} \lvert c_{ab} \rvert [t_{ba}]_{z_b} \cdot [2]_{z_b}\\
	&\quad 
	- [2]_{z_a} \cdot t_{ba}^{-1} \lvert c_{ba} \rvert [t_{ab}]_{z_a}
	-2 t_{ab}^{-1} \lvert c_{ab} \rvert [t_{ba}-1]_{z_b} \\
	&=  t_{ab}^{-1} \lvert c_{ab} \rvert ( [t_{ba}+1]_{z_b} - [t_{ba}-1]_{z_b})
	- t_{ba}^{-1} \lvert c_{ba} \rvert ( [t_{ab}+1]_{z_a} - [t_{ab}-1]_{z_a}) \\
	&=  t_{ab}^{-1} \lvert c_{ab} \rvert (z^{t_{ba}c_b} + z^{-t_{ba}c_b})
	- t_{ba}^{-1} \lvert c_{ba} \rvert (z^{t_{ab}c_a} + z^{-t_{ab}c_a}) \\
	&=0.
	\end{align*}
	Here, we use $[n]_{z} \cdot [2]_{z} = [n+1]_{z} + [n-1]_{z}$
	to derive the second equality.
	For the case (ii), we have
	\begin{align*}
	X &=
	t_{ab}^{-1} \lvert c_{ab} \rvert [t_{ba}- \lvert p-k \rvert ]_{z_b} \cdot [2]_{z_b}\\
	&\quad - t_{ab}^{-1} \lvert c_{ab} \rvert ([t_{ba}-\lvert p-k \rvert -1]_{z_b} + [t_{ba}-\lvert p-k \rvert +1]_{z_b} )  \\
	&=0.
	\end{align*}
	For the case (iii), we have
	\begin{align*}
	X = t_{ab}^{-1} \lvert c_{ab} \rvert [1]_{z^b}
	-t_{ba}^{-1} \lvert c_{ba} \rvert [1]_{z^a} = 0.
	\end{align*}
\end{proof}

\begin{theorem}\label{theorem:T-datum from cartan}
	Let $N_0,N_+,N_- \in \mat_{H \times H} (\ZZ[z])$ be the matrices defined by
	\begin{align*}
	N_{\varepsilon} &= \bigl( z^{c_a} \tilde{n}_{am,bk}^{\varepsilon} \bigr)_{am,bk \in H} \quad\quad (\varepsilon \in \{0, {+}, {-}\}).
	\end{align*}
	Then the triple $(N_0,N_+,N_-)$ and the pair $(A_+,A_-)=(N_0-N_+,N_0-N_-)$ satisfy
	\begin{enumerate}
		\item the conditions (\ref{item:N1}), (\ref{item:N2}), and (\ref{item:N4}),
		\item the symplectic relation $A_+ A_-^\dagger = A_- A_+^\dagger$,
		\item and the condition (\ref{item:N3}) if and only if
		\begin{align}\label{eq:Cartan condition}
			\text{$c_{ab}  \mid c_{ba} $ or $c_{ba}  \mid c_{ab} $ for any $1 \leq a,b \leq n$}.
		\end{align}
	\end{enumerate}
	Consequently, for any Cartan matrix $C$ satisfies the condition \eqref{eq:Cartan condition} and any integer $\level$ greater than or equal to $2$, the triple $(A_+,A_-,I_{H})$ is a T-datum of size $\lvert H \rvert$.
\end{theorem}
\begin{proof}
	The conditions (\ref{item:N1}), (\ref{item:N2}), and (\ref{item:N4}) are obvious from the definition.
	The symplectic relation follows from Lemma \ref{lemma:symmetric cartan level}
	and the fact that $\tilde{n}_{am,bk}^{\varepsilon}$ are invariant under $z \mapsto z^{-1}$.
	The condition (\ref{item:N3}) is equivalent to $(t_{ba}-1)c_b < c_a$ for any $a,b$ such that $a \sim b$, and this is equivalent to $\lcm(c_a,c_b) < c_a + c_b$ for any $a,b$ such that $a \sim b$.
	This happens if and only if $c_a \mid c_b$ or $c_b \mid c_a$ for any $a,b$ such that $a \sim b$, and this is equivalent to the condition \eqref{eq:Cartan condition}.
\end{proof}

\begin{remark}
	If the Cartan matrix in Theorem \ref{theorem:T-datum from cartan} and its symmetrizer satisfy the condition
	\begin{align*}
		c_{ab} < -1 \; \Rightarrow \; c_a = - c_{ba}=1,
	\end{align*}
	which implies \eqref{eq:Cartan condition}, the mutation loop corresponding to the T-datum
	$(A_+,A_-,I_H)$ is explicitly constructed in \cite{NakanishiTamely}.
	The T-system associated with this T-datum is a certain truncation
	a T-system associated with Kirillov-Reshetikhin modules of the quantum affinization of a quantum Kac-Moody algebra~\cite{Hernandez,KNSaff} (a truncation and a quantum Kac-Moody algebra are associated with $\level$ and $C$, respectively).
\end{remark}

\section{Periodic Y/T-systems}
\label{section:periodic yt}
\subsection{Finite type T-data}

\begin{definition}\label{def:finite type}
	We say that a T-datum $\alpha$ is of \emph{finite type} if
	the set
	$\{T_a(u) \in \mathscr{T}^{\circ}(\alpha,R,Y) \mid (a,u) \in R \}$ is a finite set.
\end{definition}

Definition \ref{def:finite type} does not depend on $R$ since the set being considered is a finite set for some $R$ if and only if it is a finite set for $[1,r] \times \ZZ$ by (\ref{item:R3}).
We will see that this is also independent of the choice of $Y$.

\begin{definition}\label{def:principal and universal solutions}
	Let $\alpha$ be a T-datum and $R$ be a consistent subset for $\alpha$.
	\begin{enumerate}
		\item We define $Y_{\mathrm{prin}}(\alpha,R)$ to be the solution of the Y-system
		associated with $(\alpha,R)$ in $\mathrm{Trop}(u_{a,p})_{(a,p) \in R_{\mathrm{in}}}$ such that
		$u_{a,p} = y_{a,p}$ for any $(a,p) \in R_{\mathrm{in}}$, where $y_{a,p}$ is defined by \eqref{eq:y from Y}.
		By Theorem \ref{theorem:embedding into cluster algebra}, the T-algebra $\mathscr{T}^{\circ}(\alpha,R,Y)$ is embedded into the cluster algebra with principal coefficients (see~\cite{FZ4} for the definition of cluster algebras with principal coefficients).
		\item We define $Y_{\mathrm{univ}}(\alpha,R)$ to be the solution of the Y-system
		associated with $(\alpha,R)$ in $\QQ_{\mathrm{sf}}(u_{a,p})_{(a,p) \in R_{\mathrm{in}}}$ such that
		$u_{a,p} = y_{a,p}$ for any $(a,p) \in R_{\mathrm{in}}$, where $y_{a,p}$ is defined by \eqref{eq:y from Y}.
	\end{enumerate}
\end{definition}

\begin{definition}\label{def:periodic}
	Let $\alpha$ be a T-datum and $R$ be a consistent subset for $\alpha$.
	Let $\Omega$ be a integer with $t \mid \Omega$,
	where $t$ is the integer in (\ref{item:R3}) in Definition \ref{def:consistent subset}.
	\begin{enumerate}
		\item 
		We say that a solution $(Y_a(u))_{(a,u) \in R}$ of the Y-system associated with $(\alpha,R)$ is \emph{periodic with period $\Omega$} if $Y_a(u)=Y_a(u+\Omega)$ for any $(a,u) \in R$.
		\item 
		We say that the T-system associated with $(\alpha,R,Y)$ is \emph{periodic with period $\Omega$} if
		$Y$ is periodic with period $\Omega$ and $T_a(u) = T_a(u+\Omega)$ in $\mathscr{T}^{\circ}(\alpha,R,Y)$ for any $(a,u) \in R$.
	\end{enumerate}
\end{definition}

Definition \ref{def:periodic} also does not depend on $R$.
By the synchronicity phenomenon in cluster algebras~\cite{nakanishi2019synchronicity},
we have the following assertion:

\begin{theorem}\label{theorem:periodic conditions}
	Let $\alpha$ be a T-datum and $R$ be a consistent subset for $\alpha$.
	Let $\Omega$ be a integer with $t \mid \Omega$,
	where $t$ is the integer in (\ref{item:R3}) in Definition \ref{def:consistent subset}.
	Then the following conditions are equivalent:
	\begin{enumerate}
		\item The T-system associated with $(\alpha,R,Y)$ is periodic with period $\Omega$ for some $Y$.
		\item The T-system associated with $(\alpha,R,Y)$ is periodic with period $\Omega$ for any $Y$.
		\item $Y_{\mathrm{prin}}(\alpha,R)$ is periodic with period $\Omega$.
		\item $Y_{\mathrm{univ}}(\alpha,R)$ is periodic with period $\Omega$.
	\end{enumerate}
\end{theorem}
\begin{proof}
	This follows from Theorem \ref{theorem:embedding into cluster algebra} together with the synchronicity phenomenon in cluster algebras~\cite[Theorem 5.2 and 5.5]{nakanishi2019synchronicity}.
\end{proof}

It is easy to see that $\alpha$ is of finite type (for some $Y$) if and only if the condition (1) in Theorem \ref{theorem:periodic conditions} holds for some $\Omega>0$.
Therefore, Theorem \ref{theorem:periodic conditions} implies that Definition \ref{def:finite type} does not depend on $Y$.

\subsection{Simultaneous positivity of finite type T-data}
For any matrix $A \in \mat_{r \times r} (\ZZ[z^{\pm1}])$,
we define $\mathring{A} \in \mat_{r \times r} (\ZZ)$ by
$\mathring{A} = A|_{z=1}$.
For any vector $u,v \in \RR^r$, we write $u>v$ and $u \geq v$ if all components of the vector $u-v$ are positive and non-negative, respectively.
The following is the main theorem in this section,
which gives a effective method to determine that a given T-datum is not of finite type.

\begin{theorem}\label{theorem:simultaneous positivity}
	Let $\alpha=(A_+,A_-,D)$ be a T-datum.
	If $\alpha$ is of finite type, then there exists a vector $v>0$ such that $\mathring{A}_+^{\mathsf{T}} v >0$ and $\mathring{A}_-^{\mathsf{T}} v >0$.
\end{theorem}
\begin{proof}
	Without loss of generality we can assume that $\alpha$ is indecomposable.
	It is sufficient to find a vector $v \geq 0$ such that $\mathring{A}_+^{\mathsf{T}} v >0$ and $\mathring{A}_-^{\mathsf{T}} v >0$ since such a vector plus a sufficiently small positive vector is a desired vector.
	Let $c \in [1,r]$.
	Let $\mathfrak{t}^{(c)} = (\mathfrak{t}_a^{(c)}(u))_{(a,u) \in [1,r] \times \ZZ}$ be the family of integers defined in Section \ref{section:tropical T},
	that is, $\mathfrak{t}_a^{(c)}(u)$ is the minus of the lowest power of $T_c(0)$ in $T_a(u)$, where $T_a(u)$ is written as a Laurent polynomial in $(T_c(p))_{(c,p) \in R_{\mathrm{in}}}$.
	We also define the family of integers $\tilde{\mathfrak{t}}^{(c)} = (\tilde{\mathfrak{t}}_a^{(c)}(u))_{(a,u) \in [1,r] \times \ZZ}$,
	where $\tilde{\mathfrak{t}}_a^{(c)}(u)$ is the highest power of $T_c(0)$ in $T_a(u)$.
	By the definitions, we have $\mathfrak{t}_a^{(c)}(u) + \tilde{\mathfrak{t}}_a^{(c)}(u) \geq 0$ for any $(a,u) \in [1,r] \times \ZZ$.
	By Proposition 2.6 in~\cite{ReadingStella},
	the family of integers $\tilde{\mathfrak{t}}^{(c)}$ is uniquely
	determined by the initial conditions
	\begin{align}\label{eq:tropical T initial tilde}
	\tilde{\mathfrak{t}}_a^{(c)} (p)
	=
	\begin{cases}
	1 & \text{if $(a,p) = (c,0)$,}\\
	0	& \text{if $(a,p) \neq (c,0)$ and $0 \leq p <p_a$,}
	\end{cases}
	\end{align}
	together with the following recurrence relation for each $(a,u) \in [1,r] \times \ZZ$:
	\begin{align}\label{eq:tropical T tilde}
	\sum_{b,p} n_{ba;p}^0 \tilde{\mathfrak{t}}_b^{(c)} (u+p) &=
	\max \biggl(\sum_{b,p} n_{ba;p}^- \tilde{\mathfrak{t}}_b^{(c)} (u+p),\sum_{b,p} n_{ba;p}^+ \tilde{\mathfrak{t}}_b^{(c)} (u+p) \biggr) .
	\end{align}
	The family of integers $\mathfrak{t}^{(c)}$ and $\tilde{\mathfrak{t}}^{(c)}$ satisfy the same recurrence relation, but have the different initial conditions.
	
	Let $v_a^{(c)}$ and $\tilde{v}_a^{(c)}$ be the integers defined by
	\begin{align*}
		v_a^{(c)} = \sum_{u=0}^{\Omega-1} \mathfrak{t}_a^{(c)} (u),\quad
		\tilde{v}_a^{(c)} = \sum_{u=0}^{\Omega-1} \tilde{\mathfrak{t}}_a^{(c)} (u),
	\end{align*}
	where $\Omega$ is a period of the T-system.
	By the periodicity of the T-system,
	we have
	\begin{align*}
		v_a^{(c)} = \sum_{u=0}^{\Omega-1} \mathfrak{t}_a^{(c)} (u+p),\quad
		\tilde{v}_a^{(c)} = \sum_{u=0}^{\Omega-1} \tilde{\mathfrak{t}}_a^{(c)} (u+p)
	\end{align*}
	for any $p \in \ZZ$.
	By summing up \eqref{eq:tropical T} with respect to the period,
	we have
	\begin{align}
	\label{eq:periodic tropical T 1}
	\sum_{b} \mathring{n}_{ba}^0 v_b^{(c)} &=
	\sum_{u=0}^{\Omega-1}
	\max \biggl(\sum_{b,p} n_{ba;p}^- \mathfrak{t}_b^{(c)} (u+p),\sum_{ba;p} n_{ba;p}^+ \mathfrak{t}_b^{(c)} (u+p) \biggr) \\
	\label{eq:periodic tropical T 2}
	&\geq
	\max \biggl(\sum_{b} \mathring{n}_{ba}^- v_b^{(c)} ,\sum_{b} \mathring{n}_{ba}^+ v_b^{(c)} \biggr),
	\end{align}
	where $\mathring{n}_{ba}^{\varepsilon} = \sum_{p} n_{ba,p}$.
	This implies that
	$\mathring{A}_+^\mathsf{T} v^{(c)} \geq 0$ and
	$\mathring{A}_-^\mathsf{T} v^{(c)} \geq 0$.
	Similarly, we have $\mathring{A}_+^\mathsf{T} \tilde{v}^{(c)} \geq 0$ and
	$\mathring{A}_-^\mathsf{T} \tilde{v}^{(c)} \geq 0$ by summing up \eqref{eq:tropical T tilde} with respect to the period.
	Let $v$ and $\tilde{v}$ be the vectors defined by
	\begin{align*}
		v=
		\sum_{c=1}^r
		\begin{bmatrix}
			v_1^{(c)}\\
			\vdots \\
			v_r^{(c)}
		\end{bmatrix},\quad
		\tilde{v}=
		\sum_{c=1}^r
		\begin{bmatrix}
		\tilde{v}_1^{(c)}\\
		\vdots \\
		\tilde{v}_r^{(c)}
		\end{bmatrix}.
	\end{align*}
	We then define a vector $v'$ by $v'=v + \tilde{v}$.
	We have $v' \geq 0$ since $\mathfrak{t}_a^{(c)}(u) + \tilde{\mathfrak{t}}_a^{(c)}(u) \geq 0$.
	We also have $\mathring{A}_+^\mathsf{T} v' \geq 0 $ and $\mathring{A}_-^\mathsf{T} v' \geq 0 $.
	Therefore, if we prove that
	$\mathring{A}_+^\mathsf{T} v>0 $ and $\mathring{A}_-^\mathsf{T} v>0 $, the assertion of the theorem follows.
	
	From \eqref{eq:tropical Y hat 1}, \eqref{eq:tropical Y hat pm}, and \eqref{eq:periodic tropical T 2},
	the $a$-th component of $\mathring{A}_{\pm}^{\mathsf{T}} v$ is positive
	if and only if there exists $(c,u) \in [1,r] \times \ZZ$ such that $\maxzero{\pm \mathfrak{y}_{a}^{(c)} (u) } \neq 0$.
	Therefore the $a$-th component of $\mathring{A}_{\pm}^{\mathsf{T}} v$ is positive if the $a$-th column of $N_{\pm}$ is non-zero by (3) in Lemma \ref{lemma:tropical T and Y around initial}.
	It remains to prove that the $a$-th component of $\mathring{A}_{\pm}^{\mathsf{T}} v$
	is also positive when the $a$-th column of $N_{\pm}$ is zero.
	If both the $a$-th columns of $N_+$ and $N_-$ are zero,
	the assertion of the theorem follows from Corollary \ref{cor:Npm=0}.
	Thus we can assume that either the $a$-th column of $N_+$ or $N_-$ is non-zero.
	Without loss of generality we assume that the $a$-th column of $N_+$ is non-zero and the $a$-th column of $N_-$ is zero.
	Let $n_{ca;p}^{+} z^p$ be a term in the $a$-th column of $N_+$
	with the minimal degree among the terms in this column.
	Now we have
	\begin{align*}
		\maxzero{-\hat{\mathfrak{y}}_a^{(\sigma(c))}(-p_{\sigma(c)}-p)} &= 
		\sum_{b,q} (n_{ba;q}^{0}-n_{ba;q}^{-}) \mathfrak{t}_b^{(\sigma(c))} (-p_{\sigma(c)}-p+q)\\
		&=\sum_{b,q} n_{ba;q}^{0} \mathfrak{t}_b^{(\sigma(c))} (-p_{\sigma(c)}-p+q)\\
		&=\max\biggl(\sum_{b,q} n_{ba;q}^{+} \mathfrak{t}_b^{(\sigma(c))}(-p_{\sigma(c)}-p+q) ,0 \biggr) \\
		&= \max\bigl( n_{ca;p}^{+} \mathfrak{t}_c^{(\sigma(c))}(-p_{\sigma(c)}) ,0 \bigr) \\
		&= \max( n_{ca;p}^{+} ,0) \\
		&= n_{ca;p}^+,
	\end{align*}
	and this implies that the $a$-th component of $\mathring{A}_{-}^{\mathsf{T}} v$ is positive.
\end{proof}

\begin{example}\label{example:simultaneous positivity}
	\leavevmode
	\begin{enumerate}
		\item A T-datum of size $1$ (Theorem \ref{theorem:T-datum of size 1}) is of finite type if and only if $(A_+,A_-)$ is one of the following three pairs of matrices for some $p>0$:
		\begin{align*}
			\begin{array}{ll}
				A_+ =	
				\begin{bmatrix}
					1+z^{2p}
				\end{bmatrix},
				&A_- =	
				\begin{bmatrix}
					1+z^{2p}
				\end{bmatrix},\vspace{1ex}\\
				A_+ =	
				\begin{bmatrix}
					1-z^{p}+z^{2p}
				\end{bmatrix},
				&A_- =	
				\begin{bmatrix}
					1+z^{2p}
				\end{bmatrix},\vspace{1ex}\\
				A_+ =	
				\begin{bmatrix}
					1+z^{2p}
				\end{bmatrix},
				&A_- =	
				\begin{bmatrix}
					1-z^{p}+z^{2p}
				\end{bmatrix}.
			\end{array}
		\end{align*}
		The if part is proved by direct calculations,
		and the only if part follows from Theorem \ref{theorem:simultaneous positivity}.
		\item A T-datum associated with a bipartite recurrent quiver, which is a special case of a T-datum in Proposition \ref{prop:double weak Cartan}, is of finite type if and only if both $A$ and $A'$ are direct sums of $ADE$ Cartan matrices~\cite{GalashinPylyavskyy}.
		In fact, Theorem \ref{theorem:simultaneous positivity} generalizes Proposition 7.1 in \cite{GalashinPylyavskyy} to arbitrary T-data.
		\item A T-datum in Example \ref{example:tensor} is of finite type if and only if both $\bar{A}$ and $\bar{A}'$ are of finite type Cartan matrices, except that one of them can be of tadpole type.
		The if part is proved in \cite{Keller},
		and the only if part follows from Theorem \ref{theorem:simultaneous positivity}.
		\item A T-datum in Theorem \ref{theorem:T-datum from cartan} is of finite type if and only if $C$ is of finite type Cartan matrix.
		The if part is proved in \cite{IIKKNa,IIKKNb},
		and the only if part follows from Proposition \ref{prop:contains Dynkin} and Theorem \ref{theorem:simultaneous positivity}.
	\end{enumerate}
\end{example}

\subsection{Special values of the dilogarithm function}
\label{section:dilog}

\begin{definition}\label{def:Cartan-like}
	Let $\alpha=(A_+,A_-,D)$ be a T-datum.
	Let $P=\diag(z^{-p_a/2})_{a \in [1,r]}$, where $p_a$ is the integer in (\ref{item:N1}).
	We say that $\alpha$ is \emph{Cartan-like} if
	both the matrices $PA_+$ and $PA_-$
	are invariant under $z \mapsto z^{-1}$.
\end{definition}

This terminology comes from the fact that T-data in Section \ref{section:commuting Cartan} satisfy this property.
All examples in Section \ref{section:examples} are also Cartan-like.
Note that $\mathring{A}_\pm$ are not Cartan matrices in general since they may not be sign-symmetric (see examples in Table \ref{table:T data r=2} and \ref{table:T-data r=3}).
The matrix $N_0$ in the Cartan-like T-datum should be a diagonal matrix.
This property is useful due to the following fact on real square matrices whose off-diagonal entries are non-positive.
As a result, we assign a positive definite symmetric matrix to any Cartan-like T-datum of finite type (Proposition \ref{prop:K and K check}).

\begin{lemma}[{\cite[Theorem 4.3]{FiedlerPtak}}]
	\label{lemma:positive matrix}
	Let $A$ be a real square matrix whose off-diagonal entries are all non-positive.
	Then the following conditions are equivalent:
	\begin{enumerate}
		\item
		there exists $v > 0$ such that $Av>0$,
		\item
		all real eigenvalues of $A$ are positive.
	\end{enumerate}
\end{lemma}

\begin{proposition}\label{prop:K and K check}
	Let $\alpha=(A_+,A_-,D)$ is a Cartan-like T-datum of finite type.
	Then the following assertions hold:
	\begin{enumerate}
		\item $\mathring{A}_+$ and $\mathring{A}_-$ are invertible.
		\item Let $K=(\kappa_{ab})_{a,b\in[1,r]}$ be the matrix defined by $K=\mathring{A}_+^{-1} \mathring{A}_-$.
		Then $KD$ is a positive definite symmetric matrix.
		\item Let $K^{\vee} = (\check{\kappa}_{ab})_{a,b\in[1,r]}$ be the matrix defined by $K^{\vee}=D^{-1}KD$.
		Then $K^{\vee} D^{\vee}$ is a positive definite symmetric matrix.
	\end{enumerate}
\end{proposition}
\begin{proof}
	By Theorem \ref{theorem:simultaneous positivity} and Lemma \ref{lemma:positive matrix}, all the real eigenvalues of $\mathring{A}_{\pm}$ are positive.
	This implies (1).
	Since $K^{\vee} = (\mathring{A}_+^{\vee})^{-1} \mathring{A}_-^{\vee}$, the assertion (3) follows from the assertion (2) for $\alpha^{\vee}$.
	We now prove (2).
	We first see that $K$ is symmetric due to the symplectic relation.
	Suppose that there exists an eigenvector $v$ of $K$
	with a non-positive eigenvalue.
	Let us denote by $-\lambda$ this eigenvalue.
	Then we have $(\mathring{A}_- + \lambda \mathring{A}_+) v =0$.
	Thus $0$ is an eigenvalue of $\mathring{A}_- + \lambda \mathring{A}_+$.
	Since $\lambda \geq 0$, all off-diagonal entries in $\mathring{A}_- + \lambda \mathring{A}_+$ are non-positive.
	Moreover, this matrix satisfies the condition (1) in Lemma \ref{lemma:positive matrix} by Theorem \ref{theorem:simultaneous positivity}.
	Thus its real eigenvalues are positive by Lemma \ref{lemma:positive matrix},
	a contradiction. 
\end{proof}

The function
\begin{align*}
\mathrm{Li}_2(z) := \sum_{n=1}^\infty \frac{z^n}{n^2} \quad (\lvert z \rvert < 1)
\end{align*}
is called the \emph{dilogarithm function}.
The \emph{Rogers dilogarithm function} is a function on the interval $(0,1)$ defined as follows:
\begin{align*}
L(x) = \mathrm{Li}_2(x) + \frac{1}{2} \log (x) \log(1-x).
\end{align*}
We can define $L(0)=0$ and $L(1)=\pi^2 / 6$ by continuity.

For any T-datum $\alpha=(A_+,A_-,D)$,
we denote by $d_a$ and $d_a^{\vee}$ the
$a$-th entries in $D$ and $D^{\vee}$, respectively.

\begin{theorem}
	\label{theorem:dilog rational}
	Let $\alpha = (A_+,A_-,D)$ be a Cartan-like T-datum of finite type.
	Let $K^{\vee} = (\check{\kappa}_{ab})_{a,b\in[1,r]}$ be the matrix defined in Proposition \ref{prop:K and K check}.
	\begin{enumerate}
		\item The system of equations
		\begin{align}\label{eq:f nahm eq}
			f_a = \prod_{b=1}^{r} (1-f_b)^{\check{\kappa}_{ab}} \quad\quad (a \in [1,r])
		\end{align}
		has a unique real solution such that $0<f_a<1$ for any $a \in [1,r]$.
		\item Let $(f_a)_{a \in [1,r]}$ be the unique solution in (1).
		Define the real number $c_{\alpha}$ by
		\begin{align*}
			c_{\alpha} := \frac{6}{\pi^2} \sum_{a=1}^{r} d_a L(f_a).
		\end{align*}
		Then we have $c_\alpha \in \QQ$.
	\end{enumerate}
\end{theorem}
\begin{proof}
	We define a function $F_{\alpha}(x): [0,\infty)^r \to \RR$ by
	\begin{align*}
		F_{\alpha} (x) = \frac{1}{2} x^{\mathsf{T}} K^{\vee} D^{\vee} x
		+ \sum_{a=1}^{r} (d_a^{\vee})^{-1} \mathrm{Li}_2 (\exp(-d_a^{\vee} x_a)).
	\end{align*}
	By setting $f_a = 1- \exp(-d_a^{\vee} x_a)$,
	we see that
	the statement (1) is equivalent to saying that the function $F_\alpha(x)$ has a unique critical point in $(0,\infty)^r$.
	This follows from the fact that $K^{\vee} D^{\vee}$ is a positive definite symmetric matrix, as in the proof of Lemma 2.1 in \cite{VlasenkoZwegers}.
	
	We now prove (2).
	From the result in \cite[Section 6]{Nakb},	
	the value
	\begin{align}\label{eq:sum of dilog 0<=u<N}
		\frac{6}{\pi^2}
		\sum_{\substack{(a,u) \in [1,r] \times \ZZ \\ 0 \leq u <\Omega}}  d_a L\biggl(\frac{Y_a(u)}{1 \oplus Y_a(u)}\biggr)
	\end{align}
	is an integer for any solution $Y=(Y_a(u))_{(a,u)\in [1,r]\times \ZZ}$ of the Y-system associated with $(\alpha,[1,r]\times \ZZ)$ in the semifield $\RR_{>0}$,
	where $\Omega>0$ is a period of $Y_{\mathrm{univ}}(\alpha,[1,r] \times \ZZ)$.
	Moreover, this value is independent of the choice of $Y$.
	In fact, Nakanishi~\cite{Nakb} proved that these facts follow from the sign coherence property of cluster algebras, which was proved by Gross, Hacking, Keel, and Kontsevich~\cite{GHKK} for skew-symmetrizable cluster algebras.
	
	It is easy to see that
	the system of equations \eqref{eq:f nahm eq} is equivalent to
	\begin{align*}
		\prod_{b=1}^r f_b^{\sum_{p\in\ZZ} (\check{n}_{ab;p}^{0}-\check{n}_{ab;p}^{+})} = 
		\prod_{b=1}^r (1-f_b)^{\sum_{p\in\ZZ}(\check{n}_{ab;p}^{0}-\check{n}_{ba;p}^{-})}\quad\quad(a \in [1,r]).
	\end{align*}
	Thus the family $Y=(Y_a(u))_{(a,u) \in [1,r] \times \ZZ}$ defined by
	$Y_a(u) = f_a/(1-f_a)$ is a solution of the Y-system associated with $(\alpha,[1,r] \times \ZZ)$ in $\RR_{>0}$.
	Since this is a constant solution with respect to $u$, the integer \eqref{eq:sum of dilog 0<=u<N} is equal to $\Omega c_{\alpha}$.
	Thus $c_{\alpha}$ is a rational number.
\end{proof}

\begin{table}[t]
	\begin{align*}
	\begin{array}{llll}
	A_+ & A_- & K & c_{\alpha} \\ \hline \vspace{-2ex} \\
	\begin{bmatrix}
	1+z^2 & -z \\
	-z & 1+z^2
	\end{bmatrix} &
	\begin{bmatrix}
	1+z^2 & 0 \\
	0 & 1+z^2
	\end{bmatrix}  &
	\begin{bmatrix}
	4/3 & 2/3 \\
	2/3 & 4/3
	\end{bmatrix} &
	4/5 \\ \vspace{-2ex} \\
	\begin{bmatrix}
	1+z^2 & -z \\
	-z & 1+z^2
	\end{bmatrix} &
	\begin{bmatrix}
	1-z+z^2 & 0 \\
	0 & 1-z+z^2
	\end{bmatrix}  &
	\begin{bmatrix}
	2/3 & 1/3 \\
	1/3 & 2/3
	\end{bmatrix} &
	1 \\ \vspace{-2ex} \\
	\begin{bmatrix}
	1+z^2 & -z\\
	-z-z^5 & 1+z^6
	\end{bmatrix}&
	\begin{bmatrix}
	1+z^2 & 0\\
	-z^3 & 1+z^6
	\end{bmatrix}  &
	\begin{bmatrix}
	3/2 & 1\\
	1 & 2
	\end{bmatrix} &
	5/7   \\ \vspace{-2ex} \\
	\begin{bmatrix}
	1 + z^{2} & -z \\
	-z - z^{2} & 1 + z^{3}
	\end{bmatrix} &
	\begin{bmatrix}
	1 -z + z^{2} & 0 \\
	0 & 1 + z^{3}
	\end{bmatrix}&
	\begin{bmatrix}
	1 & 1 \\
	1 & 2
	\end{bmatrix}&
	3/4  \\ \vspace{-2ex} \\
	\begin{bmatrix}
	1 + z^{2} & -z \\
	-z - z^{5} - z^{9} & 1 + z^{10}
	\end{bmatrix}&
	\begin{bmatrix}
	1 + z^{2} & 0 \\
	-z^{3} - z^{7} & 1 + z^{10}
	\end{bmatrix}&
	\begin{bmatrix}
	2 & 2 \\
	2 & 4
	\end{bmatrix} &
	4/7
	\end{array}
	\end{align*}
	\caption{Examples of Cartan-like T-data of finite type of size $2$, where $D=I_2$ in these examples.}
	\label{table:T data r=2}
\end{table}

\begin{example}
	We give some examples of Cartan-like T-data of finite type
	of size $2$ and size $3$ in Table \ref{table:T data r=2} and \ref{table:T-data r=3}, respectively,
	where the matrix $D$ in these examples are the identity matrices.
	We also show the positive definite symmetric matrix $K$ and the rational number $c_{\alpha}$ associated with these T-data.
	The rational number $c_{\alpha}$ can be computed by using Theorem 6.8 in~\cite{Nakb}.
\end{example}

\begin{table}[t]
	{\footnotesize 
	\begin{align*}
		\begin{array}{llll}
			A_+ & A_- & K & c_{\alpha} \\ 
			\hline 
			\vspace{-2ex} \\
			\begin{bmatrix}
			1 + z^{2} & -z & 0 \\
			-z & 1 + z^{2} & -z \\
			0 & -z & 1 + z^{2}
			\end{bmatrix}&
			\begin{bmatrix}
			1 + z^{2} & 0 & 0 \\
			0 & 1 + z^{2} & 0 \\
			0 & 0 & 1 + z^{2}
			\end{bmatrix}&
			\begin{bmatrix}
			3/2 & 1 & 1/2 \\
			1 & 2 & 1 \\
			1/2 & 1 & 3/2
			\end{bmatrix} &
			1 \\
			\vspace{-2ex} \\
			\begin{bmatrix}
			1 + z^{2} & -z & 0 \\
			-z & 1 + z^{2} & -z \\
			0 & -z & 1 + z^{2}
			\end{bmatrix}&
			\begin{bmatrix}
			1 -z+ z^{2} & 0 & 0 \\
			0 & 1 -z+ z^{2} & 0 \\
			0 & 0 & 1 -z+ z^{2}
			\end{bmatrix}&
			\begin{bmatrix}
			3/4 & 1/2 & 1/4 \\
			1/2 & 1 & 1/2 \\
			1/4 & 1/2 & 3/4
			\end{bmatrix} &
			9/7 \\
			\vspace{-2ex} \\
			\begin{bmatrix}
			1 + z^2 & -z & 0 \\
			-z & 1 + z^2 & -z \\
			0 & -z - z^2 & 1 + z^3
			\end{bmatrix} &
			\begin{bmatrix}
			1 - z + z^2 & 0 & 0 \\
			0 & 1 - z + z^2 & 0 \\
			0 & 0 & 1 + z^3
			\end{bmatrix}&
			\begin{bmatrix}
			1 & 1 & 1 \\
			1 & 2 & 2 \\
			1 & 2 & 3
			\end{bmatrix} &
			9/10 \\
			\vspace{-2ex} \\
			\begin{bmatrix}
			1+z^2 & 0 & -z\\
			-z^3 & 1+z^6 & 0\\
			-z-z^7 & -z^2-z^6 & 1+z^8
			\end{bmatrix}  &
			\begin{bmatrix}
			1+z^2 & -z & 0\\
			-z-z^5 & 1+z^6 & 0\\
			0 & 0 & 1+z^8
			\end{bmatrix} &
			\begin{bmatrix}
			2 &0 & 2 \\
			0 & 1 & 1 \\
			2 & 1 & 4
			\end{bmatrix}  & 
			1 \\
			\vspace{-2ex} \\
			\begin{bmatrix}
			1+z^2 & -z & 0\\
			-z & 1+z^2 & -z\\
			0 & -z & 1-z+z^2
			\end{bmatrix}  &
			\begin{bmatrix}
			1+z^2 & 0 & 0\\
			0 & 1+z^2 & 0\\
			0 & 0 & 1+z^2
			\end{bmatrix} &
			\begin{bmatrix}
			2 &2 & 2 \\
			2 & 4 & 4 \\
			2 & 4 & 6
			\end{bmatrix}  & 
			2/3\\
			\vspace{-2ex} \\
			\begin{bmatrix}
			1+z^2 & -z & 0\\
			-z-z^5 & 1+z^6 & -z^3\\
			0 & -z^3 & 1+z^6
			\end{bmatrix}  &
			\begin{bmatrix}
			1+z^2 & 0 & 0\\
			-z^3 & 1+z^6 & 0\\
			0 & 0 & 1+z^6
			\end{bmatrix} &
			\begin{bmatrix}
			2 &2 & 1 \\
			2 & 4 & 2 \\
			1 & 2 & 2
			\end{bmatrix}  & 
			4/5\\
			\vspace{-2ex} \\
			\begin{bmatrix}
			1-z+z^2 & -z & 0\\
			-z & 1+z^2 & 0\\
			0 & 0 & 1+z^{5}
			\end{bmatrix}  &
			\begin{bmatrix}
			1+z^2 & 0 & 0\\
			0 & 1+z^2 & -z\\
			-z^2-z^3 & -z-z^4 & 1+z^{5}
			\end{bmatrix} &
			\begin{bmatrix}
			4 &2 & -1 \\
			2 & 2 & -1 \\
			-1 & -1 & 1
			\end{bmatrix}  & 
			3/2
		\end{array}
	\end{align*}}
	\caption{Examples of Cartan-like T-data of finite type of size $3$, where $D=I_3$ in these examples.}
	\label{table:T-data r=3}
\end{table}

\subsection{Partition $q$-series}
Let $\alpha=(A_+,A_-,D)$ be a Cartan-like T-datum of finite type of size $r$.
We define two sets $H_{\alpha}$ and $H'_{\alpha}$ by
\begin{align*}
	H_{\alpha}&=\bigl\{ (m,l) \in \ZZ^r \times \QQ^r \mid \mathring{A}_- m = \mathring{A}_+ l \bigr\},\\
	H'_{\alpha}&= \bigl\{ \bigl( (\mathring{A}_+^{\vee})^{\mathsf{T}} n,(\mathring{A}_-^{\vee})^{\mathsf{T}} n \bigr) \mid n \in \ZZ^r \bigr\}.
\end{align*}
These are free abelian groups of rank $r$,
and the symplectic relation implies that $H'_{\alpha}$ is a subgroup of $H_{\alpha}$.
Let $S_\alpha$ be the quotient group of $H_{\alpha}$ by $H'_{\alpha}$:
$S_\alpha = H_{\alpha} / H'_{\alpha}$.
This is a finite abelian group that is isomorphic to $\ZZ^r / (\text{the rows space of $\mathring{A}_+^{\vee}$})$.
In particular, the order of $S_{\alpha}$ is $\det \mathring{A}_+$.
For any $\sigma \in S_{\alpha}$,
we denote by $\sigma_{\geq 0}$ the set $\{ (m,l) \in \sigma \mid m \geq 0 \}$.

\begin{definition}
	Let $\alpha=(A_+,A_-,D)$ be a Cartan-like T-datum of finite type.
	Let $\sigma \in S_{\alpha}$.
	We define the \emph{partition $q$-series} of $\alpha$ at $\sigma$ by
	\begin{align*}
	\mathcal{Z}_{\alpha,\sigma}(q) := \sum_{(m,l) \in \sigma_{\geq 0}} \frac{q^{\frac{1}{2}\langle m,\,l \rangle}}{\prod_{a=1}^{r} (q^{d_a^{\vee}})_{m_a}},
\end{align*}
	where $\langle m,l \rangle := m^{\mathsf{T}} D^{\vee} l$
	and $(q)_n = \prod_{i=1}^n (1-q^i)$ is the $q$-Pochhammer symbol.
	We also define the \emph{total partition $q$-series} of $\alpha$
	by
	\begin{align*}
		\mathcal{Z}_{\alpha,\mathrm{tot}}(q) := \sum_{\sigma \in S_{\alpha}}
		\mathcal{Z}_{\alpha,\sigma}(q)
		= \sum_{ m \in (\ZZ_{\geq 0})^r} \frac{q^{\frac{1}{2} m^{\mathsf{T}} K^{\vee} D^{\vee} m}}{\prod_{a=1}^{r} (q^{d_a^{\vee}})_{m_a}}.
	\end{align*}
\end{definition}

\begin{proposition}
	\begin{enumerate}
		\item The partition $q$-series $\mathcal{Z}_{\alpha,\sigma}(q)$ with $q=e^{2 \pi i \tau}$ converges to a holomorphic function on the upper half plane $\mathbb{H}=\{ \tau \in \CC \mid \image \tau > 0 \}$,
		where we set $q^{\kappa} = e^{2 \pi i \tau \kappa}$ for any $\kappa \in \QQ$.
		\item We have 
		\begin{align*}
		\lim_{\varepsilon \searrow 0} \varepsilon \log \mathcal{Z}_{\alpha,\mathrm{tot}} (e^{-\varepsilon}) = \frac{\pi^2}{6\delta} c_{\alpha},
		\end{align*}
		where $\delta =\lcm(d_1,\dots, d_r)\gcd(d_1,\dots, d_r)$ and $c_{\alpha}$ is the rational number in Theorem \ref{theorem:dilog rational}.
	\end{enumerate}
\end{proposition}
\begin{proof}
	(1) follows from the fact that $K^\vee D^\vee$ is a positive definite symmetric matrix (Proposition \ref{prop:K and K check}).
	(2) follows from the asymptotic analysis in~\cite{VlasenkoZwegers}.
\end{proof}

Let $\Gamma \subseteq \mathrm{SL}(2,\ZZ)$ be a congruence subgroup.
We say that a holomorphic function $f(\tau)$ on the upper half plane is a \emph{modular function} with respect to $\Gamma$ if $f(\tau) = f(\tfrac{a \tau +b}{c\tau +d })$ for any $\tau \in \mathbb{H}$ and
$\left[ \begin{smallmatrix}
a & b \\ c & d
\end{smallmatrix}\right] \in \Gamma$,
and $f(\tau)$ is meromorphic at each cusp of $\Gamma$.

\begin{conjecture}\label{conj:q-series}
	Let $\alpha=(A_+,A_-,D)$ be a Cartan-like T-datum of finite type.
	Then there exists a congruence subgroup $\Gamma \subseteq \mathrm{SL}(2,\ZZ)$ such that
	$q^{-c_{\alpha}/24} \mathcal{Z}_{\alpha,\sigma} (q)$ with $q=e^{2 \pi i \tau}$ is a modular function with respect to $\Gamma$ for any $\sigma \in S_{\alpha}$,
	where $c_{\alpha}$ is the rational number in Theorem \ref{theorem:dilog rational}.
\end{conjecture}

\begin{remark}
	For any solution $(f_a)_{a \in [1,r]} \in \overline{\QQ}^r$ of \eqref{eq:f nahm eq}, we can define the element
	\begin{align}\label{eq:bloch element}
		\sum_{a=1}^r d_a [ f_a ] \in \mathcal{B}(F),
	\end{align}
	where $F$ is a number field containing the solution, and $\mathcal{B}(F)$ is the Bloch group of $F$.
	By the result in \cite[Section 6]{Nakb},
	we see that the element \eqref{eq:bloch element} is a torsion (see \cite{Lee}).
	Conjecture \ref{conj:q-series} can be regarded as a version of Nahm's Conjecture~\cite{Nahm,Zagier}, which relates torsions in Bloch groups and the modularity of $q$-hypergeometric series.
\end{remark}

\begin{theorem}\label{theorem:modularity r=1}
	Conjecture \ref{conj:q-series} holds for $r=1$.
\end{theorem}
\begin{proof}
	From (1) in Example \ref{example:simultaneous positivity},
	it is sufficient to prove the following three cases:
	\begin{align*}
		\alpha_1&=\bigl(
			1+z^{2p},
			1+z^{2p},
			d
		\bigr),\\
		\alpha_2&=\bigl(	
		1-z^p+z^{2p},
		1+z^{2p},
		d
		\bigr),\\
		\alpha_3&=\bigl(	
		1+z^{2p},
		1-z^p+z^{2p},
		d
		\bigr).
	\end{align*}
	For these three cases,
	we have $S_{\alpha_1} \cong \ZZ / 2\ZZ$, $S_{\alpha_2} \cong 0$, and $S_{\alpha_3} \cong \ZZ / 2\ZZ$.
	We also have $L_{\alpha_1} = d/2$, $L_{\alpha_2} = 2d/5$, and $L_{\alpha_3} = 3d/5$.
	
	We first consider $\alpha_2$ because its proof is the simplest and follows from a well-known discussion (e.g., see~\cite[Chapter II, Section 3]{Zagier}).
	In this case, the partition $q$-series is given by
	\begin{align*}
		\mathcal{Z}_{\alpha_2,0}(q) = \sum_{n \in \ZZ_{\geq 0}} \frac{q^{dn^2}}{(q^d)_n}.
	\end{align*}
	Using the Rogers-Ramanujan identity
	\begin{align*}
		\sum_{n=0}^\infty \frac{q^{n^2}}{(q)_{n}} = 
		\prod_{\substack{n>0 \\ n \equiv \pm 1 \!\!\! \pmod{5}}} \frac{1}{1-q^{n}},
	\end{align*}
	together with the Jacobi triple product identity,
	we have
	\begin{align}\label{eq:Z alpha 2}
		q^{-d/60} \mathcal{Z}_{\gamma}(q) = \frac{1}{2 \eta(q^d)}
		\sum_{n \in \ZZ} a(n) q^{dn^2/40},
	\end{align}
	where $\eta(q) =q^{1/24} \prod_{n=1}^{\infty} (1-q^n)$ is the Dedekind eta, and
	\begin{align*}
		a(n) =
		\begin{cases}
		1 & \text{if $n \equiv \pm 1 \pmod{20}$},\\
		-1 & \text{if $n \equiv \pm 9 \pmod{20}$},\\
		0 & \text{otherwise}.
		\end{cases}
	\end{align*}
	Since the right-hand side in \eqref{eq:Z alpha 2} is the ratio of modular forms of weight $1/2$, it is a modular function.
	Thus we obtain the assertion for $\alpha_2$.

	We now prove the assertion for $\alpha_1$ and $\alpha_3$.
	The partition $q$-series in these cases are given by
	\begin{align*}
		\begin{array}{ll}
			{\displaystyle
			\mathcal{Z}_{\alpha_3,0}(q) = \sum_{n \in \ZZ_{\geq 0}}
			\frac{q^{dn^2}}{(q^d)_{2n}}},
			&
			{\displaystyle
			\mathcal{Z}_{\alpha_3,1}(q) = \sum_{n \in \ZZ_{\geq 0}}
			\frac{q^{d(n^2+n+\frac{1}{4})}}{(q^d)_{2n+1}}},\\
			{\displaystyle
			\mathcal{Z}_{\alpha_1,0}(q) = \sum_{n \in \ZZ_{\geq 0}}
			\frac{q^{2dn^2}}{(q^d)_{2n}}},
			&
			{\displaystyle
			\mathcal{Z}_{\alpha_1,1}(q) = \sum_{n \in \ZZ_{\geq 0}}
			\frac{q^{d(2n^2+2n+\frac{1}{2})}}{(q^d)_{2n+1}}}.
		\end{array}
	\end{align*}
	To prove the assertion for $\alpha_1$ and $\alpha_3$,
	we use the following Rogers–Ramanujan type identities (see~\cite[S. 98, S. 94, S. 83, and S. 86]{mc2008rogers} and references therein):
	\begin{align}
		\label{eq:RR mod 20 1}
		\sum_{n=0}^\infty \frac{q^{n^2}}{(q)_{2n}} &= 
		\prod_{\substack{n>0 \\ n \equiv \pm 1,\pm 3,\pm 4,\pm 5, \pm 7, \pm 9 \!\!\! \pmod{20}}} \frac{1}{1-q^{n}},\\
		\label{eq:RR mod 20 2}
		\sum_{n=0}^\infty \frac{q^{n^2+n}}{(q)_{2n+1}} &= 
		\prod_{\substack{n>0 \\ n \equiv \pm 1,\pm 2,\pm 5,\pm 6, \pm 8, \pm 9 \!\!\! \pmod{20}}} \frac{1}{1-q^{n}},\\
		\label{eq:RR mod 16 1}
		\sum_{n=0}^\infty \frac{q^{2n^2}}{(q)_{2n}} &= 
		\prod_{\substack{n>0 \\ n \equiv \pm 2,\pm 3,\pm 4, \pm 5 \!\!\! \pmod{16}}} \frac{1}{1-q^{n}},\\
		\label{eq:RR mod 16 2}
		\sum_{n=0}^\infty \frac{q^{2n^2+2n}}{(q)_{2n+1}} &= 
		\prod_{\substack{n>0 \\ n \equiv \pm 1,\pm 4,\pm 6, \pm 7 \!\!\! \pmod{16}}} \frac{1}{1-q^{n}}.
	\end{align}
	Using \eqref{eq:RR mod 20 1} and \eqref{eq:RR mod 20 2} together with the quintuple product identity, we have
	\begin{align*}
		q^{-d/40} \mathcal{Z}_{\alpha_3,\sigma}(q) = \frac{1}{2\eta(q^d)}
		\sum_{n \in \ZZ} a_{3,\sigma}(n) q^{dn^2/60},
	\end{align*}
	where
	\begin{align*}
		a_{3,0}(n) &=
		\begin{cases}
			1 & \text{if $n \equiv \pm 1 \pmod{30}$},\\
			-1 & \text{if $n \equiv \pm 11 \pmod{30}$},\\
			0 & \text{otherwise},
		\end{cases}\\
		a_{3,1}(n) &=
		\begin{cases}
			1 & \text{if $n \equiv \pm 4 \pmod{30}$},\\
			-1 & \text{if $n \equiv \pm 14 \pmod{30}$},\\
			0 & \text{otherwise}.
		\end{cases}
	\end{align*}
	Thus we obtain the assertion for $\alpha_3$.
	Similarly, using \eqref{eq:RR mod 16 1} and \eqref{eq:RR mod 16 2} together with the quintuple product identity, we have
	\begin{align*}
	q^{-d/48} \mathcal{Z}_{\alpha_1,\sigma} (q) = 
	\frac{1}{2\eta(q^d)} \sum_{n \in \ZZ} a_{1,\sigma}(n) q^{dn^2/48},
	\end{align*}
	where
	\begin{align*}
		a_{1,0}(n)&=
		\begin{cases}
			1 & \text{if $n \equiv \pm 1 \pmod{24}$},\\
			-1 & \text{if $n \equiv \pm 7 \pmod{24}$},\\
			0 & \text{otherwise},
		\end{cases}\\
		a_{1,1}(n)&=
		\begin{cases}
			1 & \text{if $n \equiv \pm 5 \pmod{24}$},\\
			-1 & \text{if $n \equiv \pm 11 \pmod{24}$},\\
			0 & \text{otherwise}.
		\end{cases}
	\end{align*}
	Thus we obtain the assertion for $\alpha_1$.
\end{proof}

We give some examples supporting Conjecture \ref{conj:q-series} for $r \geq 2$.

\begin{example}[Zagier's lists]
	\label{example:modular Zagier's lists}
	Any $2 \times 2$ or $3 \times 3$ matrix $K$ for the Cartan-like T-data in Table \ref{table:T data r=2} and \ref{table:T-data r=3} appears in lists of Zagier~\cite[Table 2 and 3]{Zagier} as an example where 
	\begin{align*}
	\sum_{m \in (\ZZ_{\geq 0})^r} \frac{q^{\frac{1}{2} m^{\mathsf{T}}K m + B^{\mathsf{T}} m + C}}{\prod_{a=1}^r (q)_{m_a}}
	\end{align*}
	appears to be a modular function for some $B \in \QQ^r$ and $C \in \QQ$.
	We can see that all sporadic examples with $B=0$ in his lists are obtained from $(A_+,A_-)$ or $(A_-,A_+)$ in our Table \ref{table:T data r=2} and \ref{table:T-data r=3}.
\end{example}

\begin{example}[Andrew-Gordon identity]
	\label{example:modular Andrew-Gordon}
	Let $\alpha$ be the Cartan-like T-datum associated with the tadpole type $T_r$ (see Example \ref{example:tadpole}).
	It is of finite type since its T-system can be obtained from the T-system associated with the bipartite belt of type $A_{2r}$, which is periodic, by an identification of variables.
	Since $\det \mathring{A}_+ = 1$,
	we have $S_{\alpha}=0$.
	By using Theorem 6.1 in~\cite{Nakb},
	we see that the rational number $c_{\alpha}$ is given by $c_{\alpha} = 1 - 3/(2r + 3)$.
	The partition $q$-series of $\alpha$ is given by
	\begin{align*}
		\mathcal{Z}_{\alpha,0}(q) = \sum_{n \in (\ZZ_{\geq 0})^r}
		\frac{q^{N_1^2 + \dots + N_r^2}}{ (q)_{n_1}  \cdots  (q)_{n_r}},
	\end{align*}
	where $N_a = n_a + \dots + n_r$.
	Using the Andrew-Gordon identity~\cite{Andrews74}
	\begin{align*}
		\sum_{n \in (\ZZ_{\geq 0})^r}
		\frac{q^{N_1^2 + \dots + N_r^2}}{ (q)_{n_1} \cdots (q)_{n_r}}
		= \prod_{\substack{n>0 \\ n \not\equiv 0,\pm (r+1) \!\!\! \pmod{2r+3}}} \frac{1}{1-q^n},
	\end{align*}
	together with the Jacobi triple product identity, we have
	\begin{align*}
		q^{-c_{\alpha}/24} \mathcal{Z}_{\alpha,0}(q) = \frac{1}{2 \eta(q)} \sum_{n \in \ZZ} a(n) q^{n^2 / (8(2r+3)) },
	\end{align*}
	where
	\begin{align*}
		a(n)=
		\begin{cases}
		1 & \text{if $n \equiv \pm 1 \pmod{4(2r+3)}$},\\
		-1 & \text{if $n \equiv \pm (4r+5) \pmod{4(2r+3)}$},\\
		0 & \text{otherwise}.
		\end{cases}
	\end{align*}
	This implies that $q^{-c_{\alpha}/24}\mathcal{Z}_{\alpha,0}(q)$ is a modular function.
\end{example}

\begin{example}[Fermionic formulas]
	\label{example:modular fermionic formula}
	For any quantum affine algebra $U_q(\hat{\mathfrak{g}})$ and positive integer with $\level \geq 2$,
	the level $\level$ restricted T-system and Y-system for $U_q(\hat{\mathfrak{g}})$ are defined (see~\cite{KNS2011}).
	Reading the exponents in the T-system and Y-system in~\cite[Section 2]{KNS2011},
	we can obtain the Cartan-like T-datum $\alpha(U_q(\hat{\mathfrak{g}}),\level)$,
	where we replace a normalization of the parameter $u$ appropriately so that $u \in \ZZ$ and the T-datum satisfies (\ref{item:N1}), and we also discard the parameter $\Omega$ in~\cite[Section 2.4]{KNS2011} for twisted $\hat{\mathfrak{g}}$.
	Explicitly, the T-datum $\alpha(U_q(\hat{\mathfrak{g}}),\level)$ is given in Table \ref{table:quantum affine T-datum},
	where in the first line we denote by $\alpha(X_n,\level)$ the T-datum in Theorem \ref{theorem:T-datum from cartan} associated with the Cartan matrix of type $X_n$ and the integer $\level$,
	and in the remaining lines we denote by $\alpha(Y \otimes Z)$ the T-datum obtained by the tensor product construction in Example \ref{example:tensor} from the Cartan matrices of types $Y$ and $Z$.
	\begin{table}
		\begin{align*}
			\begin{array}{l|l}
				\text{type of $\hat{\mathfrak{g}}$} &
				\text{T-datum $\alpha(U_q(\hat{\mathfrak{g}}),\level)$}\\ \hline
				X_{n}^{(1)} & \alpha(X_n , \level) \\
				A_{2n-2}^{(2)}  & \alpha(A_{\level-1} \otimes C_n) \\
				A_{2n}^{(2)} & \alpha(A_{\level-1} \otimes T_n) \\
				D_{n+1}^{(2)} & \alpha(A_{\level-1} \otimes B_n) \\
				E_6^{(2)} & \alpha(A_{\level-1} \otimes F_4) \\
				D_4^{(3)} & \alpha(A_{\level-1} \otimes G_2)
			\end{array}
		\end{align*}
		\caption{T-data associated with quantum affine algebras.}
		\label{table:quantum affine T-datum}
	\end{table}
	The T-datum $\alpha(U_q(\hat{\mathfrak{g}}),\level)$ is of finite type for any $U_q(\hat{\mathfrak{g}})$ and $\level$ by the periodicity results in~\cite{Keller,IIKNS,IIKKNa,IIKKNb}.
	The partition $q$-series of $\alpha(U_q(\hat{\mathfrak{g}}),\level)$ divided by a product of the Dedekind eta coincide with the $q$-series version of the fermionic formulas defined in~\cite[Section 5]{HKOTT}.
	They conjectured that these $q$-series coincide with string functions of integrable highest modules of $\hat{\mathfrak{g}}$~\cite[Conjecutre 5.3]{HKOTT}.
	If this conjecture holds, Conjecture \ref{conj:q-series} for $\alpha(U_q(\hat{\mathfrak{g}}),\level)$ follows from the results by Kac and Peterson~\cite{KacPet}.
\end{example}

\begin{example}[$q$-series from Nil-DAHA]
	\label{example:modular nil-daha}
	Let $X_n$ be the type of a finite type Cartan matrix,
	and $p$ be an integer with $p \geq 2$.
	Consider the T-datum $\alpha(X_n \otimes A_{p-1})$,
	where the meaning of this notation is the same as that in Example \ref{example:modular fermionic formula}.
	This is of finite type by~\cite{Keller}.
	Then the partition $q$-series of $\alpha(X_n \otimes A_{p-1})$
	are special cases of the $q$-series studied by Cherednik and Feigin in the theory of Fourier transform of nilpotent double affine Hecke algebras~\cite[Corollary 1.3]{CherednikFeigin}.
	In fact, they proved that their $q$-series are modular functions~\cite[Theorem 2.3]{CherednikFeigin}.
\end{example}

\bibliographystyle{amsplain} 
\bibliography{../bib/yyyy}
\end{document}